\documentclass[12pt,reqno]{amsart}   	
\usepackage[margin=1in]{geometry}                		
\usepackage{stmaryrd}
\usepackage{mathdots}
\usepackage{bbm}
\usepackage{amsmath,amssymb,amsfonts}
\usepackage{amsthm}
\usepackage{enumerate}
\usepackage{yfonts}
\usepackage{multicol}
\usepackage{ragged2e}
\usepackage{faktor}
\usepackage{pgfplots}
\usepackage{esint}

\usepackage{pinlabel}
\usepackage{graphicx}	
\usepackage{subcaption}
\usepackage{tikz}
\usepackage{tikz-cd}

\usepackage[shortlabels]{enumitem}
\usepackage{longtable}

\newtheorem{theorem}{Theorem}
\newtheorem*{theorem*}{Theorem}

\newtheorem{proposition}{Proposition}
\newtheorem*{proposition*}{Proposition}

\newtheorem{corollary}[proposition]{Corollary}
\newtheorem*{corollary*}{Corollary}

\newtheorem{lemma}[proposition]{Lemma}
\newtheorem*{lemma*}{Lemma}

\newtheorem{remark}[proposition]{Remark}
\newtheorem*{rmk*}{Remark}

\newtheorem*{claim*}{Claim}

\theoremstyle{definition}
\newtheorem{definition}[proposition]{Definition}
\newtheorem*{definition*}{Definition}

\numberwithin{equation}{section}
\numberwithin{theorem}{section}
\numberwithin{proposition}{section}
\numberwithin{figure}{section}

\usepackage{setspace}
\usepackage{hyperref}
\hypersetup{colorlinks,allcolors=blue}
\usepackage[capitalize]{cleveref}

\usepackage{verbatim}
\usepackage{tikz}
\usepackage{tikz-cd}
\usetikzlibrary{decorations.markings}
\usetikzlibrary{angles,calc,intersections,quotes,arrows,fit}
\usetikzlibrary{shapes.geometric}

\usepackage{faktor}

\newcommand{\E}{\mathbb E}
\newcommand{\N}{\mathbb{N}}
\newcommand{\R}{\mathbb{R}}
\newcommand{\Z}{\mathbb{Z}}
\newcommand{\M}{\mathbb{M}}

\newcommand{\C}{\mathcal{C}}

\newcommand{\K}{\mathcal{K}}
\newcommand{\Q}{\mathcal{Q}}
\newcommand{\T}{\mathcal{T}}
\newcommand{\Ell}{\mathscr L}

\newcommand{\Nc}{\mathcal{N}}

\usepackage{xparse}
\NewDocumentCommand{\Hom}{mmg}{\ensuremath{\text{Hom}_{\IfNoValueTF{#3}{}{#3}}(#1,#2)}}
\NewDocumentCommand{\Tor}{mmmg}{\ensuremath{\text{Tor}_{#3}^{\IfNoValueTF{#4}{}{#4}}(#1,#2)}}
\NewDocumentCommand{\Ext}{mmmg}{\ensuremath{\text{Ext}^{#3}_{\IfNoValueTF{#4}{}{#4}}(#1,#2)}}

\usepackage[utf8]{inputenc}
\usepackage[T1]{fontenc}
\usepackage{lmodern}
\usepackage{mathrsfs}

\usepackage[backend = biber,maxbibnames=10,sorting=none]{biblatex}
\DeclareFieldFormat[misc]{title}{\mkbibquote{#1}}
\hypersetup{breaklinks=true}
\usepackage{doi}
\usepackage{xurl}
\allowdisplaybreaks
\title{Rigorous Derivation of the Wave Kinetic Equation for $\beta$-FPUT System}
\author[Boyang Wu]{Boyang Wu}
\address{Department of Mathematics, University of Michigan, Ann Arbor, MI 48109}
\email{boyangwu@umich.edu}
\pgfplotsset{compat=1.18}				
\begin{document}
\begin{abstract}
Wave kinetic theory has been suggested as a way to understand the longtime statistical behavior of the Fermi--Pasta--Ulam--Tsingou (FPUT) system, with the aim of determining the thermalization time scale. The latter has been a major problem since the model was introduced in the 1950s. In this thesis we establish the wave kinetic equation for a reduced evolution equation obtained from the $\beta$--FPUT system by removing the non-resonant terms. We work in the kinetic limit $N\to \infty$ and $\beta\to 0$ under the scaling laws $\beta=N^{-\gamma}$ with $0<\gamma<1$. The result holds up to the sub-kinetic time scale $T=N^{-\epsilon}\min\bigl(N,N^{5\gamma/4}\bigr)=N^{-\epsilon}T_{\mathrm{kin}}^{5/8}$ for $\epsilon\ll 1$, where $T_{\mathrm{kin}}$ represents the kinetic (thermalization) timescale. The novelties of this work include the treatment of non-polynomial dispersion relations, and the introduction of a robust phase renormalization argument to cancel dangerous divergent interactions. 
\end{abstract}
\maketitle
\tableofcontents

\section{Introduction}
\label{section-kinetictheory}
\subsection{The $\beta$-FPUT model} \label{sec-wkt}
In this thesis we derive the \emph{wave kinetic equation} (WKE) for $\beta$-Fermi-Pasta-Ulam-Tsingou ($\beta$-FPUT) system and at the kinetic time scale, for the partial range of scaling laws between the large box and weak nonlinearity limits. This is the first step aiming at providing rigorous mathematical foundation for the Wave Turbulence (WT) theory on FPUT system. The Fermi-Pasta-Ulam-Tsingou system was first analyzed in 1955 \cite{EJS} to study the thermal equipartition of crystals. The $\beta$-FPUT model has $N$ identical masses connected by nonlinear springs with the elastic force $F=-\kappa\Delta q+\beta\Delta q^3$, where $\kappa,\beta\neq 0$ are spring elastic constants. And the equation of motion for a $\beta$-FPUT chain can be expressed as:
\begin{align}
m\Ddot{q}_j&=\kappa\left[(q_{j+1}-q_j)-(q_j-q_{j-1})\right]+\beta\left[(q_{j+1}-q_j)^3-(q_j-q_{j-1})^3\right], \label{eq-1.1}
\end{align}
where $j=0,1,...,N-1$ and $q_j(t)$ represents the displacement from the equilibrium position of the particle $j$ with periodic boundary conditions: $q_0=q_N$ (see e.g. \cite{EJS} and \cite[equation (6)]{bustamante2019exact}). Then we can get the Hamiltonian:
\begin{align}
H=\sum_{j=0}^{N-1}\left(\frac{p_j^2}{2m}+\frac{\kappa}{2}(q_{j+1}-q_j)^2+\frac{\beta}{4}(q_{j+1}-q_j)^4\right),
\end{align}
where $p_j(t)$ represents the momentum of particle $j$. 
\subsubsection{Discrete Fourier transform}
Denote $\Z_N:=N^{-1}\Z$. Next, we perform a Discrete Fourier transform and its inverse to the $\beta$-FPUT system as follows:
\begin{align}
Q_k&=\frac{1}{N}\sum_{j=0}^{N-1}q_je^{-i2\pi kj},\mbox{ }q_j=\sum_{k=0}^{1-1/N}Q_ke^{i2\pi jk};\label{eq-1.3}\\
P_k&=\frac{1}{N}\sum_{j=0}^{N-1}p_je^{-i2\pi kj},\mbox{ }p_j=\sum_{k=0}^{1-1/N}P_ke^{i2\pi jk},\label{eq-1.4}
\end{align}
where $Q_k$, $P_k$ are Fourier transforms of the displacement and momentum respectively for $k \in \Z_N\cap[0,1)$. From the relation:
\[\dot{q_j}=\frac{\partial{H}}{\partial{p_j}}=\frac{p_j}{m},\] we should get:
\begin{align}
P_k=m\dot{Q}_k. \label{eq-1.5}
\end{align}
And we should also note that:
\begin{align}
Q_{k}=Q_{1-k}^*,\mbox{ }P_{k}=P_{1-k}^*, \mbox{ for }k\in \Z_N\cap[0,1).
\end{align}
Plugging (\ref{eq-1.3}) and (\ref{eq-1.4}) into the equation of motion (\ref{eq-1.1}), for $j=0,1,...,N-1$, we get:
\begin{align}
m\left(\Ddot{Q}_{k_1}+\omega_{k_1}^2{Q}_{k_1}\right)=\beta\sum_{k_2+k_3+k_4=k_1}\hat{T}_{-1,2,3,4}{Q}_{k_2}{Q}_{k_3}{Q}_{k_4},\label{eq-1.7}
\end{align}
where the sums $\sum_{k_2+k_3+k_4=k_1}$ are taking over $\{k_2,k_3,k_4:k_2+k_3+k_4=k_1 \mbox{ mod } 1\}$, and
\begin{align*}
&\omega_{k}=2\sqrt{\frac{\kappa}{m}}\left|\sin(\pi k)\right|,\\
&\hat{T}_{-1,2,3,4}=16\sin(\pi (k_2+k_3+k_4))\sin(\pi k_2)\sin(\pi k_3)\sin(\pi k_4).
\end{align*}
\subsubsection{Normal modes}
We then introduce the normal modes $a_k(t)$ as:
\begin{align}
a_k=\frac{\sqrt{2}}{2}\left[(m\omega_k)^{\frac{1}{2}}Q_k+i(m\omega_k)^{-\frac{1}{2}}P_k\right],  \mbox{ for }k\in \Z_N\cap(0,1)
\end{align}
Then the evolution equation for $a_{k_1}(t)$ becomes:
\begin{align}
i\dot{a}_{k_1}=\omega_{k_1}{a}_{k_1}+\frac{\beta}{3}\sum_{k_2+k_3+k_4=k_1}\tilde{T}_{-1,2,3,4}(a_{k_2}+a_{1-{k_2}}^*)(a_{k_3}+a_{1-{k_3}}^*)(a_{k_4}+a_{1-{k_4}}^*),
\end{align}
where 
\begin{align} 
    \tilde{T}_{-1,2,3,4}&=\frac{-3\hat{T}_{-1,2,3,4}}{4m^2\sqrt{\omega_{k_1}\omega_{k_2}\omega_{k_3}\omega_{k_4}}}= -\frac{3}{4\kappa^2}\iota(k_2+k_3+k_4)\iota(k_2)\iota(k_3)\iota(k_4) \prod_{i=1}^4 \sqrt{\omega_{k_i}}.
\end{align}
where we define $\iota(x):=\text{sgn} \sin (\pi x)$ a $2$-periodic function on $\mathbb{R}$. Rewrite $k_2'=1-{k_2},k_3'=1-{k_3}$, and $k_4'=1-{k_4}$, we get:
\begin{align}
i\dot{a}_{k_1}=&\omega_{k_1}{a}_{k_1}+\frac{\beta}{3}\sum_{k_2,k_3,k_4}\bigg[\tilde{T}_{-1,2,3,4}a_{k_2}a_{k_3}a_{k_4}\delta(k_1-k_2-k_3-k_4)\notag\\
&+3\tilde{T}_{-1,2,-3,4}a_{k_2}a_{k_3}^*a_{k_4}\delta(k_1-k_2+k_3-k_4)+3\tilde{T}_{1,2,3,-4}a_{k_2}^*a_{k_3}^*a_{k_4}\delta(k_1+k_2+k_3-k_4)\notag\\
&+\tilde{T}_{1,2,3,4}a_{k_2}^*a_{k_3}^*a_{k_4}^*\delta(k_1+k_2+k_3+k_4)\bigg],\label{eq-1.25}
\end{align}
where 
\begin{align}
    \tilde{T}_{-1,2,3,4}&=-\frac{3}{4\kappa^2}\iota(k_2+k_3+k_4)\iota(k_2)\iota(k_3)\iota(k_4) \prod_{i=1}^4 \sqrt{\omega_{k_i}}\\
    \tilde{T}_{-1,2,-3,4}&=\frac{3}{4\kappa^2} \iota(k_2-k_3+k_4)\iota(k_2)\iota(k_3)\iota(k_4) \prod_{i=1}^4 \sqrt{\omega_{k_i}},\\
    \tilde{T}_{1,2,3,-4}&=\frac{3}{4\kappa^2}\iota(k_2+k_3-k_4)\iota(k_2)\iota(k_3)\iota(k_4) \prod_{i=1}^4 \sqrt{\omega_{k_i}},\\
    \tilde{T}_{1,2,3,4}&=-\frac{3}{4\kappa^2}\iota(k_2+k_3+k_4)\iota(k_2)\iota(k_3)\iota(k_4) \prod_{i=1}^4 \sqrt{\omega_{k_i}}.
\end{align}
\subsubsection{The reduced evolution equation} \label{change-of-variable}
Rewriting $b(t,k)=\sqrt{N}a(t,k)$ for $k \in \Z_N\cap(0,1)$, and removing the non-resonant part in (\ref{eq-1.25}), we arrive at the following reduced evolution equation for $b_k$:
\begin{align}
i\dot{b_k}(t) = \omega_kb_k(t)+\frac{\beta}{N}\sum_{k_1-k_2+k_3=k}T_{k,1,2,3}b_{k_1}(t)b^*_{k_2}(t)b_{k_3}(t), \tag{R-FPUT} \label{FPU}
\end{align}
where $k, k_i \in \Z_N\cap(0,1)$, and we define:
\begin{align*}
T_{k,1,2,3}:=\frac{3}{4\kappa^2} \iota(k_1-k_2+k_3)\iota(k_1)\iota(k_2) \iota(k_3) \sqrt{\omega_k\omega_{k_1}\omega_{k_2}\omega_{k_3}}.
\end{align*}
The Hamiltonian for \eqref{FPU} in normal variables $b_k$ becomes:
\begin{align}
H = \sum_{k\in \Z_N\cap(0,1)}\omega_k|b_k|^2+\frac{\beta}{2N}\sum_{k_1-k_2+k_3-k_4=0}T_{1,2,3,4}b_1^*b_2b_3^*b_4.\label{Ham}
\end{align}
We will prove the WKE for \eqref{FPU} instead of the full original evolution equation \eqref{eq-1.25} in this thesis and leave the complete derivation for future work.
\subsubsection{The initial data}\label{sec.initialdata}
To apply the Wave Turbulence formalism and formally derive the kinetic equation for \eqref{FPU}, we need to define well-prepared initial conditions for $b_k$:
\begin{align}
b_k(0) = \sqrt{n_{\text {in}}(k)}\eta_k(\varrho), \tag{DAT} \label{DAT}
\end{align}
where $n_{\text {in}}\geq 0$ represents the deterministic amplitude of an $O(1)$ smooth function (or a function with sufficient regularity) defined on standard torus of size $1$, namely $\mathbb{T}=[0,1]$ with periodic boundary conditions. $\eta_k(\varrho)$'s are i.i.d. mean-zero complex random variables with unit variance, which are assumed to be either standard complex Gaussian or uniform distributed on the unit circle. $\varrho$ labels an individual random sample on which the variables $\eta_k$ are evaluated. For the reduced evolution equation \eqref{FPU}, if we make a phase transition $b_k \rightarrow b_ke^{i\theta}$, the initial distribution of $b_k$ does not change, so we can get:
\begin{align}
    \langle b_kb_{k'}\rangle = \langle b_ke^{i\theta}b_{k'}e^{i\theta}\rangle = e^{i2\theta}\langle b_kb_{k'}\rangle
\end{align}
for any $\theta \in \mathbb{R}$, which means $\langle b_kb_{k'}\rangle \equiv 0$ for any $k,k'\in \Z_N\cap (0,1)$ and we only need to consider $\langle |b_k|^2\rangle$. This is different from the original equation that we need to consider both  $\langle |a_k|^2\rangle$ and $\langle a_ka_{1-k}\rangle$ for $k \in \Z_N\cap (0,1)$.

Our goal is to rigorously derive the wave kinetic equation (WKE) for the dynamic of the wave spectrum $n(t,k)=\mathbb{E}|b_k(t)|^2$, such that $n(0,k)=n_{\text {in}}(k)$.
\subsubsection{The wave kinetic equation}
The wave kinetic equation is given by:
\begin{equation}\label{WKE} \tag{WKE}
\begin{cases}
\partial_t n(t,k) = \K(n(t))\left(k\right), \\
n(0,k) = n_{\mathrm{in}}(k),
\end{cases}
\end{equation}
where $n_{\text {in }}$ is as in Section \ref{sec.initialdata}, and the nonlinearity collision operator $\mathcal{K}$ is given by
\begin{equation}\label{KIN}\tag{KIN}
\K(\phi)(\xi) = \int_{\begin{subarray}{b}(\xi_1, \xi_2, \xi_3 ) \in \mathbb{T}^{3} \\ \xi_1 - \xi_2 + \xi_3 = \xi\end{subarray}} \phi_\xi\phi_{\xi_1} \phi_{\xi_2} \phi_{\xi_3} \left(\frac{1}{\phi_\xi} - \frac{1}{\phi_{\xi_1}} + \frac{1}{\phi_{\xi_2}} - \frac{1}{\phi_{\xi_3}} \right) \boldsymbol{\delta}(\omega_1 - \omega_2 + \omega_3 - \omega)d\xi_1 d\xi_2 d\xi_3,
\end{equation}
where we denote, for $i = 1,2,3$,
\begin{align*}
&\phi = \phi(\xi,t), \hspace{.5cm} \phi_i = \phi(\xi_i,t), \\
&\omega = 2\sin\left(\pi \xi\right), \hspace{.5cm} \omega_i = 2\sin\left(\pi \xi_i\right).
\end{align*}
Here and below $\boldsymbol{\delta}$ denotes the Dirac delta. Given any Schwartz initial data $n_{\text {in}}(k)$, the equation (WKE) has a unique local solution $n=n(k,t)$ on some short time interval depending on $n_{\text {in}}$.
\subsection{Statement of the main result} The main result is stated as follows.
\begin{theorem} \label{main}
Fix $\gamma \in(0,1)$. Fix a smooth function $n_{\mathrm {in }}\in C^{\infty}(\mathbb{T}\rightarrow [0,\infty))$, $A \geq 40$, and $\epsilon \ll 1$. Consider the equation \eqref{FPU} with random initial data \eqref{DAT}, and assume $\beta=N^{-\gamma}$ so that $T_{\text{kin}}=\frac{N^{2 \gamma}}{4\pi}$. Fix $T=N^{-\epsilon}\min(N,N^{\frac{5}{4}\gamma})$.
Then, with probability $\geq 1 - N^{-A}$, for ${N^{0+} \leq t \leq T}$,
\[\mathbb E\left|b(k,t)\right|^2 = n_{\mathrm{in}}(k)  + \frac{t}{T_{\mathrm{kin}}} \mathcal K(n_{\mathrm{in}})\left(k\right) + o_{\ell_k^\infty}\left(\frac{t}{T_{\mathrm{kin}}}\right).\]
\end{theorem}
\subsection{Background literature}Tracing back to 1923, Fermi targeted to prove the quasi-ergodicity of general Hamiltonian systems in \cite{Fermi1,Fermi2}. Here the quasi-ergodicity refers to the idea that a system trajectory will be eventually dense in the phase space. Around 1954, Fermi, Pasta, Ulam, and Tsingou (FPUT) started to use a numerical approach to verify that a system with arbitrarily small nonlinearity would reach thermalization.
\subsubsection{FPUT recurrence}
Fermi's expectation was to see the redistribution of the energy among the Fourier modes, due to the presence of the cubic and quartic nonlinearity in the Hamiltonian, at the thermalization. Surprisingly, the unexpected outcomes largely deviated from their original conjecture: the system presented a quasi-periodic behavior with energy almost returning to the initial condition after some energy exchange. The first quasi-period was at about 29,000 computation cycles, and the energy of the first Fourier mode returns to more than $97\%$ of its initial energy \cite{EJS}. This recurrent phenomenon, which was then called the “FPUT recurrence”, indicated the failure of arguments that any nonlinearity is sufficient to guarantee ergodicity. A lot of attempts \cite{schneider2000counter,zabusky1965interaction,israiljev1965statistical, zaslavskiui1972stochastic,gallavotti2007fermi,rink2006proof,henrici2008results,FPU,ponno2011two,ponno2005fermi,bocchieri1970anharmonic,carati1999specific,carati2004definition,berchialla2004localization,bambusi2006metastability,carati2007averaging,bambusi2015some} has been made to explore the FPUT recurrence and thermalization.
\subsubsection{Wave Turbulence theory}
The WT theory was first introduced in 1929 by Peierls to study the phonons kinetics in anharmonic crystals \cite{peierls1929}. It is not until the 1960s, that the WT theory, focusing on the wave kinetic equation and spectrum, was reinvestigated in the plasma physics \cite{vedenov1967theory} and water waves \cite{WaterWaves62,WaterWaves63}. The main theoretical consequences during that period were the Kolmogorov-Zakharov spectra solutions to the kinetic equation \cite{zakharov2012kolmogorov}, discovered by Zakharov in 1965, and the further inspired theoretical developments were summarized by Zakharov \emph{et al.} in \cite{zakharov1965weak}. 

The wave kinetic equation occupies in wave turbulence the same position that Boltzmann’s equation holds in dilute--gas theory: it furnishes an effective, statistical description of the second--moment dynamics of weakly nonlinear dispersive waves \cite{TextbookWT}.  Formally the WKE is obtained in a \emph{kinetic limit} where the domain size grows unboundedly $N\to\infty$, while the interaction strength tends to zero, $\beta\to0$.  A precise \emph{scaling law}---typically $\beta=N^{-\gamma}$ with $\gamma>0$---specifies how the two parameters are tied together.

In the recent twenty years, the rigorous justification of the WT theory via the kinetic equation has been done in dimensions \(d\ge2\). The first rigorous step towards wave‐kinetic theory for the cubic nonlinear Schr\"odinger (NLS) equation was the derivation of the \emph{Continuous Resonant} (CR) equation—describing the large–box limit of the deterministic dynamics—via Birkhoff normal‐form reductions by Faou \emph{et al.} \cite{2DNLS} and independently by Buckmaster \emph{et al.} \cite{BGHSCR,BGHSdyn} in 2016.  Buckmaster \emph{et al.} then incorporated random initial data and employed Feynman diagram expansions to validate the WKE up to times \(t< T_{\mathrm{kin}}^{1/2}\) \cite{BGHSonset}. The rigorous convergence of Feynman diagrams for wave systems was first proved by Spohn in 1977 \cite{Spohn08} and later extended to global‐in‐time results by Erdős–Yau \cite{ErdosYau} and Erdős–Salmhofer–Yau \cite{ErdosSalmhoferYau}. Shortly thereafter, Deng and Hani \cite{2019} and Collot and Germain \cite{CG1,CG2} extended this result—still in \(d\ge2\)—to times \(t\lesssim T_{\mathrm{kin}}^{1-\epsilon}\) under scaling laws \(0<\gamma\le1\) and \(\gamma=1\), respectively, thus achieving the analog of Lanford’s short‐time theorem for the NLS WKE.  In follow‐up works Deng and Hani pushed their analysis to the full kinetic time \(T_{\mathrm{kin}}\) \cite{WKE,WKE2023} and ultimately to arbitrarily long times \cite{WKElong}.  Remarkably, the combinatorial and renormalisation methods they introduced were then adapted to give a parallel long‐time derivation of the Boltzmann equation for particle systems in \cite{Boltzmannlong}. In 2024, Vassilev \cite{ODW} has derived the wave kinetic equations for the Majda-McLaughlin-Tabak (MMT) model at $d=1$ of size $L$ with the dispersion relation of $\omega(k) = |k|^\sigma$, and $0 < \sigma \leq 2$ and $\sigma \neq 1$ up to timescales $T \sim L^{-\epsilon} T_{\mathrm{kin}}^{5/8}$.
\subsubsection{The Wave Turbulence approach to the FPUT model}
The early applications of WT theory to the $\alpha$-FPUT model, done by Lvov \emph{et al.} in 2002 \cite{biello2002stages,kramer2002application}, analyzed the time scales of the energy spectrum approaching the equilibrium. In 2015, it was found that the equipartition of energy of $\alpha$-FPUT model for $N=16,32,64$ resulted from the six-wave interactions by Onorato \emph{et al.} \cite{FPU}, which can also be applied to $\beta$-FPUT model \cite{pistone2018universal}. Bustamante \emph{et al.} continued to analyze the exact resonances for the FPUT system in 2018 \cite{bustamante2019exact}. In the same year, Lvov and Onorato find a different time scale $O(\beta^{-1})$ at equilibrium for the $\beta$-FPUT model using the Chirikov overlap criterion, when the nonlinearity $\beta$ increases to the level of $0.01$ \cite{lvov2018double}. They further tested the $\beta$-FPUT model under a low-temperature regime \cite{dematteis2020coexistence}. Their recent work in 2022, continues to check numerically the anomalous scaling of the energy conductivity for the $\beta$-FPUT model \cite{de2022anomalous}. They used a previous result by Lukkarinen and Spohn \cite{AET}, in which the authors derived the solutions to the collision integral of the $\beta$-FPUT model and reduce it to a single integral in the limiting sense. More studies have been performed on the FPUT model: Pezzi \emph{et al.} \cite{pezzi2025multi} identified the dominance of four-wave processes in the diatomic $\alpha$-FPUT chain and Ganapa \cite{ganapa2023quasiperiodicity} showed divergence of the canonical transformation under stronger nonlinearity or larger system size in $\alpha$-FPUT model.
\subsection{Overview of the proof}
Here, we outline the main strategy and methodology used to establish the main theorem. The proof is based on a diagrammatic expansion of the solution $c_k(t)$ as follows:
\begin{equation*}
    c_k = c_k^{(0)} + c_k^{(1)} + c_k^{(2)} + \ldots \ldots+ c_k^{(n)} + R^{(n+1)}_k=c^{\leq n}_k +R^{(n+1)}_k,
    \end{equation*}
where $n$ is chosen to be a sufficiently large threshold, $c^{(j)}$'s represent the $j$-th iterates, and $R^{(n+1)}$ denotes the remainder term.

\begin{enumerate}
    \item It has long been understood (see \cite{TextbookWT}) that the interactions corresponding to $k_1=k$ or $k_3=k$ in \eqref{FPU} lead to a phase divergence in the iterates, due to the largeness of the term $2\sum_{k_1}T_{k,1,1,k} |c_{k_1}(s)|^2$. This requires a phase renormalization argument to delete those interactions as much as possible. While there have been several formal attempts at this renormalization \cite{TextbookWT,SW,chibbaro20184}, such formal attempts cannot be easily executed at a rigorous level since they introduce random terms to the phase factor, which spoils the pairing structure that is needed in the remainder of the proof. To remedy this, we introduce a deterministic phase renormalization in which the renormalized phase function satisfies the following ODE system:
    
    \begin{equation}\label{eq-freqshift} 
    \partial_s\widetilde \omega_k(s)=\frac{2\beta T}{N}\sum_{k_1}T_{k,1,1,k} \mathbb E|c^{\leq n}_{k_1}(s)|^2.
    \end{equation}

    Here, $c^{\leq n}_{k_1}$ represents the sum of ternary tree expansion terms up to order $n$, which depend themselves on $\widetilde \omega_k$. The shift itself is \emph{deterministic} and defined implicitly as the solution of the above ODE. This allows for the divergent factor $2\sum_{k_1}T_{k,1,1,k} |c_{k_1}(s)|^2$ to be canceled as much as possible. 

    After this thesis work was finalized, the recent preprint of Deng-Ionescu-Pusateri \cite{deng2025wave} appeared on arXiv, in which a similar phase renormalization technique is also employed in the context of the long-time existence question for the two-dimensional gravity water waves system on large tori. 

    \item We define the \textit{ternary trees} in Section \ref{sec-tt} to perform the Feynman diagram expansion so that $c^{(j)}$ is tracked by the sum of all ternary trees of order $j$. We keep track of the nonlinear wave interactions by the \textit{decorations} of the wavenumbers in the ternary trees. Unlike \cite{WKE,WKE2023,ODW}, since we do not remove the entire nonlinear frequency shifts, we need to deal with the \textit{degenerate nodes} (branching nodes with decorations $\{k,k_2\}=\{k_1,k_3\}$) which are still left in the ternary trees and construct an enhanced tree structure to fix specific children having the same decorations. To further analyze the correlations of the iterates, we define the \textit{enhanced couples} of enhanced ternary trees without complete pairing of the leaves originated from the children of the degenerate node with the same decoration in Section \ref{sec-couple1}.
    \item To reduce the Feynman diagram analysis to counting problems on couples, we express the couples into the form of \textit{molecules} in Section \ref{sec-mol}. We adapt the cancellation process in Section \ref{irregular} and the algorithm in Section \ref{subsec-ass} from \cite{ODW} to incorporate the molecules involved with degenerate atoms. Repeating the definitions in \cite{WKE,ODW}, we introduce the \textit{irregular chains} which leads to large bad counting estimates. We could remove the irregular chains in \textit{congruent couples} via \textit{splicing} operations due to cancellations after we sum up those congruent couples. After the splicing operation, we need to remove the degenerate atoms in the molecule via a preprocessing step which will not result in any loss before the algorithm starts. With verification of the same structure of molecules in Proposition \ref{prop-op-bound}:
    \begin{equation} \label{eq-rough-bound}
    \text{\# two-vector countings} \leq 3 \times (\text{\# three-vector countings} -1),
    \end{equation}
    as in \cite{ODW}, we could use the algorithm to prove for bounds on couples. The algorithm we adapt from \cite{ODW} can only reach  $T < \beta^{-\frac{5}{4}}$ instead of the kinetic time scale $T_{\mathrm{kin}}$ because the five-vector counting used in \cite{WKE,WKE2023} are not possible for $d=1$ as shown in \cite[Proposition 3.7]{ODW}.
    \item Using the integral estimate Proposition \ref{integral_est_1} in Section \ref{sec-int} and the counting estimates we prove in Section \ref{section-counting}, we are able to prove the Proposition \ref{bound-couple}. The main technical difficulties in the proof arise when we deal with unconserved nonlinear frequency shifts in the time integrals and the sinusoidal dispersion relation in the counting estimates. The more general version of the integral estimate Proposition \ref{integral_est} can be adapted and applied to more general models with unconserved nonlinear frequency shifts.
    \item Following the same approach as in \cite{WKE,WKE2023,ODW}, we bound the remainder terms using a contraction mapping argument in sections \ref{section-remainder} \& \ref{section-main} by defining the \textit{flower trees} and \textit{flower couples}, and applying the same preprocessing and counting process as in the proof for bounds on couples to bound and invert the linear operator $1 - \Ell$, for $\mathscr L$ defined in a suitable space. We express $(1 - \mathscr L)^{-1}$ as:
    \[
    (1 - \mathscr{L})^{-1} 
    \;=\; (1 - \mathscr L^{n_h})^{-1}\,\bigl(1 + \mathscr{L} + \cdots + \mathscr{L}^{n_h-1}\bigr),
    \]
    with some large order $n_h$ to bound the powers of $\Ell$ by Proposition \ref{prop-remainder} and $(1 - \mathscr L^{n_h})^{-1}$ via Neumann series.
    \item We prove the main theorem by isolating the negligible higher iterates and dealing with iterates $0,1,$ and $2$. We need to prove the exact resonances to be lower order terms and quasi-resonances to be the main contribution terms. After separating the lower order terms in the time integrals, we use the fact for smooth functions $f$,
    \begin{align*}
    \left|t\int\left| \frac{\sin(t x/2)}{t x/2}\right|^2 f(x) d x - 2\pi f(0)\right| \lesssim C(f) t^{-\frac{1}{2}},
    \end{align*}
    where $C(f)$ depends on the $L^\infty$ norms of $f$ and $f'$ to derive the wave kinetic equation for the $\beta$-FPUT model.
\end{enumerate}
\subsection{Future horizons} We rigorously derived the WKE for the reduced evolution equation \eqref{FPU} instead of the full evolution equation \eqref{eq-1.25}. For the future directions:
\begin{enumerate}
    \item By relying on a standard normal forms argument, we hope to deal with the full FPUT system (including the non-resonant interactions). 
    \item The general framework developed in this thesis, namely the nonlinear frequency shift, is robust, and allows for a similar derivation of the WKE to other models with various dispersion  relations and unconserved frequency shifts, e.g. discrete nonlinear Klein–Gordon model and other anharmonic lattices.
\end{enumerate}
\noindent \textit{Acknowledgements.} The author would like to thank his advisor Zaher Hani and academic sibling Katja Dimitrova Vassilev for the numerous insightful conversations and discussions through the preparation of this thesis. This research was supported by the Simons Collaboration Grant on Wave Turbulence and NSF grant DMS-1936640. 

\section{Preparations and Main Tools}
\label{section-preparations}
\label{sec-prep}
The purpose of this section is to introduce the Feynman diagrams infrastructure and the choice of our nonlinear frequency shift. 
\subsection{Ternary tree expansion} \label{sec-tt}
\begin{definition} (Trees) \label{def:Trees}
\begin{enumerate}[label=(\roman*)]
    \item A \textit{ternary tree} $\T$ is a rooted tree where each branching node has precisely three children.  Let $\mathcal{T}$ be a ternary tree, $\mathcal{N} = \mathcal{T} \backslash \mathcal{L}$ denote the set of branching nodes, and $n$ be their number, $\mathcal{L}$ denote the set of leaves, and $l = 2n+1$ be their number. The \emph{scale} of a ternary tree $\mathcal{T}$ is defined as $|\mathcal{T}|=n$, i.e. the number of branching nodes.
    \item Denote the \textit{root} node of a ternary tree $\mathcal{T}$ as $\mathfrak{r}$ and the children subtrees of $\mathcal{T}$ as $\mathcal{T}_1,\mathcal{T}_2$ and $\mathcal{T}_3$ from left to right. Let $\mathcal{T}=\bullet$ denote a \emph{trivial} tree with scale 0, that is, with only the root node $\mathfrak{r}$.
    \item We assign to each ternary tree $\mathcal{T}$ a sign $+\text{ or }-$. For any node $\mathfrak{n}\in \mathcal{N}$, let its children be $\mathfrak{n}_1,\mathfrak{n}_2,\mathfrak{n}_3$ from left to right. We assign each node the sign $\zeta_{\mathfrak{n}}\in\{\pm1\}$ such that: $\zeta_{\mathfrak{r}}=\text{sign}(\mathcal{T})$, and $\zeta_{\mathfrak{n}}=\zeta_{\mathfrak{n}_1}=-\zeta_{\mathfrak{n}_2}=\zeta_{\mathfrak{n}_3}$.
\end{enumerate}
\end{definition}
\begin{figure}[ht]
\centering
\begin{tikzpicture}[scale=1.2,
  level distance=1.5cm,
  level 1/.style={sibling distance=4cm},
  level 2/.style={sibling distance=1.5cm},
  level 3/.style={sibling distance=0.8cm}]
  \tikzset{
    solid node/.style={circle,draw,fill=black,inner sep=1.5pt},
    hollow node/.style={circle,draw,inner sep=1.5pt}
  }

  \node[solid node,label=above:{\footnotesize $(\mathfrak r,+)$}] (root) {}
    child {
      node[solid node,label=left:{\footnotesize $(\mathfrak n_1,+)$}] {}
      child { node[solid node,label=below:{\footnotesize $(\mathfrak m_1,+)$}] {} }
      child { node[solid node,label=below:{\footnotesize $(\mathfrak m_2,-)$}] {} }
      child { node[solid node,label=below:{\footnotesize $(\mathfrak m_3,+)$}] {} }
    }
    child {
      node[solid node,label=right:{\footnotesize $(\mathfrak n_2,-)$}] {}
      child { node[solid node,label=below:{\footnotesize $(\mathfrak l_1,-)$}] {} }
      child { node[solid node,label=below:{\footnotesize $(\mathfrak l_2,+)$}] {} }
      child { node[solid node,label=below:{\footnotesize $(\mathfrak l_3,-)$}] {} }
    }
    child {
      node[solid node,label=right:{\footnotesize $(\mathfrak n_3,+)$}] {}
      child { node[solid node,label=below:{\footnotesize $(\mathfrak p_1,+)$}] {} }
      child { node[solid node,label=below:{\footnotesize $(\mathfrak p_2,-)$}] {} }
      child { node[solid node,label=below:{\footnotesize $(\mathfrak p_3,+)$}] {} }
    };
\end{tikzpicture}
\caption{A ternary tree of scale 4 with root $(\mathfrak r,+)$ and three branching nodes $\mathfrak n_i$, each with three children $\mathfrak m_j$, $\mathfrak l_j$, or $\mathfrak p_j$ labeled by their signs.}
\label{fig:ternary-tree}
\end{figure}
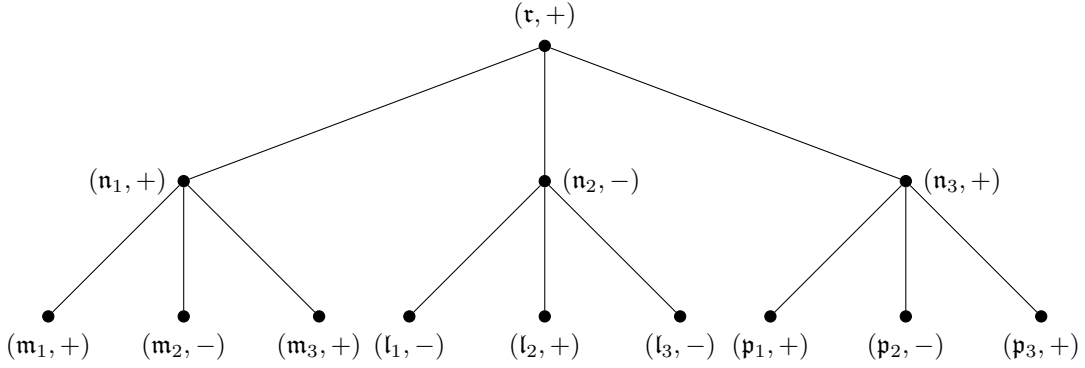

For each ternary tree $\mathcal{T}$, we define $b_k^{\mathcal{T}}(t)$'s by the following equation inductively:
\begin{align}
b_k^{\bullet}(t)&=\sqrt{n_{\text{in}}(k)}\eta_{k}(\varrho)\\
i\dot{b_k^{\mathcal{T}}}(t) &= \frac{\beta}{N}\sum_{k_1-k_2+k_3=k}T_{k,1,2,3}b_{k_1}^{\mathcal{T}_1}(t){b_{k_2}^{\mathcal{T}_2}}^*(t)b_{k_3}^{\mathcal{T}_3}(t)e^{i\Omega_{k} t},
\end{align}
where we define: 
\begin{align*}
\Omega_{k,1,2,3}&=\Omega_{k}:=\omega_{k}-\omega_{k_1}+\omega_{k_2}-\omega_{k_3},
\end{align*}
And we use ${\mathcal{T}_2}^*$ to denote the transformation: $\text{sign}(\mathcal{T}_2^*)=-\text{sign}(\mathcal{T}_2)$, i.e. ${\mathcal{T}_2^*}$ to have opposite sign to ${\mathcal{T}_2}$. Then by induction, we can show that $b_k^{\mathcal{T}}(t)$ has the expression:
\begin{align}
b_k^{\mathcal{T}}(t)=\left(\frac{\beta}{N}\right)^{|\mathcal{T}|}\sum_{\mathcal{D}}\prod_{\mathfrak{n}\in \mathcal{N}}(-i\zeta_{\mathfrak{n}}T_{\mathfrak{n},\mathfrak{l},\mathfrak{n}_2,\mathfrak{n}_3})\mathcal{A}_{\mathcal{T}}(t,\Omega[\mathcal{N}])\prod_{\mathfrak{l}\in \mathcal{L}}\sqrt{n_{\text{in}}(k_{\mathfrak{l}})}\eta_{k_{\mathfrak{l}}}^{\zeta_{\mathfrak{l}}}(\varrho),
\end{align}
where the coefficients $\mathcal{A}_{\mathcal{T}}(t,\Omega[\mathcal{N}])$'s are defined inductively by:
\begin{align*}
\mathcal{A}_{\bullet}(t,\Omega[\mathcal{N}])&=1,\\
\mathcal{A}_{\mathcal{T}}(t,\Omega[\mathcal{N}])&=\int_0^te^{\zeta i\Omega_{\mathfrak{r}}t'}\prod_{j=1}^3\mathcal{A}_{\mathcal{T}_j}(t',\Omega[\mathcal{N}_j])dt'.
\end{align*}
Removing the linear term by substituting $b_k'(t):=e^{i\omega_kTt}b_k(Tt)$, where $0 \le t \le 1$, we can get the evolution equation for $b_k'(t)$:
\begin{align}
i\dot{b_k'}(t) = \frac{\beta T}{N}\sum_{k_1-k_2+k_3=k}T_{k,1,2,3}b_{k_1}'(t)b'^*_{k_2}(t)b_{k_3}'(t)e^{i\Omega_k Tt}
\end{align}
We can also express $b_k'(t)$ using the ternary tree expansion by:
\begin{align}
b_k'(t) = \sum_{|\mathcal{T}^+|\le n}b_k^{\mathcal{T}^+}(t)+\mathcal{R}^{(n+1)}_{k},\label{eq-1}
\end{align}
where the sum is taken over all trees with $+$ sign and scales less than or equal to $n$, and $\mathcal{R}^{(n+1)}_{k}$ denotes the remainder term.
\subsection{Nonlinear frequency shift} \label{sec-fs}
Suppose $c_k(t)$ is obtained from $b_k'(t)$ by removing an undetermined nonlinear frequency  $\widetilde{\omega}_k(t)$  satisfying  $\widetilde{\omega}_k(0)=0$:
\begin{align}\label{phase change}
c_k(t) = b_k'(t)e^{i\widetilde{\omega}_k(t)},
\end{align}
We define the $Z$ norm for $\textbf{c}(t) = (c_k(t))_{k \in \Z_N\cap (0,1)}$ as:
\begin{equation}
\|\textbf{c}\|_Z^2 = \sup_{0 \leq t \leq 1} \frac{1}{N} \sum_{k \in \Z_N\cap (0,1)} |c_k(t)|^2.
\end{equation}
Then we can derive the evolution equation for $c_k(t)$ by differentiating (1.8) and separate the degenerate sums, i.e. $k \in \{k_1,k_3\}$, as follows:
\begin{align}
i\dot{c_k}(t) =& \left[i\dot{b_k'}(t)-\dot{\widetilde{\omega}}_k(t)b_k'(t)\right]e^{i\widetilde{\omega}_k(t)}\notag\\
=& i\dot{b_k'}(t)e^{i\widetilde{\omega}_k(t)}-\dot{\widetilde{\omega}}_k(t)c_k(t)\notag\\
=&\frac{\beta T}{N}\left[\sum_{\substack{k_1-k_2+k_3=k\\k_1,k_3 \neq k_2}}T_{k,1,2,3}b_{k_1}'(t)b'^*_{k_2}(t)b_{k_3}'(t)e^{i\left(\Omega_{k} Tt+\widetilde{\omega}_k(t)\right)}\right.\notag\\
&+\left.\left(\sum_{k_1}T_{k,1,1,k}|b_{k_1}'(t)|^2+\sum_{k_3}T_{k,k,3,3}|b_{k_3}'(t)|^2-\frac{N}{\beta T}\dot{\widetilde{\omega}}_k(t)\right)c_k\vphantom{\sum_{\substack{k_1-k_2+k_3=k\\k_1,k_3 \neq k_2}}}\right]\notag\\
=&\frac{\beta T}{N}\left[\sum_{\substack{k_1-k_2+k_3=k\\k_1,k_3 \neq k_2}}T_{k,1,2,3}c_{k_1}(t)c^*_{k_2}(t)c_{k_3}(t)e^{i\left(\Omega_{k} Tt+\widetilde{\omega}_k(t)-\widetilde{\omega}_{k_1}(t)+\widetilde{\omega}_{k_2}(t)-\widetilde{\omega}_{k_3}(t)\right)}\right.\notag\\
&+\left.\left(\sum_{k_1}T_{k,1,1,k}|c_{k_1}(t)|^2+\sum_{k_3}T_{k,k,3,3}|c_{k_3}(t)|^2-\frac{N}{\beta T}\dot{\widetilde{\omega}}_k(t)\right)c_k\vphantom{\sum_{\substack{k_1-k_2+k_3=k\\k_1,k_3 \neq k_2}}}\right] \label{eq-2}
\end{align}
Integrating both sides of (\ref{eq-2}), we can get:
\begin{align}
c_k(t) =&c_k(0)-i\frac{\beta T}{N}\int_0^t\left[\sum_{\substack{k_1-k_2+k_3=k\\k_1,k_3 \neq k_2}}T_{k,1,2,3}c_{k_1}(s)c^*_{k_2}(s)c_{k_3}(s)e^{i\left(\Omega_{k} Ts+\widetilde{\omega}_k(s)-\widetilde{\omega}_{k_1}(s)+\widetilde{\omega}_{k_2}(s)-\widetilde{\omega}_{k_3}(s)\right)}\right.\notag\\
&+\left.\left(\sum_{k_1}T_{k,1,1,k}|c_{k_1}(s)|^2+\sum_{k_3}T_{k,k,3,3}|c_{k_3}(s)|^2-\frac{N}{\beta T}\dot{\widetilde{\omega}}_k(s)\right)c_k\vphantom{\sum_{\substack{k_1-k_2+k_3=k\\k_1,k_3 \neq k_2}}}\right]ds \label{eqn.c_k res factor}\\
=&c_k(0)+F(c)\label{eq-3}
\end{align}
\subsection{The choice of frequency shift} \label{sec-choicefs}
At this point, we have not chosen what $\widetilde \omega_k$ is in equation \eqref{phase change}. We would like to choose it to cancel as much as possible of the divergent factor $\sum_{k_1}T_{k,1,1,k} |c_{k_1}(s)|^2+\sum_{k_3}T_{k,k,3,3}|c_{k_3}(s)|^2$ in parentheses appearing in \eqref{eqn.c_k res factor}. At the same time, we would like $\widetilde \omega_k$ to be deterministic (otherwise it becomes difficult to control the moments of the solution) and to have an explicit expansion in terms of the initial data. For this, we will seek a combined expansion of $(c_k(t), \widetilde \omega_k)$ as follows: Similar to (\ref{eq-1}), we want a tree expression for $c_k(t)$ as follows:
\begin{align}
c_k(t) = c^{\leq n}_k(t)+\mathcal{R}^{(n+1)}_{k}=\sum_{|\mathcal{T}^+|\le n}c_k^{\mathcal{T}^+}(t)+\mathcal{R}^{(n+1)}_{k},\label{eq-4}
\end{align}
where $c_k^{\mathcal{T}}(t)$ is defined inductively:
\begin{align}
c_k^{\bullet}(t)&=b_k^{\bullet}(t)=\sqrt{n_{\text{in}}(k)}\eta_{k}(\varrho)=c_k(0),\label{eq-11}\\
c_k^{\mathcal{T}}(t) &= -i\frac{\beta T}{N}\int_0^t\left[\sum_{\substack{k_1-k_2+k_3=k\\k_1,k_3 \neq k_2}}T_{k,1,2,3}c_{k_1}^{\mathcal{T}_1}(s){c_{k_2}^{\mathcal{T}_2}}^*(s)c_{k_3}^{\mathcal{T}_3}(s)e^{i\left(\Omega_{k} Ts+\widetilde \Omega_k(s)\right)}\right.\notag\\
&+\sum_{k_1}T_{k,1,1,k}\left(c_{k_1}^{\mathcal{T}_1}(s){c_{k_1}^{\mathcal{T}_2}}^*(s)-\mathbb{E}\left[c_{k_1}^{\mathcal{T}_1}(s){c_{k_1}^{\mathcal{T}_2}}^*(s)\right]\right)c_k^{\mathcal{T}_3}(s)\notag\\
&+\left.c_k^{\mathcal{T}_1}(s)\sum_{k_3}T_{k,k,3,3}\left({c_{k_3}^{\mathcal{T}_2}}^*(s)c_{k_3}^{\mathcal{T}_3}(s)-\mathbb{E}\left[{c_{k_3}^{\mathcal{T}_2}}^*(s)c_{k_3}^{\mathcal{T}_3}(s)\right]\right)\vphantom{\sum_{\substack{k_1-k_2+k_3=k\\k_1,k_3 \neq k_2}}}\right]ds.\label{eq-12}
\end{align}
Here, $\mathcal{T}$ are ternary trees with subtrees $\mathcal{T}_1,\mathcal{T}_2$ and $\mathcal{T}_3$. The factor $\widetilde \Omega_k(s)$ is given by 
\begin{equation}\label{eq.def of tilde Omega}
\widetilde \Omega_k(s):=\widetilde \omega_k(s)-\widetilde{\omega}_{k_1}(s)+\widetilde{\omega}_{k_2}(s)-\widetilde{\omega}_{k_3}(s)
\end{equation}
and $\widetilde \omega_k$ is chosen such that
\begin{equation}\label{eq.def on tilde omega}
\partial_s\widetilde \omega_k(s)=\frac{\beta T}{N}\left(\sum_{k_1}T_{k,1,1,k} \mathbb E|c^{\leq n}_{k_1}(s)|^2+\sum_{k_3}T_{k,k,3,3}\mathbb E|c^{\leq n}_{k_3}(s)|^2\right).
\end{equation}

Note that this is an ODE system for $(\widetilde \omega_k)_k$. However, in this particular case, and due to the definition of $T_{k,l,l,k}= \frac{3}{4\kappa^2}\omega_{k}\omega_{l}$, this ODE system reduces to a scalar ODE since 
\begin{align}
\partial_s{\widetilde{\omega}}_k(s) &= \frac{\beta T}{N}\left(\sum_{|\mathcal{T}_1|,|\mathcal{T}_2|\le n}\sum_{k_1}T_{k,1,1,k}\mathbb{E}\left[c_{k_1}^{\mathcal{T}_1}(s){c_{k_1}^{\mathcal{T}_2}}^*(s)\right]+\sum_{|\mathcal{T}_2|,|\mathcal{T}_3|\le n}\sum_{k_3}T_{k,k,3,3}\mathbb{E}\left[{c_{k_3}^{\mathcal{T}_2}}^*(s)c_{k_3}^{\mathcal{T}_3}(s)\right]\right)\notag\\
&=\frac{2\beta T}{N}\sum_{|\mathcal{T}_1|,|\mathcal{T}_2|\le n}\sum_{l}T_{k,l,l,k}\mathbb{E}\left[c_{l}^{\mathcal{T}_1}(s){c_{l}^{\mathcal{T}_2}}^*(s)\right]\label{eq-8}\\
&=\omega_k\left[\frac{3\beta T}{2\kappa^2N}\sum_{|\mathcal{T}_1|,|\mathcal{T}_2|\le n}\sum_{l}\omega_l\mathbb{E}\left[c_{l}^{\mathcal{T}_1}(s){c_{l}^{\mathcal{T}_2}}^*(s)\right]\right]
\end{align}

Therefore, writing
\begin{equation}\label{eq.def of A_k}
\widetilde \omega_k(s)=\omega_k A(s), \textit{ }A(0)=0,
\end{equation}
we obtain that 
\begin{align}
\widetilde \Omega_k&=\Omega_k A(s), \\
\dot{A}(s) &= G(s,A(s)) = \frac{3\beta T}{2\kappa^2N}\sum_{|\mathcal{T}_1|,|\mathcal{T}_2|\le n}\sum_{l}\omega_l\mathbb{E}\left[c_{l}^{\mathcal{T}_1}(s){c_{l}^{\mathcal{T}_2}}^*(s)\right]. \label{A(s)}
\end{align}

It is clear from \eqref{eq-12} and \eqref{eq.def of A_k} that $|c_l^{\T}|$ and hence the function $G(s,A(s))$ are bounded for almost every $\varrho$, and that $G(s, A(s))$ is smooth in its variables, which guarantees the global existence and uniqueness of $A(s)$ for almost all $\varrho$.
\subsection{The remainder}
It remains to write down the equation for $\mathcal R_k^{(n+1)}$, which will justify our choices for $c_k^{\mathcal T}$ and $\widetilde \omega_k$ above. Equating (\ref{eq-3}) and (\ref{eq-4}), we have:
\begin{align*}
c^{\leq n}_k(t)+\mathcal{R}^{(n+1)}_{k} = c_k(0)+F(c^{\leq n}+\mathcal{R}^{(n+1)})
\end{align*}
Then we can get the expression for the remainder $\mathcal{R}^{(n+1)}_{k}$:
\begin{align*}
\mathcal{R}^{(n+1)}_{k} = \left[c_k(0)+F(c^{\leq n})-c^{\leq n}_k(t)\right]+\left[F(c^{\leq n}+\mathcal{R}^{(n+1)})-F(c^{\leq n})\right]= I + II
\end{align*}
Calculating the first term:
\begin{align}
I=&c_k(0)+F(c^{\leq n})-c^{\leq n}_k(t)\notag\\
=& -i\frac{\beta T}{N}\int_0^t\left\{\left[\sum_{|\mathcal{T}_1|,|\mathcal{T}_2|,|\mathcal{T}_3|\le n}\sum_{\substack{k_1-k_2+k_3=k\\k_1,k_3 \neq k_2}}T_{k,1,2,3}c_{k_1}^{\mathcal{T}_1}(s){c_{k_2}^{\mathcal{T}_2}}^*(s)c_{k_3}^{\mathcal{T}_3}(s)e^{i\left(\Omega_{k} Ts+\widetilde{\Omega}_k(s)\right)}\right.\right.\notag\\
&+\left(\sum_{|\mathcal{T}_1|,|\mathcal{T}_2|\le n}\sum_{k_1}T_{k,1,1,k}c_{k_1}^{\mathcal{T}_1}(s){c_{k_1}^{\mathcal{T}_2}}^*(s)-\frac{N}{2\beta T}\dot{\widetilde{\omega}}_k(s)\right)\sum_{|\mathcal{T}_3|\le n}c_k^{\mathcal{T}_3}(s)\notag\\
&+\left.\sum_{|\mathcal{T}_1|\le n}c_k^{\mathcal{T}_1}(s)\left(\sum_{|\mathcal{T}_2|,|\mathcal{T}_3|\le n}\sum_{k_3}T_{k,k,3,3}{c_{k_3}^{\mathcal{T}_2}}^*(s)c_{k_3}^{\mathcal{T}_3}(s)-\frac{N}{2\beta T}\dot{\widetilde{\omega}}_k(s)\right)\vphantom{\sum_{\substack{k_1-k_2+k_3=k\\k_1,k_3 \neq k_2}}}\right]\label{eq-5}\\
&-\sum_{|\mathcal{T}_1|+|\mathcal{T}_2|+|\mathcal{T}_3|\le n-1}\left[\sum_{\substack{k_1-k_2+k_3=k\\k_1,k_3 \neq k_2}}T_{k,1,2,3}c_{k_1}^{\mathcal{T}_1}(s){c_{k_2}^{\mathcal{T}_2}}^*(s)c_{k_3}^{\mathcal{T}_3}(s)e^{i\left(\Omega_{k} Ts+\widetilde \Omega_k(s)\right)}\right.\notag\\
&+\sum_{k_1}T_{k,1,1,k}\left(c_{k_1}^{\mathcal{T}_1}(s){c_{k_1}^{\mathcal{T}_2}}^*(s)-\mathbb{E}\left[c_{k_1}^{\mathcal{T}_1}(s){c_{k_1}^{\mathcal{T}_2}}^*(s)\right]\right)c_k^{\mathcal{T}_3}(s)\notag\\
&+\left.\left.c_k^{\mathcal{T}_1}(s)\sum_{k_3}T_{k,k,3,3}\left({c_{k_3}^{\mathcal{T}_2}}^*(s)c_{k_3}^{\mathcal{T}_3}(s)-\mathbb{E}\left[{c_{k_3}^{\mathcal{T}_2}}^*(s)c_{k_3}^{\mathcal{T}_3}(s)\right]\right)\vphantom{\sum_{\substack{k_1-k_2+k_3=k\\k_1,k_3 \neq k_2}}}\right]\right\}ds. \label{eq-6}
\end{align}
This expression justifies our choice of $\partial_s\widetilde \omega_k$ as in \eqref{eq.def on tilde omega} and \eqref{eq-8}, which leads to a large of amount of cancellations in the expressions \eqref{eq-5} and \eqref{eq-6} above. In fact, by \eqref{eq-8}, we have that 
\begin{align}
I=&c_k(0)+F(c^{\leq n})-c^{\leq n}_k(t)\notag\\
=&-i\frac{\beta T}{N}\int_0^t\sum_{\substack{|\mathcal{T}_1|,|\mathcal{T}_2|,|\mathcal{T}_3|\le n\\|\mathcal{T}_1|+|\mathcal{T}_2|+|\mathcal{T}_3|\ge n}}\left[\sum_{\substack{k_1-k_2+k_3=k\\k_1,k_3 \neq k_2}}T_{k,1,2,3}c_{k_1}^{\mathcal{T}_1}(s){c_{k_2}^{\mathcal{T}_2}}^*(s)c_{k_3}^{\mathcal{T}_3}(s)e^{i\left(\Omega_{k} Ts+\widetilde{\Omega}_k(s)\right)}\right.\notag\\
&+\sum_{k_1}T_{k,1,1,k}\left(c_{k_1}^{\mathcal{T}_1}(s){c_{k_1}^{\mathcal{T}_2}}^*(s)-\mathbb{E}\left[c_{k_1}^{\mathcal{T}_1}(s){c_{k_1}^{\mathcal{T}_2}}^*(s)\right]\right)c_k^{\mathcal{T}_3}(s)\notag\\
&+\left.c_k^{\mathcal{T}_1}(s)\sum_{k_3}T_{k,k,3,3}\left({c_{k_3}^{\mathcal{T}_2}}^*(s)c_{k_3}^{\mathcal{T}_3}(s)-\mathbb{E}\left[{c_{k_3}^{\mathcal{T}_2}}^*(s)c_{k_3}^{\mathcal{T}_3}(s)\right]\right)\vphantom{\sum_{\substack{k_1-k_2+k_3=k\\k_1,k_3 \neq k_2}}}\right]ds. 
\end{align}
Introducing the notation: $\textbf{u}=(u_{k}(t))_k,\text{ }\textbf{v}=(v_{k}(t))_k,\text{ and }\textbf{w}=(w_{k}(t))_k$, we define:
\begin{align}
&\mathcal{W}(\textbf{u},\textbf{v},\textbf{w})_k(t)
:=-i\frac{\beta T}{N}\int_0^t\sum_{\substack{k_1-k_2+k_3=k\\k_1,k_3 \neq k_2}}T_{k,1,2,3}u_{k_1}(s)v_{k_2}^*(s)w_{k_3}(s)e^{i\left(\Omega_{k} Ts+\widetilde{\Omega}_k(s)\right)}ds.
\end{align}
Then, we define:
\begin{align}
\mathscr{R}_{k} := I=\sum_{\substack{|\mathcal{T}_1|,|\mathcal{T}_2|,|\mathcal{T}_3|\le n\\|\mathcal{T}_1|+|\mathcal{T}_2|+|\mathcal{T}_3|\ge n}}\mathcal{W}(c^{\mathcal{T}_1},c^{\mathcal{T}_2},c^{\mathcal{T}_3})_k(t)+ \mathscr{R}_D^{\mathcal{T}_1,\mathcal{T}_2,\mathcal{T}_3}\label{eq-9}
\end{align}
where 
\begin{align}
    \mathscr{R}_D^{\mathcal{T}_1,\mathcal{T}_2,\mathcal{T}_3} = -i\frac{\beta T}{N}\int_0^t \sum_{k_1}T_{k,1,1,k}\left(c_{k_1}^{\mathcal{T}_1}(s){c_{k_1}^{\mathcal{T}_2}}^*(s)-\mathbb{E}\left[c_{k_1}^{\mathcal{T}_1}(s){c_{k_1}^{\mathcal{T}_2}}^*(s)\right]\right)c_k^{\mathcal{T}_3}(s)\notag\\
    +c_k^{\mathcal{T}_1}(s)\sum_{k_3}T_{k,k,3,3}\left({c_{k_3}^{\mathcal{T}_2}}^*(s)c_{k_3}^{\mathcal{T}_3}(s)-\mathbb{E}\left[{c_{k_3}^{\mathcal{T}_2}}^*(s)c_{k_3}^{\mathcal{T}_3}(s)\right]\right)ds
\end{align}
Similarly, the second term can be expressed as:
\begin{align}
&II=F(c^{\leq n}+\mathcal{R}^{(n+1)})-F(c^{\leq n})\notag\\
&=-i\frac{\beta T}{N}\int_0^t\sum_{\substack{k_1-k_2+k_3=k\\k_1,k_3 \neq k_2}}T_{k,1,2,3}\left(c^{\leq n}_{k_1}(s)+\mathcal{R}^{(n+1)}_{k_1}\right)\left(c^{\leq n *}_{k_2}(s)+\mathcal{R}^{(n+1)}_{k_2}\right)\notag\\
&\qquad\qquad\qquad\qquad\qquad\qquad\qquad\qquad\qquad \times\left(c^{\leq n}_{k_3}(s)+\mathcal{R}^{(n+1)}_{k_3}\right)e^{i\left(\Omega_{k} Ts+\widetilde \Omega_k(s)\right)}\notag\\
&+2\sum_{k_1}T_{k,1,1,k}\left(\left|c^{\leq n}_{k_1}(s)+\mathcal{R}^{(n+1)}_{k_1}\right|^2 - \E \left|c^{\leq n}_{k_1}(s) \right|^2\right)\left(c^{\leq n}_{k}(s)+\mathcal{R}^{(n+1)}_{k}\right)\notag\\
&-\sum_{\substack{k_1-k_2+k_3=k\\k_1,k_3 \neq k_2}}T_{k,1,2,3}c^{\leq n}_{k_1}(s)c^{\leq n *}_{k_2}(s)c^{\leq n}_{k_3}(s)e^{i\left(\Omega_{k} Ts+\widetilde \Omega_k(s)\right)}\notag\\
&-2\sum_{k_1}T_{k,1,1,k}\left(|c^{\leq n}_{k_1}(s)|^2-\E |c^{\leq n}_{k_1}(s)|^2\right)c^{\leq n}_k \; ds
\end{align}
We further define:
\begin{align}
\mathscr{L}_k(\mathcal{R}^{(n+1)})&:=\sum_{|\mathcal{T}_1|,|\mathcal{T}_2| \le n}^{\text{cyc}}\mathcal{W}(c^{\mathcal{T}_1},c^{\mathcal{T}_2},\mathcal{R}^{(n+1)})_k(t)+(\mathscr{L}_D)_k \label{eq-linearop}\\
\mathscr{Q}_k(\mathcal{R}^{(n+1)},\mathcal{R}^{(n+1)})&:=\sum_{|\mathcal{T}_1| \le n}^{\text{cyc}}\mathcal{W}(c^{\mathcal{T}_1},\mathcal{R}^{(n+1)},\mathcal{R}^{(n+1)})_k(t)+(\mathscr{Q}_D)_k\\
\mathscr{C}_k(\mathcal{R}^{(n+1)},\mathcal{R}^{(n+1)},\mathcal{R}^{(n+1)}) &:=\mathcal{W}(\mathcal{R}^{(n+1)},\mathcal{R}^{(n+1)},\mathcal{R}^{(n+1)})_k(t) +(\mathscr{C}_D)_k
\end{align} \label{eq-multilinear}
where $\Ell, \mathscr Q, $ and $\mathscr C$ are linear, bilinear, and trilinear operators respectively, and
\begin{align}
    (\mathscr{L}_D)_k &:= -i\frac{\beta T}{N}\int_0^t 2\sum_{k_1}T_{k,1,1,k}\left[\left(\left|c^{\leq n}_{k_1}(s)\right|^2 - \E \left|c^{\leq n}_{k_1}(s) \right|^2\right)\mathcal{R}^{(n+1)}_{k}\right.\notag\\
    &{\hspace{3.5cm}}\left.+\left(c^{\leq n*}_{k_1}(s)\mathcal{R}^{(n+1)}_{k_1}+c^{\leq n}_{k_1}(s)\mathcal{R}^{(n+1)*}_{k_1}\right)c^{\leq n}_{k}(s)\right]ds \label{eq-LD}\\
    (\mathscr{Q}_D)_k &:= -i\frac{\beta T}{N}\int_0^t 2\sum_{k_1}T_{k,1,1,k}\left[\left(c^{\leq n*}_{k_1}(s)\mathcal{R}^{(n+1)}_{k_1}+c^{\leq n}_{k_1}(s)\mathcal{R}^{(n+1)*}_{k_1}\right)\mathcal{R}^{(n+1)}_{k}\right.\notag\\
    &{\hspace{3.5cm}}\left.
+\left|\mathcal{R}^{(n+1)}_{k_1}\right|^2c^{\leq n}_{k}(s) \right]ds\\
    (\mathscr{C}_D)_k &:= -i\frac{\beta T}{N}\int_0^t 2\sum_{k_1}T_{k,1,1,k}\left|\mathcal{R}^{(n+1)}_{k_1}\right|^2\mathcal{R}^{(n+1)}_{k} ds
\end{align}
Then by calculation we can verify that:
\begin{align}
    II=\mathscr{L}_k(\mathcal{R}^{(n+1)})+\mathscr{Q}_k(\mathcal{R}^{(n+1)},\mathcal{R}^{(n+1)})+\mathscr{C}_k(\mathcal{R}^{(n+1)},\mathcal{R}^{(n+1)},\mathcal{R}^{(n+1)}).\label{eq-10}
\end{align}
Therefore (\ref{eq-9}) and (\ref{eq-10}) give that 
\begin{align}
\mathcal{R}^{(n+1)}_{k}&=\mathscr{R}_k+\mathscr{L}_k(\mathcal{R}^{(n+1)})+\mathscr{Q}_k(\mathcal{R}^{(n+1)},\mathcal{R}^{(n+1)})+\mathscr{C}_k(\mathcal{R}^{(n+1)},\mathcal{R}^{(n+1)},\mathcal{R}^{(n+1)}),
\end{align}
which is equivalent to:
\begin{align}
    \mathcal{R} = (1 - \Ell)^{-1}(\mathscr R + \mathscr Q(\mathcal{R},\mathcal{R}) + \mathscr C(\mathcal{R},\mathcal{R},\mathcal{R})), \label{eq-op}
\end{align}
provided that $1 - \mathscr L$ is invertible in a suitable space.
\subsection{Couples and Diagrammetic expansions}\label{sec-couple1}
\begin{definition}[Enhanced Trees]
\label{def:enhanced-trees}
An \emph{enhanced tree} is a ternary tree $\mathcal{T}$ (as in Definition~\ref{def:Trees}) equipped with additional structure specified by:
\begin{enumerate}[label=(\roman*)]
\item A designated set $\mathcal{N}_{D}$ of \emph{degenerate branching nodes} within $\mathcal{T}$.
\item For each degenerate node $\mathfrak{n} \in \mathcal{N}_{D}$ with children $\{\mathfrak{n}_1,\mathfrak{n}_2,\mathfrak{n}_3\}$, we specify a choice of a distinguished child $\mathfrak{n}_{c}^*$ taken from $\{\mathfrak{n}_1,\mathfrak{n}_3\}$  (children with the same sign), and label the remaining child of those two as $\mathfrak{n}_{c}$.
\end{enumerate}
\end{definition}

\begin{definition}[Decorations of Enhanced Trees] \label{def:decorations-enhanced-trees}
A decoration $\mathscr D$ of an enhanced tree $\T$ is a set of vectors $\{(k_{\mathfrak n})_{\mathfrak n \in \T}\}$ such that $k_\mathfrak{n}\in \Z_N\cap (0,1)$. Further, if $\mathfrak n$ is a branching node, 
\begin{equation*}
\zeta_{\mathfrak n} k_{\mathfrak n} = \zeta_{\mathfrak n_1} k_{\mathfrak n_1} + \zeta_{\mathfrak n_2} k_{\mathfrak n_2} +\zeta_{\mathfrak n_3} k_{\mathfrak n_3},
\end{equation*}
where $\mathfrak n_1, \mathfrak n_2, \mathfrak n_3$ are the three children of $\mathfrak n$ labeled from left to right. We say $\mathscr D$ is a $k$-decoration if $k_{\mathfrak r} = k$. In addition, whenever $\mathfrak{n}\in \mathcal{N}_{D}$ is a \emph{degenerate} branching node with children $\{\mathfrak{n}_c^*,\mathfrak{n}_2,\mathfrak{n}_c\}$, we impose the \emph{enhanced} constraints:
\[
k_{\mathfrak{n}_2} = k_{\mathfrak{n}_c^*}
\quad\text{and}\quad
k_{\mathfrak{n}} = k_{\mathfrak{n}_c},
\]
following the notation in Definition~\ref{def:enhanced-trees}.
We also define 
\begin{equation}
\epsilon_{\mathscr D} := \prod_{\mathfrak n \in \mathcal N} \epsilon_{k_{\mathfrak n_1}k_{\mathfrak n_2}k_{\mathfrak n_3}}. 
\end{equation}
where 
\begin{equation} \label{eq-eps}
\epsilon_{k_{\mathfrak{n}_1}k_{\mathfrak{n}_2}k_{\mathfrak{n}_3}} := \begin{cases}
-T_{k_\mathfrak{n},k_{\mathfrak{n}_1},k_{\mathfrak{n}_2},k_{\mathfrak{n}_3}} & \text{if } k_{\mathfrak{n}_1}=k_{\mathfrak{n}_2}=k_{\mathfrak{n}_3}; \\
+T_{k_\mathfrak{n},k_{\mathfrak{n}_1},k_{\mathfrak{n}_2},k_{\mathfrak{n}_3}} & \text{otherwise}.
\end{cases}
\end{equation}
\end{definition}
\begin{figure}[ht]
\centering
\begin{tikzpicture}[scale=1.2,
  level distance=1.5cm,
  level 1/.style={sibling distance=4cm},
  level 2/.style={sibling distance=1.5cm},
  level 3/.style={sibling distance=0.8cm}]
  \tikzset{
    solid node/.style={circle,draw,fill=black,inner sep=1.5pt},
    deckdiamond/.style={
      shape=diamond, draw=black, fill=white, minimum size=6pt, inner sep=0pt
    }
  }

  \node[solid node,label=above:{\footnotesize $(\mathfrak r,+)$}] (root) {}
    child {
      node[solid node,label=left:{\footnotesize $(\mathfrak n_1,+)$}] {}
      child { node[solid node,label=below:{\footnotesize $(\mathfrak m_1,+)$}] {} }
      child { node[solid node,label=below:{\footnotesize $(\mathfrak m_2,-)$}] {} }
      child { node[solid node,label=below:{\footnotesize $(\mathfrak m_3,+)$}] {} }
    }
    child {
      node[deckdiamond,label=right:{\footnotesize $(\mathfrak n_2,-)$}] {}
      child { node[solid node,label=below:{\footnotesize $( \mathfrak{l}_c^*,-)$}] {} }
      child { node[solid node,label=below:{\footnotesize $(\mathfrak l_2,+)$}] {} }
      child { node[solid node,label=below:{\footnotesize $( \mathfrak{l}_c,-)$}] {} }
    }
    child {
      node[solid node,label=right:{\footnotesize $(\mathfrak n_3,+)$}] {}
      child { node[solid node,label=below:{\footnotesize $(\mathfrak p_1,+)$}] {} }
      child { node[solid node,label=below:{\footnotesize $(\mathfrak p_2,-)$}] {} }
      child { node[solid node,label=below:{\footnotesize $(\mathfrak p_3,+)$}] {} }
    };
\end{tikzpicture}
\caption{An enhanced ternary tree in which the branching node $(\mathfrak n_{2},-)$ is a degenerate node, such that $k_{\mathfrak{l}_2} = k_{\mathfrak{l}_c^*}
\text{ and }
k_{\mathfrak{n}_2} = k_{\mathfrak{l}_c}$.}
\label{fig:ternary-tree-diamond}
\end{figure}
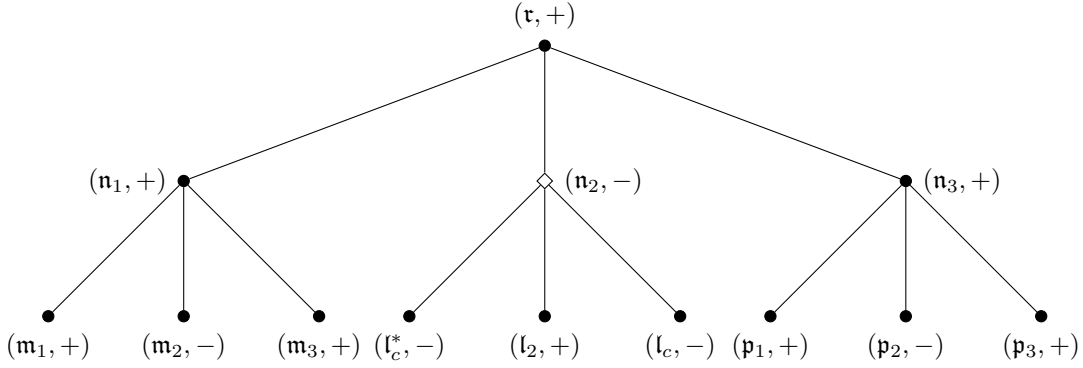
\begin{definition}{(Couples)} \label{def-couples}
A \emph{couple} $\Q$ consists of two trees $\T^+$ and $\T^-$, each labeled with opposite signs, together with a partition $\mathscr{P}$ of the set of leaves $\mathcal{L}^+ \cup \mathcal{L}^-$ into $(n+1)$ disjoint two-element subsets, where $n = n(\T^+) + n(\T^-)$ is called the \emph{order} of the couple. The partition $\mathscr{P}$ must satisfy the condition $\zeta_{\mathfrak{l}} = -\zeta_{\mathfrak{l}'}$ for every pair $\{\mathfrak{l}, \mathfrak{l}'\} \in \mathscr{P}$. For a  couple $\Q = \{\T^+, \T^-, \mathscr P\}$, we denote by $\mathcal{N} = \mathcal{N}^+ \cup \mathcal{N}^-$ the branching nodes of both trees, and by $\mathcal{L} = \mathcal{L}^+ \cup \mathcal{L}^-$ their leaves. We define
\[
\zeta(\Q) \;=\; \prod_{\mathfrak{n} \in \mathcal{N}} \bigl(-i\,\zeta_{\mathfrak{n}}\bigr).
\]
The \emph{trivial couple} is given by two trivial trees whose roots are paired.

A \emph{decoration} $\mathscr{E}$ of a couple $\Q$ is obtained by decorating each tree $\T^+$ and $\T^-$ according to $\mathscr{D}^+$ and $\mathscr{D}^-$, subject to the further requirement that $k_{\mathfrak{l}} = k_{\mathfrak{l}'}$ whenever $\{\mathfrak{l}, \mathfrak{l}'\} \in \mathscr{P}$. We define
\[
\epsilon_{\mathscr{E}} \;=\; \epsilon_{\mathscr{D}^+}\,\epsilon_{\mathscr{D}^-}.
\]
A decoration $\mathscr{E}$ is called a \emph{$k$-decoration} if $k_{\mathfrak{r}^+} = k_{\mathfrak{r}^-} = k$.
\end{definition}
\begin{definition} (Enhanced Couples) 
\label{def:enhanced-couple}
    A couple $\Q$, it is called an \textit{enhanced couple} if its partition $\mathscr P$ further satisfies that the leaves $\mathcal{L}(\mathfrak{n}_2, \mathfrak{n}_c^*)$ originated from $\{\mathfrak{n}_2, \mathfrak{n}_c^*\}$, children of any degenerate branching node $\mathfrak{n}\in \mathcal{N}_{D}$, are not completely paired together, i.e. there exist leaves $\mathfrak{l} \in \mathcal{L}(\mathfrak{n}_2, \mathfrak{n}_c^*)$ and $\mathfrak{l}' \in \mathcal{L} \setminus \mathcal{L}(\mathfrak{n}_2, \mathfrak{n}_c^*)$ such that $\{\mathfrak{l}, \mathfrak{l}'\} \in \mathscr{P}$.
\end{definition}

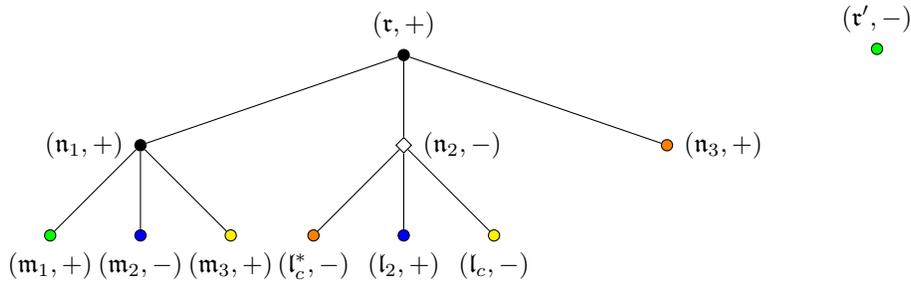
\begin{figure}[ht]
\centering
\begin{tikzpicture}[scale=1,
  baseline=(root.base),
  level distance=1.2cm,
  level 1/.style={sibling distance=3.5cm},
  level 2/.style={sibling distance=1.2cm},
  level 3/.style={sibling distance=0.5cm}]
  \tikzset{
    solid node/.style={circle,draw,fill=black,inner sep=1.5pt},
    deckdiamond/.style={
      shape=diamond, draw=black, fill=white, minimum size=6pt, inner sep=0pt
    }
  }

  \node[solid node,label=above:{\footnotesize $(\mathfrak r,+)$}] (root) {}
    child {
      node[solid node,label=left:{\footnotesize $(\mathfrak n_1,+)$}] {}
      child { node[solid node, fill=green, label=below:{\footnotesize $(\mathfrak m_1,+)$}] {} }
      child { node[solid node, fill=blue,label=below:{\footnotesize $(\mathfrak m_2,-)$}] {} }
      child { node[solid node, fill=yellow,label=below:{\footnotesize $(\mathfrak m_3,+)$}] {} }
    }
    child {
      node[deckdiamond,label=right:{\footnotesize $(\mathfrak n_2,-)$}] {}
      child { node[solid node, fill=orange,label=below:{\footnotesize $( \mathfrak{l}_c^*,-)$}] {} }
      child { node[solid node, fill=blue,label=below:{\footnotesize $(\mathfrak l_2,+)$}] {} }
      child { node[solid node, fill=yellow,label=below:{\footnotesize $( \mathfrak{l}_c,-)$}] {} }
    }
    child {
      node[solid node, fill=orange,label=right:{\footnotesize $(\mathfrak n_3,+)$}] {}
    };
\end{tikzpicture}
\hspace{0.5cm}
\begin{tikzpicture}[scale=1,
  level distance=1.2cm,
  level 1/.style={sibling distance=3.5cm},
  level 2/.style={sibling distance=1.2cm},
  level 3/.style={sibling distance=0.5cm}]
  \tikzset{
    solid node/.style={circle,draw,fill=black,inner sep=1.5pt},
    deckdiamond/.style={
      shape=diamond, draw=black, fill=white, minimum size=6pt, inner sep=0pt
    }
  }

  \node[solid node, fill=green,label=above:{\footnotesize $(\mathfrak r',-)$}] (root) {};
\end{tikzpicture}
\caption{An enhanced couple with leaves paired in the same color. Note that the leaves $\mathfrak{l}_c^*$ and $\mathfrak l_2$ cannot be paired together by Definition \ref{def:enhanced-couple}.}
\label{fig:ternary-tree-diamond-2}
\end{figure}
Then we can prove the following proposition:
\begin{proposition} \label{prop-treeexp}
For each enhanced ternary tree $\mathcal{T}$, we have the formula:
\begin{align}
c_k^{\mathcal{T}}(t)=\left(\frac{\beta T}{N}\right)^{|\mathcal{T}|}\sum_{\mathcal{D}}\prod_{\mathfrak{n}\in \mathcal{N}}(-i\zeta_{\mathfrak{n}}T_{k_\mathfrak{n},k_{\mathfrak{n}_1},k_{\mathfrak{n}_2},k_{\mathfrak{n}_3}})\mathcal{A}_{\mathcal{T}}(t,\Omega[\mathcal{N}], \widetilde{\Omega}[\Nc])\prod_{\mathfrak{l}\in \mathcal{L}}\sqrt{n_{\text{in}}(k_{\mathfrak{l}})}\mathcal{B}_{\mathcal{T}}(\eta_{\mathcal{T}}(\varrho))\label{eq-13}
\end{align}
where $\mathcal{D}$ denotes the decoration $(k_\mathfrak{n}: \mathfrak{n}\in \mathcal{T})$ such that $\mathfrak{r}=k$. $\mathcal{A}_{\mathcal{T}}(t,\Omega[\mathcal{N}], \widetilde{\Omega}[\Nc])$ and $\mathcal{B}_{\mathcal{T}}(\eta_{\mathcal{T}}(\varrho))$ are defined inductively as:
\begin{align}
\mathcal{A}_{\bullet}(t,\Omega[\mathcal{N}], \widetilde{\Omega}[\Nc])&=1,\\
\mathcal{A}_{\mathcal{T}}(t,\Omega[\mathcal{N}], \widetilde{\Omega}[\Nc])&=\int_0^te^{\zeta_{\mathfrak{r}} i(\Omega_{\mathfrak{r}}Tt'+\widetilde{\Omega}_{\mathfrak{r}}(t'))}\prod_{j=1}^3\mathcal{A}_{\mathcal{T}_j}(t',\Omega[\mathcal{N}_j], \widetilde{\Omega}[\Nc_j])dt',\\
\mathcal{B}_{\bullet}(\eta_{\bullet}(\varrho))&=\eta^{\zeta_{\mathfrak{r}}}_{\mathfrak{r}}(\varrho)\\
\mathcal{B}_{\mathcal{T}}(\eta_{\mathcal{T}}(\varrho))&=\prod_{j=1}^3\mathcal{B}_{\mathcal{T}_j}(\eta_{\mathcal{T}_j}(\varrho))-\mathbbm{1}_{\mathfrak{r}\in \mathcal{N}_D}\mathcal{B}_{\mathcal{T}_{c}}(\eta_{\mathcal{T}_{c}}(\varrho))\notag\\
&\times\mathbb{E}\left(\mathcal{B}_{\mathcal{T}_{c^*}}(\eta_{\mathcal{T}_{c^*}}(\varrho))\mathcal{B}_{\mathcal{T}_2}(\eta_{\mathcal{T}_2}(\varrho))\right).
\end{align}
where $\mathcal{T}_{c}$, $\mathcal{T}_{c^*}$, and $\mathcal{T}_2$ are trees with root node $\mathfrak{n}_c$, $\mathfrak{n}_c^*$, $\mathfrak{n}_2$ respectively, which are children of degenerate node $\mathfrak{r}\in \mathcal{N}_D$.
\end{proposition}

\begin{proof}
We now prove (\ref{eq-13}) inductively. From (\ref{eq-11}-\ref{eq-12}), (\ref{eq-13}) is true for $|\mathcal{T}|=0,1$. Suppose (\ref{eq-13}) is true for $|\mathcal{T}|\le n-1$. Then for any ternary tree $\mathcal{T}$ with scale $n$, its children subtrees $\mathcal{T}_1,\mathcal{T}_2$, and $\mathcal{T}_3$ with sign $(+,-,+)$ have scales of at most $(n-1)$. In this setting $\mathcal{T}_2^*=\mathcal{T}_2$ and $c_{k_2}^{{\mathcal{T}_2}^*}(t)=c_{k_2}^{\mathcal{T}_2}(t)$. Then for $i \in \{1,2,3\}$, we have:
\begin{align}
c_{k_i}^{\mathcal{T}_i}(t)=\left(\frac{\beta T}{N}\right)^{|\mathcal{T}_i|}\sum_{\mathcal{D}_i}\prod_{\mathfrak{n}\in \mathcal{N}_i}(-i\zeta_{\mathfrak{n}}T_{k_\mathfrak{n},k_{\mathfrak{n}_1},k_{\mathfrak{n}_2},k_{\mathfrak{n}_3}})\mathcal{A}_{\mathcal{T}_i}(t,\Omega[\mathcal{N}_i], \widetilde{\Omega}[\Nc_i])\prod_{\mathfrak{l}\in \mathcal{L}_i}\sqrt{n_{\text{in}}(k_{\mathfrak{l}})}\mathcal{B}_{\mathcal{T}_i}(\eta_{\mathcal{T}_i}(\varrho)),\label{eq-14}
\end{align}
where $\mathcal{N}=\cup_i\mathcal{N}_i\cup\{\mathfrak{r}\}$, $\mathcal{L}=\cup_i\mathcal{L}_i$. We also have $\mathcal{N}_D\cup\{\mathfrak{r}\}=\cup_i\mathcal{N}^{(i)}_D\cup\{\mathfrak{r}\}$, where $\mathcal{N}^{(i)}_D$ is the set of degenerate branching nodes for $\mathcal{T}_i$.
 Then plugging in (\ref{eq-14}) to (\ref{eq-12}), we can get:
\begin{align}
c_{k}^{\mathcal{T}}(t)=&-i\zeta_{\mathfrak{r}}\left(\frac{\beta T}{N}\right)^{1+\sum_{i=1}^3|\mathcal{T}_i|}\int_0^t\left\{\mathbbm{1}_{\mathfrak{r}\notin \mathcal{N}_D}\sum_{\substack{k_1-k_2+k_3=k\\k_1,k_3 \neq k_2}}T_{k,1,2,3}\prod_{i=1}^3\left[\sum_{\mathcal{D}_i}\prod_{\mathfrak{n}\in \mathcal{N}_i}(-i\zeta_{\mathfrak{n}}T_{k_\mathfrak{n},k_{\mathfrak{n}_1},k_{\mathfrak{n}_2},k_{\mathfrak{n}_3}})\right.\right.\notag\\
&\times \left.\mathcal{A}_{\mathcal{T}_i}(t,\Omega[\mathcal{N}_i], \widetilde{\Omega}[\Nc_i])\prod_{\mathfrak{l}\in \mathcal{L}_i}\sqrt{n_{\text{in}}(k_{\mathfrak{l}})}\mathcal{B}_{\mathcal{T}_i}(\eta_{\mathcal{T}_i}(\varrho))\right]e^{i\left(\Omega_{k} Ts+\widetilde \Omega_k(s)\right)}\notag\\
&+\mathbbm{1}_{\mathfrak{r}\in \mathcal{N}_D}\sum_{k_2}T_{k,k,2,2}\prod_{i=1}^3\left[\sum_{\mathcal{D}_i}\prod_{\mathfrak{n}\in \mathcal{N}_i}(-i\zeta_{\mathfrak{n}}T_{k_\mathfrak{n},k_{\mathfrak{n}_1},k_{\mathfrak{n}_2},k_{\mathfrak{n}_3}})\mathcal{A}_{\mathcal{T}_i}(t,\Omega[\mathcal{N}_i], \widetilde{\Omega}[\Nc_i])\prod_{\mathfrak{l}\in \mathcal{L}_i}\sqrt{n_{\text{in}}(k_{\mathfrak{l}})}\right]\notag\\
&\times \left.\left[\prod_{j=1}^3\mathcal{B}_{\mathcal{T}_j}(\eta_{\mathcal{T}_j}(\varrho))-\mathcal{B}_{\mathcal{T}_{c}}(\eta_{\mathcal{T}_{c}}(\varrho))\mathbb{E}\left(\mathcal{B}_{\mathcal{T}_{c^*}}(\eta_{\mathcal{T}_{c^*}}(\varrho))\mathcal{B}_{\mathcal{T}_2}(\eta_{\mathcal{T}_2}(\varrho))\right)\right]\vphantom{\sum_{\substack{k_1-k_2+k_3=k\\k_1,k_3 \neq k_2}}}\right\}ds \label{eq-15}
\end{align}
After simplification, we find (\ref{eq-15}) recovers (\ref{eq-13}).
\end{proof}
\begin{definition}
    For any ternary tree $\mathcal{T}$ and decoration $(k_\mathfrak{n}: \mathfrak{n}\in \mathcal{T})$, we fix $d_{\mathfrak{n}}\in \{0,1\}$ for each $\mathfrak{n}\in \mathcal{N}$, and we define $q_{\mathfrak{n}}$ for each $\mathfrak{n}\in \mathcal{T}$ inductively by:
\begin{align}
    q_{\mathfrak{n}} &= 0 \text{ if } \mathfrak{n}\in \mathcal{L}\text{; } q_{\mathfrak{n}} = d_{\mathfrak{l}}q_{\mathfrak{l}}- d_{\mathfrak{n}_2}q_{\mathfrak{n}_2}+d_{\mathfrak{n}_3}q_{\mathfrak{n}_3}+\Omega_{\mathfrak{n}} \text{ if }\mathfrak{n}\in \mathcal{N}.
\end{align}
\end{definition}
Also, we include here the needed estimates and expansion of $A(t)$:
\begin{align}
\dot{A}(t) &= \frac{3\beta T}{2\kappa^2N}\sum_{|\mathcal{T}_1|,|\mathcal{T}_2|\le n}\sum_{l}\omega_l\mathbb{E}\left[c_{l}^{\mathcal{T}_1}(t){c_{l}^{\mathcal{T}_2}}^*(t)\right]\notag\\
&=\frac{3\beta T}{2\kappa^2N}\left[\sum_{\substack{|\mathcal{T}_1|=0\\|\mathcal{T}_2|=0}}\sum_{l}\omega_l\mathbb{E}\left[c_{l}^{\mathcal{T}_1}(t){c_{l}^{\mathcal{T}_2}}^*(t)\right]+\sum_{|\mathcal{T}_1|+|\mathcal{T}_2| \geq 1}\sum_{l}\omega_l\mathbb{E}\left[c_{l}^{\mathcal{T}_1}(t){c_{l}^{\mathcal{T}_2}}^*(t)\right]\right]\notag\\
&=\frac{3\beta T}{2\kappa^2N}\left[\sum_{l}\omega_ln_{\text{in}}(l)+\sum_{|\mathcal{T}_1|+|\mathcal{T}_2| \geq 1}\sum_{l}\omega_l\mathbb{E}\left[c_{l}^{\mathcal{T}_1}(t){c_{l}^{\mathcal{T}_2}}^*(t)\right]\right]\\
&=:C_0\beta T+\dot{A}_{\geq}(t).
\end{align}
Then we can express $A(t)$ as: $A(t)=C_0\beta Tt+{A}_{\geq}(t)$.
\begin{definition} \label{couple expression}
For an enhanced couple $\Q$, define
\begin{align}\label{eq-KQ}
(\K_\Q)(t,s,k)  &:= \left( \frac{\beta T}{N}\right)^n \,\zeta(\Q)\,
\sum_{\mathscr E} \epsilon_{\mathscr E}
\int_{\mathcal{E}}\prod_{\mathfrak{n} \in \mathcal{N}}e^{\zeta_{\mathfrak{n}}i(\Omega_{\mathfrak{n}}Tt_{\mathfrak{n}}+\widetilde{\Omega}_{\mathfrak{n}}(t_{\mathfrak{n}}))} d t_{\mathfrak{n}}
\;\prod_{\mathfrak{l} \in \mathcal{L}}^{+} n_{\mathrm{in}}\bigl(k_{\mathfrak{l}}\bigr)\notag\\
&\hspace{6cm}\times \E \left[\mathcal{B}_{\mathcal{T}^+}(\eta_{\mathcal{T}^+}(\varrho))\mathcal{B}^*_{\mathcal{T}^-}(\eta_{\mathcal{T}^-}(\varrho))\right],
\end{align}
where the sum is taken over all $k$-decorations $\mathscr{E}$ of $\Q$, the product 
$\prod_{\mathfrak{l} \in \mathcal{L}}^{+}$ is taken over leaves with $+$ signs, and
\begin{align}\label{eq-E}
\mathcal{E} &= \Bigl\{\,t[\mathcal{N}] : 0 < t_{\mathfrak{n}'} < t_{\mathfrak{n}}
  \text{ whenever $\mathfrak{n}'$ is a child of $\mathfrak{n}$}; \\
&\qquad\qquad\quad t_{\mathfrak{n}} < t \text{ for } \mathfrak{n} \in \mathcal{N}^{+},
  \text{ and } t_{\mathfrak{n}} < s \text{ for } \mathfrak{n} \in \mathcal{N}^{-}
\Bigr\}, \nonumber
\end{align}
where $t[\mathcal{N}]$ denotes the time vector $(t_\mathfrak{n})_{\mathfrak{n} \in \mathcal{N}}$.
\end{definition}
\begin{lemma}{(Complex Isserlis' Theorem)}
Given $k_j \in \Z_N \cap (0,1)$ and $\{\zeta_j\}$ for $1 \leq j \leq n$, then 
\begin{equation} \label{eq-isserlis}
\E \left[\prod_{j = 1}^n \eta_{k_j}^{\zeta_j}(\varrho) \right] = \sum_{\mathscr P} \prod_{\{j, j'\} \in \mathscr P} \boldsymbol{1}_{k_j = k_{j'}}.
\end{equation}
where $\eta_{k}^{+}=\eta_{k}$ and $\eta_{k}^{-}=\overline{\eta_{k}}$, and the summation is taken over all partitions $\mathscr P$ of $1, 2, \ldots, n$ into two-element subsets $\{j, j'\}$ such that $\zeta_j = -\zeta_{j'}$.
\end{lemma}
\begin{proof}
Here, we repeat the proof in \cite{WKE}. Let $\{k^{(1)},\dots,k^{(r)}\}$ be the collection of distinct vectors in
$\{k_j : 1 \le j \le n\}$.  
For every $1\le i\le r$ and $\zeta\in\{\pm\}$ define
\[
  p_i^\zeta
  \;=\;
  \bigl|\{\,j : k_j = k^{(i)},\; \zeta_j = \zeta\,\}\bigr|.
\]
With this notation we can rewrite the joint moment as  
\begin{equation*}
\E \left[\prod_{j = 1}^n \eta_{k_j}^{\zeta_j}(\varrho) \right] = \E \left[ \prod_{i = 1}^r \eta_{k^{i}}^{p_i^+}\overline{\eta_{k^{i}}} ^{p_i^-} \right] = \begin{cases}
(p_1^+)! \ldots (p_r^+)! & \textrm{if } p_i^+ = p_i^- \textrm{ for all $i$,} \\
0 & \mathrm{otherwise}. 
\end{cases}
\end{equation*}
For fixed $\mathscr P$ in \eqref{eq-isserlis} the product is 0 or 1. If there is $i$ such that $p_i^+ \neq p_i^-$, there is no partition $\mathscr P$. Otherwise, the product is 1 if and only if each pair $\{j,j'\} \in \mathscr P$ satisfies $\zeta_j = -\zeta_{j'}$, and  $k_j = k_{j'} = k^{(i)}$ for some $i$. So, the number of such $\mathscr P$ is $(p_1^+)! \ldots (p_r^+)!$. 
\end{proof}
Note that by Isserlis' Theorem, we have:
\begin{equation}
\E \left(c^{\T^+}_k(t)\overline{c^{\T^-}_k(s)}\right) = \sum_{\mathscr P} (\K_{\Q})(t,s,k). 
\end{equation}
\subsection{Molecules}\label{sec-mol}
\begin{definition}[Molecules]\label{def-molecules}
A \textit{molecule} $\mathbb{M}$ is a directed graph whose vertices are called \textit{atoms} and edges are called \textit{bonds}. Multiple bonds between the same pair of atoms and self-loops are allowed. Each atom has out-degree at most 2 and in-degree at most 2. We write $v \in \mathbb{M}$ for an atom $v$ in $\mathbb{M}$, and $\ell \in \mathbb{M}$ for a bond $\ell$ in $\mathbb{M}$. If $v$ is an endpoint of $\ell$, we write $\ell \sim v$. For each such pair $(v,\ell)$, define
\[
\zeta_{v,\ell} = \begin{cases}
1 & \text{if $\ell$ is outgoing from $v$,}\\
-1 & \text{if $\ell$ is incoming to $v$.}
\end{cases}
\]
We further require that $\mathbb{M}$ does not have any connected components consisting only of atoms each having degree 4 (where the degree is considered in the undirected sense).

For a molecule $\mathbb{M}$, define the quantity
\[
\chi := E - V + F,
\]
where $E$ is the number of bonds, $V$ is the number of atoms, and $F$ is the number of connected components.

An \textit{atomic group} in $\mathbb{M}$ is a subset of atoms together with all bonds connecting those atoms. A single bond $\ell$ is called a \textit{bridge} if removing it increases the number of connected components by one.
\end{definition}

\begin{definition}[Molecules of couples]\label{couple-molecule}
Let $\Q$ be a nontrivial couple. We define the corresponding \textit{molecule} $\mathbb{M} = \mathbb{M}(\Q)$ as follows. The atoms of $\mathbb{M}$ correspond to the branching nodes $\mathfrak{n} \in \mathcal{N}$ of $\Q$. For any two branching nodes $\mathfrak{n}_1, \mathfrak{n}_2$, we draw a bond between the corresponding atoms $v_1, v_2$ if:
\begin{enumerate}
\item One of $\mathfrak{n}_1, \mathfrak{n}_2$ is a parent of the other (Parent-Child pair), or
\item A child of $\mathfrak{n}_1$ is paired with a child of $\mathfrak{n}_2$ as leaves (Leaf Pair).
\end{enumerate}

We may form a \textit{labeled molecule} by labeling each bond. For a Parent-Child (PC) bond, we label the bond as ``PC'' and place a ``P'' at the atom corresponding to the parent branching node and a ``C'' at the atom corresponding to the child branching node. For a Leaf Pair (LP) bond, we label the bond as ``LP''.

The direction of each bond is determined as follows:
\begin{enumerate}[label=(\roman*)]
\item For an LP bond, the bond is directed away from the atom whose corresponding child in the couple has sign ``$-$'' and towards the atom whose corresponding child has sign ``$+$''.
\item For a PC bond, if the child atom corresponds to a branching node with sign ``$-$'', then the direction goes from the ``P'' atom to the ``C'' atom. If the child atom corresponds to a branching node with sign ``$+$'', the direction is reversed.
\end{enumerate}

For any atom $v \in \mathbb{M}(\Q)$, let $\mathfrak{n} = \mathfrak{n}(v)$ be its corresponding branching node in $\Q$. For any bond $\ell \sim v$, define $\mathfrak{m} = \mathfrak{m}(v,\ell)$ by:
\begin{enumerate}[label=(\roman*)]
\item If $\ell$ is PC and $v$ is labeled $C$, then $\mathfrak{m} = \mathfrak{n}$.
\item If $\ell$ is PC and $v$ is labeled $P$, then $\mathfrak{m}$ is the branching node corresponding to the other endpoint of $\ell$.
\item If $\ell$ is LP, then $\mathfrak{m}$ is the leaf (from the leaf pair defining $\ell$) that is a child of $\mathfrak{n}$.
\end{enumerate}
\end{definition}
\begin{proposition}[Correspondence between Molecules and Couples]
\label{prop:molecules-couples}
We summarize the relationship between nontrivial couples and molecules as follows:
\begin{enumerate}
\item For any nontrivial couple $\mathcal{Q}$ of order $n$, the construction in Definition~\ref{couple-molecule} yields a connected \emph{molecule} $\mathbb{M}(\mathcal{Q})$.  This molecule $\mathbb{M}$ has exactly $n$ atoms, $2n-1$ bonds, and either two atoms of degree $3$, or a single atom of degree $2$, with all remaining atoms having degree $4$.
\item Conversely, given any molecule $\mathbb{M}$ with $n$ atoms (as in Definition~\ref{def-molecules}), there are at most $C^n$ different couples $\mathcal{Q}$ (if any) for which $\mathbb{M}(\mathcal{Q})$ is exactly $\mathbb{M}$.
\end{enumerate}
\end{proposition}
\begin{proof}
We repeat the relevant parts of the proofs of Proposition 9.2, 9.4, and 9.6 of \cite{WKE} here.
\noindent{(a)}\;
Let $\Q$ be a molecule of order $n\ge 1$ and write $\M=\M(\Q)$ for its associated molecular graph.  
Because each branching node of $\Q$ corresponds to exactly one atom of $\M$, the graph $\M$ contains $n$ atoms.  
Likewise, every non-root branching node gives one bond and every paired leaf contributes one further bond, for a total of  
\[
2n-1=(n-2)+(n+1)
\]
bonds in $\M$.  

A priori, $\M$ might decompose into two connected components; each component must contain at least one atom associated with a branching {\it root} whose degree is strictly smaller than~4.  Suppose a component contains $n_1$ atoms.  Since the sum of degrees in any finite graph is even and each atom has degree at most~4, that sum is bounded by $4n_1$.  For the whole graph $\M$ the degree sum equals $4n-2$.  If $\M$ were disconnected, one component would necessarily have all of its atoms of degree~4, contradicting the preceding observation.  Hence $\M$ is connected.

\noindent{(b)}\;
A mere labeling of the atoms of $\M$ does not yet specify the couple $\Q$; one also has to record the order of the corresponding nodes.  Observe that every bond $\ell\sim v$ of $\M$ corresponds to a unique node of $\Q$.  We encode that correspondence by assigning to the pair $(v,\ell)$ a \emph{code} that tells where the node sits inside its atom:  
\[
0 \text{ for the parent},\quad 1,2,3 \text{ for the left, middle, right child}.
\]
For a fixed molecular graph $\M$ there are at most $C^{n}$ distinct ways to carry out this coding, with $C$ a universal constant.

To see that such a coding reconstructs $\Q$ uniquely, note first that atoms of $\M$ bijectively represent branching nodes of $\Q$.  If two atoms $v_1,v_2\in\M$ are joined by a bond~$\ell$ and the codes of $(v_1,\ell)$ and $(v_2,\ell)$ are $0$ and $j$, then $v_2$ is the $j$-th child of $v_1$ in the couple.  Conversely, if their codes are $j$ and $k$ with $j,k\in\{1,2,3\}$, the $j$-th child of the node represented by $v_1$ is paired with the $k$-th child of the node represented by $v_2$.  Hence the entire set of codes fixes all parent–child relations and all leaf pairings, and therefore determines the couple $\Q$ uniquely.
\end{proof}
\begin{definition}[Decorations of Molecules]
\label{def:decorations-molecules}
Let $\mathbb{M}$ be a molecule. For each atom (vertex) $v \in \mathbb{M}$, we fix $k_v \in \Z_N\cap (0,1)$, with the condition $k_v = 0$ if $v$ has degree 4, and $\alpha_v \in \mathbb{R}$.

A \emph{$(k_v,\alpha_v)$-decoration} of $\mathbb{M}$ assigns an integer $k_\ell \in \Z_N\cap (0,1)$ to each bond $\ell \in \mathbb{M}$, in such a way that for each atom $v$,
\[
\sum_{\ell \sim v} \zeta_{v,\ell} \,k_\ell \;=\; k_v,
\quad\text{and}\quad
\bigl|\Gamma_v - \alpha_v\bigr|\;<\;T^{-1},
\]
where
\[
\Gamma_v = \sum_{\ell: \ell\sim v} \zeta_{v,\ell} \, \omega(k_\ell).
\]
Here, the summation $\sum_{\ell \sim v}$ runs over all bonds $\ell$ such that $\ell\sim v$, and $\zeta_{v,\ell}$ is defined to be $+1$ whenever $\ell$ is outgoing from $v$, and $-1$ otherwise. $\omega(\cdot)$ is the dispersion relation taking values in $\mathbb{R}$. 

\medskip

\noindent\textbf{$k$-Decorations.}
Suppose $\mathbb{M}$ arises from a nontrivial couple $\mathcal{Q}$, and let $k\in\Z_N\cap (0,1)$. A \emph{$k$-decoration} of $\mathbb{M}$ is a special $(k_v,\alpha_v)$-decoration where
\[
k_v \;=\;
\begin{cases}
0 & \text{if $v$ has degree 2 or 4,}\\
+k & \text{if $v$ has out-degree 2 and in-degree 1,}\\
-k & \text{if $v$ has out-degree 1 and in-degree 2.}
\end{cases}
\]
Given any $k$-decoration of $\mathcal Q$ in the sense of Definition \ref{def-couples}, define a $k$-decoration of $\mathbb M(\mathcal Q)$ such that $k_{\ell} = k_{\mathfrak m(v,\ell)}$ for an endpoint $v$ of $\ell$. $k_\ell$ is well-defined and independent of the choice of $v$, and gives a one-to-one correspondence between $k$-decorations of $\mathcal Q$ and $k$-decorations of $\mathbb M(\mathcal Q)$. Moreover, for such decorations we have 
\begin{equation}
\Gamma_v = \begin{cases}
0 & \text{if } v \text{ has degree 2,} \\
-\zeta_{\mathfrak n(v)}\Omega_{\mathfrak n(v)} & \text{if } v \text{ has degree 4,} \\
-\zeta_{\mathfrak n(v)}\Omega_{\mathfrak n(v)} + \omega(k) & \text{if } v \text{ has out-degree 2 and in-degree 1,} \\
-\zeta_{\mathfrak n(v)}\Omega_{\mathfrak n(v)} - \omega(k) & \text{if } v \text{ has out-degree 1 and in-degree 2.} \\
\end{cases}
\end{equation}
\end{definition}
\begin{figure}
\centering

\begin{tikzpicture}[
        scale = .8, baseline=(top.base),
        level distance=1.4cm,
        level 1/.style={sibling distance=2.1cm},
        level 2/.style={sibling distance=0.8cm},
        hollow node/.style  ={circle,draw,inner sep=1.6},
        diamond node/.style ={diamond,draw,inner sep=1.6},
        solid node/.style   ={circle,draw,fill=black,inner sep=1.6},
        red node/.style     ={circle,draw,fill=red,inner sep=1.6},
        blue node/.style    ={circle,draw,fill=blue,inner sep=1.6},
        purple node/.style  ={circle,draw,fill=purple,inner sep=1.6},
        orange node/.style  ={circle,draw,fill=orange,inner sep=1.6},
        yellow node/.style  ={circle,draw,fill=yellow,inner sep=1.6},
        green  node/.style  ={circle,draw,fill=green,inner sep=1.6}]
  \node[hollow node, label = left: {\tiny $k$}] (top) {1}
    child{node[hollow node, label = left: {\tiny $m$}]{2}
      child{node[green node,  label = below: {\tiny $k$}]{}}
      child{node[blue node,   label = below: {\tiny $p$}]{}}
      child{node[yellow node, label = below: {\tiny $l$}]{}}
    }
    child{node[diamond node, label = left:{\tiny $l$}]{3}
      child{node[orange node, label = below: {\tiny $p$}]{}}
      child{node[blue node,   label = below: {\tiny $p$}]{}}
      child{node[yellow node, label = below: {\tiny $l$}]{}}
    }
    child{node[orange node,  label = below: {\tiny $p$}]{}};
\end{tikzpicture}
\hspace{0.5cm}
\begin{tikzpicture}[
        scale = 1, baseline=(top.base),
        level distance=1.4cm,
        level 1/.style={sibling distance=1.3cm},
        level 2/.style={sibling distance=1.0cm},
        hollow node/.style  ={circle,draw,inner sep=1.6},
        diamond node/.style ={diamond,draw,inner sep=1.6},
        green  node/.style  ={circle,draw,fill=green,inner sep=1.6}]
  \node[green node, label = left: {\tiny $k$}] (top) {};
\end{tikzpicture}
\hspace{1.2cm}
\begin{tikzpicture}[
        scale=0.8, baseline=(top.base), >=stealth,
        open node/.style   ={circle,draw,inner sep=1.8},
        diamond node/.style={diamond,draw,inner sep=1.5},
        empty node/.style  ={inner sep=.5,outer sep=0}]
  \def\L{4}
  \def\shift{3pt}
  \coordinate (A) at (0,0);           
  \coordinate (B) at (\L,0);          
  \coordinate (C) at ($(A)+(60:\L)$); 

  \node[open node]    (N2) at (A) {2};
  \node[diamond node] (N3) at (B) {3};
  \node[open node]    (top) at (C) {1}; 

  \draw[->,thick,transform canvas={yshift=\shift}]
      (N3) -- node[empty node,pos=.4,fill=white] {$l$} (N2);
  \draw[->,thick,transform canvas={yshift=-\shift}]
      (N2) -- node[empty node,pos=.4,fill=white] {$p$} (N3);

  \draw[->,thick,transform canvas={xshift=\shift, yshift=2.5pt}]
      (N3) -- node[empty node,pos=.4,fill=white] {$p$} (top);
  \draw[->,thick,transform canvas={xshift=-\shift, yshift=-2.5pt}]
      (top) -- node[empty node,pos=.4,fill=white] {$l$} (N3);

  \draw[->,thick]
      (N2) -- node[empty node,midway,fill=white] {$m$} (top);
\end{tikzpicture}
\caption{An example of an enhanced couple and its corresponding enhanced molecule.}
\end{figure}
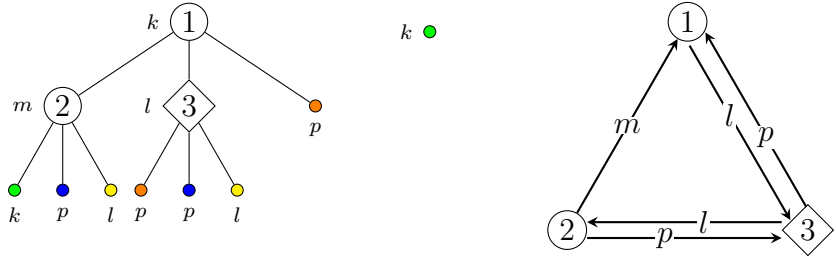
\begin{definition}[Enhanced Molecule of an Enhanced couple]
\label{def:enhanced-molecule}
Let $\Q$ be an enhanced couple of order $n$. We construct its \emph{enhanced molecule} $\mathbb{M}(\Q)$ in the same manner that one constructs a molecule from a standard couple (cf.\ Definition~\ref{couple-molecule}), but incorporating the following \emph{enhanced} configuration:
\begin{enumerate}[label=(\roman*)]
    \item For any degenerate branching node $\mathfrak{n}\in\mathcal{N}_D$, there exists a corresponding atom $v$ in the enhanced molecule $\mathbb{M}(\Q)$.
    \item The atoms $v_c^*,v_2$, and $v_c$ correspond precisely to the children $\{\mathfrak{n}_c^*,\mathfrak{n}_2, \mathfrak{n}_c\}$ of the degenerate branching node $\mathfrak{n}$ or their corresponding paired branching nodes.
    \item Moreover, removing a single degenerate atom $v$ together with all bonds  $\ell$ such that $\ell\sim v$, then adding a new (PC) bond linking $v_c$ (if it exists in the molecule) to $v_p$ (if $v$ has a (PC) bond $\ell_p$ connecting it to its parent $v_p$) the connectivity of $\mathbb{M}(\Q)$ remains unchanged, i.e. $\Delta F = 0$.
\end{enumerate}
\end{definition}
\begin{definition}[Degenerate Atoms]
\label{def:degenerate-atoms}
Consider an enhanced molecule $\mathbb{M}$ equipped with a $(k_v,\alpha_v)$-decoration as above. We say that an atom $v \in \mathbb{M}$ is \emph{degenerate} if there exist two \emph{degenerate} bonds $\ell_1, \ell_2 \sim v$, pointing in opposite directions at $v$, such that $k_{\ell_1}=k_{\ell_2}$. Furthermore, $v$ is \emph{fully degenerate} if \emph{all} bonds $\ell\sim v$ share the same decoration $k_\ell$. In particular, for each degenerate branching node in the couple $\Q$, there corresponds a degenerate atom in the molecule $\mathbb{M}(\Q)$.
\end{definition}

\section{Main Estimates}
\label{section-mainest}
The purpose of this section is to provide the statements of the main estimates including the integral estimates and the Feynman diagram estimates.
\subsection{The integral estimates} \label{sec-int}
\begin{proposition}\label{integral_est}
    For any ternary tree $\mathcal{T}$, fix functions $f_\mathfrak{n}:[0,T]\to\mathbb{C}$ for each node $\mathfrak{n} \in \mathcal{T}$:
    \begin{align}
        \left|f_{\mathfrak{n}}(Ts)\right| &\leq D_{\mathfrak{n}}, \label{eq-50001}\\
        \left|\frac{d}{ds} f_{\mathfrak{n}}(Ts)\right| &=\left|T\dot f_{\mathfrak{n}}(Ts)\right| \leq D_{\mathfrak{n}}, \label{eq-50002}
    \end{align}
for some constants $D_{\mathfrak{n}}$. Suppose $\mathcal{T}$ has subtree children $\mathcal{T}_j$, $j=1,2,3$ of the root $\mathfrak{r}$, and the function $ \mathcal{A}^f_{\mathcal{T}}(s,\Omega[\mathcal{N}], \widetilde{\Omega}[\Nc])$ with $0 \leq s \leq 1$, is defined as:
\begin{align}
    \mathcal{A}^f_{\bullet}(s,\Omega[\mathcal{N}], \widetilde{\Omega}[\Nc])&=1,\label{eq-4002}\\
    \mathcal{A}^f_{\mathcal{T}}(s,\Omega[\mathcal{N}], \widetilde{\Omega}[\Nc])&=\int_0^se^{\zeta_{\mathfrak{r}} i(\Omega_{\mathfrak{r}}Ts'+\widetilde{\Omega}_{\mathfrak{r}}(s'))}f_{\mathfrak{r}}(Ts')\prod_{j=1}^3\mathcal{A}^f_{\mathcal{T}_j}(s',\Omega[\mathcal{N}_j], \widetilde{\Omega}[\Nc_j])ds'\\
    &=\int_{\mathcal{E}}\prod_{\mathfrak{n}\in \mathcal{N}}e^{\zeta_{\mathfrak{n}} i(\Omega_{\mathfrak{n}}Ts_{\mathfrak{n}}+\widetilde{\Omega}_{\mathfrak{n}}(s_{\mathfrak{n}}))}f_{\mathfrak{n}}(Ts_{\mathfrak{n}})ds_{\mathfrak{n}}
\end{align}
where $\mathcal{E} =\left \{s[\mathcal{N}]:0<s_{\mathfrak{n}'}<s_{\mathfrak{n}}<s\text{ whenever }{\mathfrak{n}'}\in \mathcal{N}\text{ is a child of }{\mathfrak{n}}\in \mathcal{N}\right\}$.
For any node $\mathfrak{n} \in \mathcal{N}$, $\Omega_{\mathfrak{n}}$'s are constants and for some constant $C_0$:
\begin{align}
     \widetilde{\Omega}_{\mathfrak{n}}(s) &:= \Omega_{\mathfrak{n}}\left[C_0\beta Ts+{A}_{\geq}(s)\right], \\
     \left|\dot{A}_{\geq}(s)\right| & \lesssim  \beta T,\\
     \left|\ddot{A}_{\geq}(s)\right| & \lesssim \beta T^{\frac{9}{5}}.
\end{align}
Then for any node $\mathfrak{n}\in \mathcal{N}$ there exist some constants $C_{\mathfrak{n}}$ such that $C_{\mathfrak{n}} \leq C_d D_{\mathfrak{n}}$ for some constant $C_d$, and
\begin{align}
    \left|\mathcal{A}^f_{\mathcal{T}}(s,\Omega[\mathcal{N}], \widetilde{\Omega}[\Nc])\right|&\leq \sum_{(d_{\mathfrak{n}}:\mathfrak{n}\in \mathcal{N})}\prod_{\mathfrak{n}\in \mathcal{N}}C_{\mathfrak{n}}{\langle Tq_{\mathfrak{n}}\rangle}^{-1}. \label{eq-400001}
\end{align}
assuming $T < \beta^{-\frac{5}{4}}$.
\end{proposition} 
\begin{proof}
We know that $\left|\dot{A}_{\geq}({s_{\mathfrak{n}}})\right| \lesssim \beta T$, which gives that:
\begin{align}
    \left|\frac{T}{T+C_0\beta T+\dot{A}_{\geq}({s_{\mathfrak{n}}})}\right| \lesssim 1. \label{eq-4007}
\end{align}
Applying integration by parts to scale $|\mathcal{T}|=1$ case:
\begin{align}
    &\left|\int_0^{s}e^{\zeta_{\mathfrak{r}} i(\Omega_{\mathfrak{r}}Ts'+\widetilde{\Omega}_{\mathfrak{r}}(Ts'))}f_{\mathfrak{r}}(Ts')ds'\right|\notag\\
    =&\;\frac{1}{iT\Omega_{\mathfrak{r}}}\left|\frac{e^{i\left(T\Omega_{\mathfrak{r}} s+\widetilde{\Omega}_{\mathfrak{r}}(s)\right)}Tf_{\mathfrak{r}}(Ts)}{T+C_0\beta T+\dot{A}_{\geq}(s)}-\frac{f_{\mathfrak{r}}(0)}{1+C_0\beta}+\int_0^{s}\frac{e^{i\left(T\Omega_{\mathfrak{r}} s'+\widetilde{\Omega}_{\mathfrak{r}}(s')\right)}T\ddot{A}_{\geq}(s')f_{\mathfrak{r}}(Ts')}{(T+C_0\beta T+\dot{A}_{\geq}(s'))^{2}}ds'\right.\notag\\
    &\left.+\int_0^{s}\frac{e^{i\left(T\Omega_{\mathfrak{r}} s'+\widetilde{\Omega}_{\mathfrak{r}}(s')\right)}T^2\dot{f}_{\mathfrak{r}}(Ts')}{T+C_0\beta T+\dot{A}_{\geq}(s')}ds'\right|\lesssim \frac{1}{\langle T\Omega_{\mathfrak{r}}\rangle}D_{\mathfrak{r}}\left(3+\beta T^{\frac{4}{5}}\right).\label{eq-4003}  
\end{align}
From (\ref{eq-400002}) and (\ref{eq-4003}), we know (\ref{eq-400001}) is true for $|\mathcal{T}|=0,1$ since $\beta T^{\frac{4}{5}}< 1$. Suppose (\ref{eq-400001}) is true for $|\mathcal{T}|\le n-1$. Any ternary tree $\mathcal{T}$ with scale $n$ can be obtained from a scale $(n-1)$ tree $\mathcal{T}'\subseteq \mathcal{T}$ by attaching a scale $1$ tree $\mathcal{T}_1 \subseteq \mathcal{T}$ to one of its lowest level leaves $\mathfrak{l}$. This leaf $\mathfrak{l}$ should be chosen so that its siblings have at most scale $1$. Let $\mathfrak{p}$ denote the parent of the lowest level leaf $\mathfrak{l}$ we chose, $\mathcal{T}_\mathfrak{p}  \subseteq \mathcal{T}$ denote the subtree rooted at $\mathfrak{p}$, and $\mathcal{N}^-_\mathfrak{p}=\mathcal{N}\setminus \mathcal{T}_\mathfrak{p}$ denote the branching nodes of $\mathcal{T}$ excluding the descendants of $\mathfrak{p}$ and itself.  We define:
\begin{align}
    \mathcal{A}(s[\mathcal{N}^-_\mathfrak{p}]):= \prod_{\mathfrak{n}\in \mathcal{N}^-_\mathfrak{p}}e^{\zeta_{\mathfrak{n}} i(\Omega_{\mathfrak{n}}Ts_{\mathfrak{n}}+\widetilde{\Omega}_{\mathfrak{n}}(s_{\mathfrak{n}}))}f_{\mathfrak{n}}(Ts_{\mathfrak{n}})ds_{\mathfrak{n}},
\end{align}
where $s[\mathcal{N}^-_\mathfrak{p}]$ denotes the time vector $(s_\mathfrak{n})_{\mathfrak{n} \in \mathcal{N}^-_\mathfrak{p}}$. Without loss of generality, we can assume $\zeta_{\mathfrak{p}}=\zeta_{\mathfrak{l}}=+1$. Then $\mathcal{A}^f_{\mathcal{T}'}(s,\Omega[\mathcal{N}'], \widetilde{\Omega}[\Nc'])$ and $\mathcal{A}^f_{\mathcal{T}}(s,\Omega[\mathcal{N}], \widetilde{\Omega}[\Nc])$ can also be expressed as:
\begin{align}
    \mathcal{A}^f_{\mathcal{T}'}(s,\Omega[\mathcal{N}'], \widetilde{\Omega}[\Nc'])&=\int_{\mathcal{E}_\mathfrak{p}} \mathcal{A}(s[\mathcal{N}^-_\mathfrak{p}])e^{i\left(T\Omega_{\mathfrak{p}}{s_{\mathfrak{p}}}+\widetilde{\Omega}_{\mathfrak{p}}({s_{\mathfrak{p}}})\right)}f_{\mathfrak{p}}(T{s_{\mathfrak{p}}})\prod_{j=2}^3\mathcal{A}^f_{\mathcal{T}_j}({s_{\mathfrak{p}}},\Omega[\mathcal{N}_j], \widetilde{\Omega}[\Nc_j])d{s_{\mathfrak{p}}},\\
    \mathcal{A}^f_{\mathcal{T}}(s,\Omega[\mathcal{N}], \widetilde{\Omega}[\Nc])&=\int_{\mathcal{E}_\mathfrak{p}} \mathcal{A}(s[\mathcal{N}^-_\mathfrak{p}])e^{i\left(T\Omega_{\mathfrak{p}}{s_{\mathfrak{p}}}+\widetilde{\Omega}_{\mathfrak{p}}({s_{\mathfrak{p}}})\right)}f_{\mathfrak{p}}(T{s_{\mathfrak{p}}})\notag\\
    &\times \int_0^{s_{\mathfrak{p}}}e^{i\left(T\Omega_{\mathfrak{l}} {s_{\mathfrak{l}}}+\widetilde{\Omega}_{\mathfrak{l}}({s_{\mathfrak{l}}})\right)}f_{\mathfrak{l}}(T{s_{\mathfrak{l}}})d{s_{\mathfrak{l}}}\prod_{j=2}^3\mathcal{A}^f_{\mathcal{T}_j}({s_{\mathfrak{p}}},\Omega[\mathcal{N}_j], \widetilde{\Omega}[\Nc_j])d{s_{\mathfrak{p}}}
\end{align}
where $\mathcal{E}_\mathfrak{p} = \left\{s[\mathcal{N}^-_\mathfrak{p}]:0<s_{\mathfrak{n}'}<s_{\mathfrak{n}}<s\text{ whenever }{\mathfrak{n}'}\in \mathcal{N}^-_\mathfrak{p}\text{ is a child of }{\mathfrak{n}}\in \mathcal{N}^-_\mathfrak{p}\right\}$. Then by the hypotheses there exist constants $C_{\mathfrak{n}}$'s associated with nodes $\mathfrak{n} \in \mathcal{N}'$ such that:
\begin{align}
    \left|\mathcal{A}^f_{\mathcal{T}'}(s,\Omega[\mathcal{N}'], \widetilde{\Omega}[\Nc'])\right| &\leq \sum_{(d_{\mathfrak{n}}:\mathfrak{n}\in \mathcal{N}')}\prod_{\mathfrak{n}\in \mathcal{N}'}C_{\mathfrak{n}}{\langle Tq_{\mathfrak{n}}\rangle}^{-1}, \label{eq-4004}
\end{align}
provided that:
\begin{align}
    \left|f_{\mathfrak{p}}(Ts)\right| &\leq D_{\mathfrak{p}}, \label{eq-50007}\\
    \left|\frac{d}{ds} f_{\mathfrak{p}}(Ts)\right| &=\left|T\dot f_{\mathfrak{p}}(Ts)\right| \leq D_{\mathfrak{p}}. \label{eq-50008}
\end{align}
for constant $D_{\mathfrak{p}}$. Now assume:
\begin{align}
    \left|f_{\mathfrak{l}}(Ts)\right| &\leq D_{\mathfrak{l}}, \label{eq-5009}\\
    \left|\frac{d}{ds} f_{\mathfrak{l}}(Ts)\right| &=\left|T\dot f_{\mathfrak{l}}(Ts)\right| \leq D_{\mathfrak{l}}, \label{eq-5010}
\end{align}
for constant $D_{\mathfrak{l}}$ assuming that $T < \beta^{-\frac{5}{4}}$. Then both $f_{\mathfrak{p}} \text{ and } f_{\mathfrak{l}}$ satisfy properties (\ref{eq-50001}-\ref{eq-50002}).
We calculate $\mathcal{A}^f_{\mathcal{T}}(Ts,\Omega[\mathcal{N}], \widetilde{\Omega}[\Nc])$ using integration by parts from (\ref{eq-4003}):
\begin{align}
    \mathcal{A}^f_{\mathcal{T}}(Ts,\Omega[\mathcal{N}], \widetilde{\Omega}[\Nc])=&\frac{1}{iT\Omega_{\mathfrak{l}}}\int_{\mathcal{E}_\mathfrak{p}} \mathcal{A}(s[\mathcal{N}^-_\mathfrak{p}])e^{i\left(T\Omega_{\mathfrak{p}}{s_{\mathfrak{p}}}+\widetilde{\Omega}_{\mathfrak{p}}({s_{\mathfrak{p}}})\right)}f_{\mathfrak{p}}(T{s_{\mathfrak{p}}})\left[\frac{e^{i\left(T\Omega_{\mathfrak{l}} {s_{\mathfrak{p}}}+\widetilde{\Omega}_{\mathfrak{l}}({s_{\mathfrak{p}}})\right)}Tf_{\mathfrak{l}}(T{s_{\mathfrak{p}}})}{T+C_0\beta T+\dot{A}_{\geq}({s_{\mathfrak{p}}})}\right.\notag\\
    &\left.-\frac{f_{\mathfrak{p}}(0)}{1+C_0\beta}
    +\int_0^{{s_{\mathfrak{p}}}}\frac{e^{i\left(T\Omega_{\mathfrak{l}} s_{\mathfrak{l}}+\widetilde{\Omega}_{\mathfrak{l}}(s_{\mathfrak{l}})\right)}T\ddot{A}_{\geq}(s_{\mathfrak{l}})f_{\mathfrak{l}}(T{s_{\mathfrak{l}}})}{(T+C_0\beta T+\dot{A}_{\geq}(s_{\mathfrak{l}}))^{2}}ds_{\mathfrak{l}}\right.\notag\\
    &\left.+\int_0^{{s_{\mathfrak{p}}}}\frac{e^{i\left(T\Omega_{\mathfrak{l}} s_{\mathfrak{l}}+\widetilde{\Omega}_{\mathfrak{l}}(s_{\mathfrak{l}})\right)}T^2\dot{f}_{\mathfrak{l}}(Ts_{\mathfrak{l}})}{T+C_0\beta T+\dot{A}_{\geq}(s_{\mathfrak{l}})}ds_{\mathfrak{l}}\right]\notag\\
    &\times\prod_{j=2}^3\mathcal{A}^f_{\mathcal{T}_j}({s_{\mathfrak{p}}},\Omega[\mathcal{N}_j], \widetilde{\Omega}[\Nc_j])d{s_{\mathfrak{p}}}\\
    =&I + II + III,
\end{align}
where
\begin{align}
    I :=& \frac{1}{iT\Omega_{\mathfrak{l}}}\int_{\mathcal{E}_\mathfrak{p}} \mathcal{A}(s[\mathcal{N}^-_\mathfrak{p}])e^{i\left(T(\Omega_{\mathfrak{p}}+\Omega_{\mathfrak{l}}){s_{\mathfrak{p}}}+(\widetilde{\Omega}_{\mathfrak{p}}+\widetilde{\Omega}_{\mathfrak{l}})({s_{\mathfrak{p}}})\right)} \frac{Tf_{\mathfrak{p}}(T{s_{\mathfrak{p}}})f_{\mathfrak{l}}(T{s_{\mathfrak{p}}})}{T+C_0\beta T+\dot{A}_{\geq}({s_{\mathfrak{p}}})}\notag\\
    &\times \prod_{j=2}^3\mathcal{A}^f_{\mathcal{T}_j}({s_{\mathfrak{p}}},\Omega[\mathcal{N}_j], \widetilde{\Omega}[\Nc_j])d{s_{\mathfrak{p}}},\\
    II :=& -\frac{f_{\mathfrak{p}}(0)}{iT\Omega_{\mathfrak{l}}(1+C_0\beta)}\mathcal{A}^f_{\mathcal{T}'}(s,\Omega[\mathcal{N}'], \widetilde{\Omega}[\Nc']),\\
    III :=& \frac{1}{iT\Omega_{\mathfrak{l}}}\int_{\mathcal{E}_\mathfrak{p}} \mathcal{A}(s[\mathcal{N}^-_\mathfrak{p}])e^{i\left(T\Omega_{\mathfrak{p}}{s_{\mathfrak{p}}}+\widetilde{\Omega}_{\mathfrak{p}}({s_{\mathfrak{p}}})\right)}f_{\mathfrak{p}}(T{s_{\mathfrak{p}}})\prod_{j=2}^3\mathcal{A}^f_{\mathcal{T}_j}({s_{\mathfrak{p}}},\Omega[\mathcal{N}_j], \widetilde{\Omega}[\Nc_j])
    \notag\\
    &\times\int_0^{{s_{\mathfrak{p}}}}e^{i\left(T\Omega_{\mathfrak{l}} s_{\mathfrak{l}}+\widetilde{\Omega}_{\mathfrak{l}}(s_{\mathfrak{l}})\right)}T\left[\frac{\ddot{A}_{\geq}(s_{\mathfrak{l}})f_{\mathfrak{l}}(T{s_{\mathfrak{l}}})}{(T+C_0\beta T+\dot{A}_{\geq}(s_{\mathfrak{l}}))^{2}}+\frac{T\dot{f}_{\mathfrak{l}}(Ts_{\mathfrak{l}})}{T+C_0\beta T+\dot{A}_{\geq}(s_{\mathfrak{l}})}\right]ds_{\mathfrak{l}}d{s_{\mathfrak{p}}}.
\end{align}
Note that for $I$, we define:
\begin{align}
    f^{I}_{\mathfrak{p},\mathfrak{l}}(T{s_{\mathfrak{p}}}) := \frac{Tf_{\mathfrak{p}}(T{s_{\mathfrak{p}}})f_{\mathfrak{l}}(T{s_{\mathfrak{p}}})}{T+C_0\beta T+\dot{A}_{\geq}({s_{\mathfrak{p}}})},
\end{align}
which gives:
\begin{align}
    \frac{d}{ds_{\mathfrak{p}}}f^{I}_{\mathfrak{p},\mathfrak{l}}(T{s_{\mathfrak{p}}}) = \frac{T^2\left(\dot{f}_{\mathfrak{p}}(T{s_{\mathfrak{p}}})f_{\mathfrak{l}}(T{s_{\mathfrak{p}}})+{f}_{\mathfrak{p}}(T{s_{\mathfrak{p}}})\dot{f}_{\mathfrak{l}}(T{s_{\mathfrak{p}}})\right)}{T+C_0\beta T+\dot{A}_{\geq}({s_{\mathfrak{p}}})}-\frac{Tf_{\mathfrak{p}}(T{s_{\mathfrak{p}}})f_{\mathfrak{l}}(T{s_{\mathfrak{p}}})\ddot{A}_{\geq}({s_{\mathfrak{p}}})}{(T+C_0\beta T+\dot{A}_{\geq}({s_{\mathfrak{p}}}))^2}.
\end{align}
Then from (\ref{eq-50001}-\ref{eq-50002}), (\ref{eq-4007}), we can derive:
\begin{align}
    \left|f^{I}_{\mathfrak{p},\mathfrak{l}}(T{s_{\mathfrak{p}}})\right| \leq D_{\mathfrak{p}}D_{\mathfrak{l}}\text{, }
    \left|\frac{d}{ds_{\mathfrak{p}}}f^{I}_{\mathfrak{p},\mathfrak{l}}(T{s_{\mathfrak{p}}})\right| \lesssim (2+\beta T^{\frac{4}{5}})D_{\mathfrak{p}}D_{\mathfrak{l}}, \label{eq-50003}
\end{align}
which means that the function $f^{I}_{\mathfrak{p},\mathfrak{l}}(T{s_{\mathfrak{p}}})$ also satisfies properties (\ref{eq-50001}-\ref{eq-50002}) and $I$ can be expressed by:
\begin{align}
    I = \frac{1}{iT\Omega_{\mathfrak{l}}}\int_{\mathcal{E}_\mathfrak{p}} \mathcal{A}(s[\mathcal{N}^-_\mathfrak{p}])e^{i\left(T(\Omega_{\mathfrak{p}}+\Omega_{\mathfrak{l}}){s_{\mathfrak{p}}}+(\widetilde{\Omega}_{\mathfrak{p}}+\widetilde{\Omega}_{\mathfrak{l}})({s_{\mathfrak{p}}})\right)}f^{I}_{\mathfrak{p},\mathfrak{l}}(T{s_{\mathfrak{p}}})\prod_{j=2}^3\mathcal{A}^f_{\mathcal{T}_j}({s_{\mathfrak{p}}},\Omega[\mathcal{N}_j], \widetilde{\Omega}[\Nc_j])d{s_{\mathfrak{p}}},
\end{align}
which can be derived in the form of $\mathcal{A}^f$ from a scale $(n-1)$ tree multiplied by a factor $\frac{1}{iT\Omega_{\mathfrak{l}}}$. Then from (\ref{eq-50001}), (\ref{eq-4004}), and (\ref{eq-50003}) we know there exist constants $C_{\mathfrak{n}}$'s such that:
\begin{align}
    |I| \leq \frac{1}{\langle T\Omega_{\mathfrak{l}}\rangle}\sum_{(d_{\mathfrak{n}}:\mathfrak{n}\in \mathcal{N}')}\prod_{\mathfrak{n}\in \mathcal{N}'}C_{\mathfrak{n}}{\langle Tq_{\mathfrak{n}}\rangle}^{-1} \leq \sum_{(d_{\mathfrak{n}}:\mathfrak{n}\in \mathcal{N})}\prod_{\mathfrak{n}\in \mathcal{N}}C_{\mathfrak{n}}{\langle Tq_{\mathfrak{n}}\rangle}^{-1} \label{eq-4005}
\end{align}
for $\rho \leq 1$. Here, $\langle Tq_{\mathfrak{n}}\rangle=\sqrt{1+(Tq_{\mathfrak{n}})^2}$ denotes the Japanese bracket. Similarly, we can bound $II$ by the hypothesis (\ref{eq-4004}) on $\mathcal{A}^f_{\mathcal{T}'}(s,\Omega[\mathcal{N}'], \widetilde{\Omega}[\Nc'])$ and (\ref{eq-50007}) on $f_{\mathfrak{p}}$ respectively:
\begin{align}
    |II| \leq \frac{1}{\langle T\Omega_{\mathfrak{l}}\rangle}\sum_{(d_{\mathfrak{n}}:\mathfrak{n}\in \mathcal{N}')}\prod_{\mathfrak{n}\in \mathcal{N}'}C_{\mathfrak{n}}{\langle Tq_{\mathfrak{n}}\rangle}^{-1}\leq \sum_{(d_{\mathfrak{n}}:\mathfrak{n}\in \mathcal{N})}\prod_{\mathfrak{n}\in \mathcal{N}}C_{\mathfrak{n}}{\langle Tq_{\mathfrak{n}}\rangle}^{-1} \label{eq-40010}
\end{align}
For $III$, we similarly define:
\begin{align}
    g_{\mathfrak{l}}(T{s_{\mathfrak{l}}}) &:= T\left[\frac{\ddot{A}_{\geq}(s_{\mathfrak{l}})f_{\mathfrak{l}}(T{s_{\mathfrak{l}}})}{(T+C_0\beta T+\dot{A}_{\geq}(s_{\mathfrak{l}}))^{2}}+\frac{T\dot{f}_{\mathfrak{l}}(Ts_{\mathfrak{l}})}{T+C_0\beta T+\dot{A}_{\geq}(s_{\mathfrak{l}})}\right],\\
    f^{III}_{\mathfrak{p},\mathfrak{l}}(T{s_{\mathfrak{p}}}) &:= f_{\mathfrak{p}}(T{s_{\mathfrak{p}}})\int_0^{{s_{\mathfrak{p}}}}e^{i\left(T\Omega_{\mathfrak{l}} s_{\mathfrak{l}}+\widetilde{\Omega}_{\mathfrak{l}}(s_{\mathfrak{l}})\right)}g_{\mathfrak{l}}(T{s_{\mathfrak{l}}})ds_{\mathfrak{l}}.
\end{align}
We can easily derive from (\ref{eq-50001}-\ref{eq-50002}), and (\ref{eq-4007}):
\begin{align}
    \left|g_{\mathfrak{l}}(T{s_{\mathfrak{l}}})\right| \lesssim (1+\beta T^{\frac{4}{5}})D_{\mathfrak{l}}.
\end{align}
Then we can get:
\begin{align}
    \left|f^{III}_{\mathfrak{p},\mathfrak{l}}(T{s_{\mathfrak{p}}})\right|&=\left|f_{\mathfrak{p}}(T{s_{\mathfrak{p}}})\int_0^{{s_{\mathfrak{p}}}}e^{i\left(T\Omega_{\mathfrak{l}} s_{\mathfrak{l}}+\widetilde{\Omega}_{\mathfrak{l}}(s_{\mathfrak{l}})\right)}g_{\mathfrak{l}}(T{s_{\mathfrak{l}}})ds_{\mathfrak{l}}\right|\notag\\
    &\leq \left|f_{\mathfrak{p}}(T{s_{\mathfrak{p}}})|\cdot|g_{\mathfrak{l}}(T{s_{\mathfrak{p}}})\right| \leq (1+\beta T^{\frac{4}{5}})D_{\mathfrak{p}}D_{\mathfrak{l}},\\
    \left|\frac{d}{ds_{\mathfrak{p}}}f^{III}_{\mathfrak{p},\mathfrak{l}}(T{s_{\mathfrak{p}}})\right|&=\left|\frac{d}{ds_{\mathfrak{p}}} \left[f_{\mathfrak{p}}(T{s_{\mathfrak{p}}})\int_0^{{s_{\mathfrak{p}}}}e^{i\left(T\Omega_{\mathfrak{l}} s_{\mathfrak{l}}+\widetilde{\Omega}_{\mathfrak{l}}(s_{\mathfrak{l}})\right)}g_{\mathfrak{l}}(T{s_{\mathfrak{l}}})ds_{\mathfrak{l}}\right]\right| \notag\\
    &\leq \left(\left|\frac{d}{ds_{\mathfrak{p}}}f_{\mathfrak{p}}(T{s_{\mathfrak{p}}})\right|+|f_{\mathfrak{p}}(T{s_{\mathfrak{p}}})|\right)\cdot|g_{\mathfrak{l}}(T{s_{\mathfrak{p}}})|\notag\\
    &\lesssim 2(1+\beta T^{\frac{4}{5}})D_{\mathfrak{p}}D_{\mathfrak{l}}
\end{align}
The function $f^{III}_{\mathfrak{p},\mathfrak{l}}(T{s_{\mathfrak{p}}})$ also satisfies the property (\ref{eq-50001}-\ref{eq-50002}), which means $III$ can also be derived in the form of $\mathcal{A}^f$ from a scale $(n-1)$ tree multiplied by a factor $\frac{1}{iT\Omega_{\mathfrak{l}}}$. Then there exist constants $C_{\mathfrak{n}}$'s such that:
\begin{align}
    |III| &=\left|\frac{1}{iT\Omega_{\mathfrak{l}}}\int_{\mathcal{E}_\mathfrak{p}} \mathcal{A}(s[\mathcal{N}^-_\mathfrak{p}])e^{i\left(T\Omega_{\mathfrak{p}}{s_{\mathfrak{p}}}+\widetilde{\Omega}_{\mathfrak{p}}({s_{\mathfrak{p}}})\right)}f^{III}_{\mathfrak{p},\mathfrak{l}}(T{s_{\mathfrak{p}}})\prod_{j=2}^3\mathcal{A}^f_{\mathcal{T}_j}({s_{\mathfrak{p}}},\Omega[\mathcal{N}_j], \widetilde{\Omega}[\Nc_j])
    d{s_{\mathfrak{p}}}\right|\notag\\
    &\leq  \frac{1}{\langle T\Omega_{\mathfrak{l}}\rangle}\sum_{(d_{\mathfrak{n}}:\mathfrak{n}\in \mathcal{N}')}\prod_{\mathfrak{n}\in \mathcal{N}'}C_{\mathfrak{n}}{\langle Tq_{\mathfrak{n}}\rangle}^{-1}\leq \sum_{(d_{\mathfrak{n}}:\mathfrak{n}\in \mathcal{N})}\prod_{\mathfrak{n}\in \mathcal{N}}C_{\mathfrak{n}}{\langle Tq_{\mathfrak{n}}\rangle}^{-1}. \label{eq-4006}
\end{align}
Then combining (\ref{eq-4005}), (\ref{eq-40010}), and (\ref{eq-4006}), we get:
\begin{align*}
    \left|\mathcal{A}^f_{\mathcal{T}}(s,\Omega[\mathcal{N}], \widetilde{\Omega}[\Nc])\right| \leq \sum_{(d_{\mathfrak{n}}:\mathfrak{n}\in \mathcal{N})}\prod_{\mathfrak{n}\in \mathcal{N}}C_{\mathfrak{n}}{\langle Tq_{\mathfrak{n}}\rangle}^{-1}.
\end{align*}
\end{proof}
\begin{corollary}\label{integral_est_1}
    For any ternary tree $\mathcal{T}$ with subtree children $\mathcal{T}_j$, $j=1,2,3$ of the root $\mathfrak{r}$, let the function $ \mathcal{A}_{\mathcal{T}}(s,\Omega[\mathcal{N}], \widetilde{\Omega}[\Nc])$ with $0 \leq s \leq 1$, be defined as:
\begin{align}
    \mathcal{A}_{\bullet}(s,\Omega[\mathcal{N}], \widetilde{\Omega}[\Nc])&=1,\label{eq-400002}\\
    \mathcal{A}_{\mathcal{T}}(s,\Omega[\mathcal{N}], \widetilde{\Omega}[\Nc])&=\int_0^se^{\zeta_{\mathfrak{r}} i(\Omega_{\mathfrak{r}}Ts'+\widetilde{\Omega}_{\mathfrak{r}}(s'))}\prod_{j=1}^3\mathcal{A}_{\mathcal{T}_j}(s',\Omega[\mathcal{N}_j], \widetilde{\Omega}[\Nc_j])ds'
\end{align}
where for any node $\mathfrak{n} \in \mathcal{N}$, $\Omega_{\mathfrak{n}}$'s are constants and for some constant $C_0$:
\begin{align}
     \widetilde{\Omega}_{\mathfrak{n}}(s) &= \Omega_{\mathfrak{n}}\left[(C_0\beta Ts+{A}_{\geq}(s)\right], \\
     \left|\dot{A}_{\geq}(s)\right| & \lesssim  \beta T,\\
     \left|\ddot{A}_{\geq}(s)\right| & \lesssim \beta T^{\frac{9}{5}}.
\end{align}
Then for each node $\mathfrak{n}\in \mathcal{N}$ there exists a constant $C_{\mathfrak{n}}$, such that $ \mathcal{A}_{\mathcal{T}}(s,\Omega[\mathcal{N}], \widetilde{\Omega}[\Nc])$ satisfies:
\begin{align}
    \left| \mathcal{A}_{\mathcal{T}}(s,\Omega[\mathcal{N}], \widetilde{\Omega}[\Nc])\right|&\leq \sum_{(d_{\mathfrak{n}}:\mathfrak{n}\in \mathcal{N})}\prod_{\mathfrak{n}\in \mathcal{N}}C_{\mathfrak{n}}{\langle Tq_{\mathfrak{n}}\rangle}^{-1} \label{eq-40001}
\end{align}
assuming $T < \beta^{-\frac{5}{4}}$.
\end{corollary}
\begin{proof}
    We know that:
\begin{align}
     \mathcal{A}_{\mathcal{T}}(s,\Omega[\mathcal{N}], \widetilde{\Omega}[\Nc]) = \mathcal{A}^{f \equiv1}_{\mathcal{T}}(s,\Omega[\mathcal{N}], \widetilde{\Omega}[\Nc]),
\end{align}
and (\ref{eq-400001}) automatically yields (\ref{eq-40001}) by taking $f_{\mathfrak{n}} \equiv 1$ for any node $\mathfrak{n} \in \mathcal{T}$.
\end{proof}
\subsection{Bounds on Feynman diagrams}
We fix a sufficiently large constant $L\ll N$.
\begin{proposition}(Bounds on Couples)\label{bound-couple}
    For each $1 \leq n_{\Q} \leq L^3$, $0 \leq t \leq  1$, and enhanced trees $\mathcal{T}_{1}, \mathcal{T}_{2}$, the following should hold:
\begin{align}
     \left|\sum_{|\mathcal{T}_{1}|+|\mathcal{T}_{2}|=n_{\Q}}\mathbb{E}\left[c_{k}^{\mathcal{T}_{1}}(t){c_{k}^{\mathcal{T}_{2}}}^*(t)\right]\right|&\leq C (\log N \log T)^{n_{\Q}} T^{-\frac{3}{5}}\left(\beta T^{4/5}\right)^{n_{\Q}},\label{E1}\\
     \left|\sum_{|\mathcal{T}_{1}|+|\mathcal{T}_{2}|=n_{\Q}}\mathbb{E}\left[\dot{c}_{k}^{\mathcal{T}_{1}}(t){c_{k}^{\mathcal{T}_{2}}}^*(t)\right]\right|&\leq C (\log N \log T)^{n_{\Q}-1}T^{\frac{2}{5}}\left(\beta T^{4/5}\right)^{n_{\Q}},\label{E2}
\end{align}
for some constant $C$.
\end{proposition}
\begin{proposition}{(Bound on Operator $\Ell$)} \label{prop-remainder}
With probability $\geq 1 - N^{-A}$, the linear operator $\Ell$ defined in (\ref{eq-linearop}) satisfies  
\begin{equation} \label{remainder}
||\Ell^n||_{Z \to Z} \lesssim \left(\beta T^{4/5}\right)^{n/2}N^{50}
\end{equation}
for each $ 0 \leq n \leq L$ and some $A \geq 40$. 
\end{proposition}
The proofs of these two propositions will be given in Sections \ref{section-couples} and \ref{section-remainder} respectively.
\subsection{Plan for the rest of the thesis}
In Section \ref{section-preparations}, we introduced the definitions of enhanced trees, enhanced couples and molecule structures, and we chose and performed the nonlinear frequency shift. In Section 3, we proved the integral estimates Proposition \ref{integral_est} \& \ref{integral_est_1}, and stated the bounds on couples and the operator $\Ell$.

In Section \ref{section-counting}, we prove the three-vector and two-vector counting estimates and also the convergence of iterates used for the main theorem. In Section \ref{section-couples}, we introduce the irregular chains and perform the splicing to remove the small-gap chains in Section \ref{irregular}. We remove the degenerate nodes by performing an additional preprocessing step before the algorithm, and prove that the molecule structure remains the same as in \cite{ODW} in Section \ref{subsec-ass} \& \ref{sec-optype}. We then prove Proposition \ref{bound-couple} in Section \ref{sec-boundoncouple} using a bootstrap argument.

In Section \ref{section-remainder} we introduce the flower trees and flower couples, and prove Proposition \ref{prop-remainder}. In Section \ref{section-main}, we bound the remainder terms using a contraction mapping argument. Finally, we prove the main Theorem \ref{main}, deriving the wave kinetic equation for the reduced evolution equation.

\section{Counting Estimates}
\label{section-counting}
In this section, we prove the vector counting estimates and the convergence of iterates. Denote $\Z_N:=N^{-1}\Z$ and $\mathcal{S}$ as the Schwartz space.
\begin{proposition}[Three Vector Counting] \label{vector_counting}
Let $1<T\leq N^{1-\epsilon}$ for $\epsilon \ll 1$. Then uniformly in $k\in \mathbb{Z}_N \cap (-1,2) \setminus \{0,1\}$ and $m\in \mathbb{R}$, the set:
\begin{align}
    S_3 &= \Big\{(x,y,z)\in \mathbb{Z}_N^3 \cap (0,1)^3: x-y+z = k,\notag\\
    &\left|\sin\left(\pi x\right)-\sin\left(\pi y\right)+\sin\left(\pi z\right)-|\sin\left(\pi k\right)|-m\right| \leq T^{-1}\Big\}
\end{align}
satisfies the bound:
\begin{align}
    \sum_{\substack{(x,y,z)}\in S_3}F(x,y) &\lesssim N^2T^{-1}\log T,
\end{align}
where 
\begin{align}
    F(x,y) &= |\sin(\pi x)\sin(\pi y)\sin(\pi (k-x+y))|.
\end{align}
\end{proposition}
\begin{proof}
Suppose $\Omega(x,y)=\sin\left(\pi x\right)-\sin\left(\pi y\right)+\sin\left(\pi (k-x+y)\right)$, $\varphi(x,y)\in \mathcal{S}(\mathbb{R}^2)$ equal to 1 on $B(0,1)$ and compact support on $B(0,2)$, and $\chi\in \mathcal{S}(\mathbb{R})$ is 1 on the unit ball $B(mT,1)$, and $\hat{\chi}$ is supported on a ball of constant radius $C$. Let $e(x)=\exp(2\pi i x)$. For the calculation below, we replace $F$ with a positive Schwartz function that equals to $F$ on $S_3$, and we still denote it as $F$ without affecting any results. Applying Fourier Transform on $\chi$ and Poisson Summation on $\sum_{x,y \in \Z_N} F(x,y)\varphi(x,y)e\left(\tau\Omega(x,y)\right)$, we get:
    \begin{align}
        &\quad\sum_{\substack{(x,y,z)}\in S_3}F(x,y) \leq \sum_{x,y \in \Z_N} F(x,y)\varphi(x,y)\chi\left(T\Omega(x,y)\right) \notag\\
        &= \sum_{x,y \in \Z_N} F(x,y)\varphi(x,y)\int_{-\infty}^{\infty}\hat{\chi}(\tau)e\left(\tau T\Omega(x,y)\right)d\tau \notag\\
        &=T^{-1} \sum_{x,y \in \Z_N} F(x,y)\varphi(x,y)\int_{-\infty}^{\infty}\hat{\chi}\left(\frac{\tau}{T}\right)e\left(\tau\Omega(x,y)\right)d\tau \notag\\
        &=T^{-1}\int_{-\infty}^{\infty}\hat{\chi}\left(\frac{\tau}{T}\right)\left[\sum_{x,y \in \Z_N} F(x,y)\varphi(x,y)e\left(\tau\Omega(x,y)\right)\right]d\tau \notag\\
        &=T^{-1}\int_{-\infty}^{\infty}\hat{\chi}\left(\frac{\tau}{T}\right)\int_{(x,y)\in \mathbb{R}^2} F\left(\frac{x}{N},\frac{y}{N}\right)\varphi\left(\frac{x}{N},\frac{y}{N}\right)e\left(\tau \Omega\left(\frac{x}{N},\frac{y}{N}\right)\right)dxdyd\tau \notag\\
        &+T^{-1}\int_{-\infty}^{\infty}\hat{\chi}\left(\frac{\tau}{T}\right)\sum_{(c,d)\in\mathbb{Z}^2\setminus\{(0,0)\}}\int_{(x,y)\in \mathbb{R}^2}F\left(\frac{x}{N},\frac{y}{N}\right) \varphi\left(\frac{x}{N},\frac{y}{N}\right)\notag\\
        &\qquad\qquad\qquad\qquad\qquad\qquad\qquad\qquad\quad \times e\left(\tau \Omega\left(\frac{x}{N},\frac{y}{N}\right)-cx-dy\right)dxdyd\tau \notag\\
        &=N^2T^{-1}\int_{-\infty}^{\infty}\hat{\chi}\left(\frac{\tau}{T}\right)\int_{(x,y)\in B(0,2)} F(x,y)\varphi(x,y)e\left(\tau \Omega(x,y)\right)dxdyd\tau \notag\\
        &+N^2T^{-1}\int_{-\infty}^{\infty}\hat{\chi}\left(\frac{\tau}{T}\right)\sum_{(c,d)\in\mathbb{Z}^2\setminus\{(0,0)\}}\int_{(x,y)\in B(0,2)} F(x,y)\varphi(x,y)\notag\\
        &\qquad\qquad\qquad\qquad\qquad\qquad\qquad\qquad\quad \times e\left(\tau \Omega(x,y)-cNx-dNy\right)dxdyd\tau \notag\\
        &=I + II \notag
    \end{align}
    For $I$, we can rewrite $I$ as:
    \begin{align}
        I &= N^2T^{-1}\int_{-\infty}^{\infty}\hat{\chi}\left(\frac{\tau}{T}\right)\int_{(x,y)\in B(0,2)} F(x,y)\varphi(x,y)e\left(\tau \Omega(x,y)\right)dxdyd\tau \notag\\
        &= N^2\int_{(x,y)\in B(0,2)} F(x,y)\varphi(x,y)\chi(T\Omega(x,y))dxdy
    \end{align}
    Note that:
    \begin{align}
        \partial_y\Omega(x,y) = \pi[\cos\left(\pi y\right)-\cos\left(\pi (k-x+y)\right)]
    \end{align}
    which means $\partial_y\Omega(x,y)\neq 0$ unless $(x,y) \in \{(x,y): x - k \in 2\mathbb{Z} \text{ or } k-x+2y \in 2\mathbb{Z}\}$. Let
    \begin{align}
    E &:= \{(x,y): x - k \in 2\mathbb{Z} \text{ or } k-x+2y \in 2\mathbb{Z}\},\\
    E_1 &:= \{(x,y)\in B(0,2):|(x,y)-E|\geq T^{-1}\},\\ 
    E_2 &:= \{(x,y)\in B(0,2):|(x,y)-E|< T^{-1}\}.
    \end{align}
    Here, $|(x,y)-E|$ denotes Euclidean distance from $(x,y)$ to the set $E$. Note that due to (4.5) one can solve the equation $\Omega(x,y)=\Omega_0$ for $y$ given $x$ and $\Omega_0$ on $E_1$. Then we can separate $I$ into two terms:
    \begin{align}
        I &= N^2\int_{(x,y)\in E_1}F(x,y)\varphi(x,y)\chi(T\Omega(x,y))dxdy \notag\\
        &\qquad\qquad\qquad\qquad\qquad+ N^2\int_{(x,y)\in E_2}F(x,y)\varphi(x,y)\chi(T\Omega(x,y))dxdy \\
        &= I.1 + I.2
    \end{align}
    Inside the region $E_1$, for any even integer $p$ we have:
    \begin{align}
        |x-k-p|, |k-x+2y-p| &\gtrsim T^{-1}.
    \end{align}
    Applying substitution $y\mapsto\Omega$ and $dy = d\Omega/\partial_y\Omega$ to the region $E_1$ for $I.1$, we have:
    \begin{align}
        |I.1| &= \left|N^2\int_{\mathbb{R}} \left(\int_{\substack{\Omega=s\\(x,y(s,x))\in E_1}}\varphi(x,y(s,x))\frac{F(x,y)}{\partial_y\Omega(x,y(s,x))}dx\right)\chi(Ts)ds\right| \\
        &\lesssim N^2\int_{\mathbb{R}} \left(\int_{\substack{\Omega=s\\(x,y(s,x))\in E_1}}\varphi(x,y(s,x))\left|\frac{\sin(\pi y)\sin(\pi (k-x+y))}{\cos\left(\pi y\right)-\cos\left(\pi (k-x+y)\right)}\right|dx\right)\chi(Ts)ds\label{eq-4.17}
    \end{align}
    Let $\tilde{x} = \frac{x-k}{2}$, and we can formulate the integrand as:
    \begin{align}
        &\left|\frac{\sin(\pi y)\sin(\pi (k-x+y))}{\cos\left(\pi y\right)-\cos\left(\pi (k-x+y)\right)}\right| = \left|\frac{\sin(\pi y)\sin(\pi (y-2\tilde{x}))}{2\sin(\pi \tilde{x})\sin(\pi(y-\tilde{x}))}\right|\notag\\
        &= \left|\frac{1}{2}\cdot\frac{\sin(\pi y)\sin(\pi (y-2\tilde{x}))}{\sin(\pi \tilde{x})-\sin(\pi(y-\tilde{x}))}\cdot\left(\frac{1}
        {\sin\left(\pi\tilde{x}\right)}-
        \frac{1}
        {\sin\left(\pi \left(y-\tilde{x}\right)\right)}\right)\right| \notag\\
        &= \left|\frac{1}{4}\cdot\frac{\sin(\pi y)\sin(\pi (y-2\tilde{x}))}{\cos(\frac{\pi y}{2})\sin(\pi(\frac{y}{2}-\tilde{x}))}\cdot\left(\frac{1}
        {\sin\left(\pi\tilde{x}\right)}-
        \frac{1}
        {\sin\left(\pi \left(y-\tilde{x}\right)\right)}\right)\right| \notag\\
        &= \left|\sin\left(\frac{\pi y}{2}\right)\cos\left(\pi\left(\frac{y}{2}-\tilde{x}\right)\right)\left(\frac{1}
        {\sin\left(\pi\tilde{x}\right)}-
        \frac{1}
        {\sin\left(\pi \left(y-\tilde{x}\right)\right)}\right)\right| \notag\\
        &\lesssim \left|\frac{1}
        {\sin\left(\pi\tilde{x}\right)}\right|+
        \left|\frac{1}
        {\sin\left(\pi \left(y-\tilde{x}\right)\right)}\right|. \label{eq-4.18}
    \end{align}
    Then we can estimate $|I.1|$ using (\ref{eq-4.18}):
    \begin{align}
        |I.1|&\lesssim 
        N^2\int_{\mathbb{R}} \left(\int_{\substack{\Omega=s\\(x(\tilde{x}),y(s,x))\in E_1}}
        \left|\frac{1}
        {\sin\left(\pi\tilde{x}\right)}\right|+
        \left|\frac{1}
        {\sin\left(\pi \left(y-\tilde{x}\right)\right)}\right|d\tilde{x}\right)\chi(Ts)ds\\
        &\lesssim N^2\int_{\mathbb{R}}\int_{T^{-1}}^{1}\frac{1}{\tilde{x}}\;d\tilde{x}\;\chi(Ts)\;ds = N^2T^{-1}\log T\int_{\mathbb{R}}\chi(t)dt\lesssim N^2T^{-1}\log T.
    \end{align}
    Note that the inner integral can be separated into several parts, if $y(\Omega=s)$ has more than one value, since the map $y\mapsto \Omega$ can have finitely many solutions for given fixed $k$ and $x$ in $E_1$. And inside $E_1$, we have $|\sin\left(\tilde{x}\right)|, |\sin(\pi \left(y-\tilde{x}\right)|\gtrsim T^{-1}$, which lead to counting of $\log T$. 
    On the other hand, for $I.2$, we know that $|E_2|\lesssim T^{-1}$, and then $|I.2| \lesssim N^2T^{-1}$.
    For $II$, the phase function $\Phi(x,y)=\tau \Omega(x,y)-cNx-dNy$ satisfies:
    \begin{align}
        |\triangledown_{(x,y)}\Phi(x,y)|=|\tau\triangledown_{(x,y)}\Omega(x,y)-N(c,d)|\gtrsim N|c|,
    \end{align}
    for $c \neq 0$ and $\tau\leq CT\lesssim N^{1-\epsilon}$. Then we can integrate by parts many times as:
    \begin{align}
        \triangledown_{(x,y)} e\left(\tau \Omega(x,y)-cNx-dNy\right)=\frac{e\left(\tau \Omega(x,y)-cNx-dNy\right)}{2\pi i \triangledown_{(x,y)}\Phi(x,y)},
    \end{align}
    to show that $|II|\ll N^2T^{-1}$. 
    Hence $\sum_{\substack{(x,y,z)}\in S_3}F(x,y) \lesssim N^2T^{-1}\log T$. 
    \end{proof}
    \begin{corollary} {(Convergence of Iterates)}\label{cor-iterates}
    For fixed $k\in \mathbb{Z}_N \cap (-1,2) \setminus \{0,1\}$, let $F \in \mathcal{S}(\R^2)$, $\chi \in \mathcal{S}(\R)$ with $\hat{\chi}$ supported in a ball of constant radius $C>0$. Then for all $\delta>0$ sufficiently small we have 
    \begin{align}
     \sum_{\substack{x, y\in \Z_N \cap (0,1)\\k-x+y \in (0,1)}} F(x, y)\chi(T\Omega(x, y)) =& N^2\int_{\substack{x, y\in (0,1)\\k-x+y \in (0,1)}} F(x, y)\chi(T\Omega(x, y))d x d y+ O(N^{2 - \delta}T^{-1}), \label{eq-iterate-conv}
    \end{align}
    where 
    \begin{align*}
        \Omega(x, y) &= \sin\left(\pi x\right)-\sin\left(\pi y\right)+\sin\left(\pi (k-x+y)\right)-|\sin\left(\pi k\right)|.
    \end{align*}
    \end{corollary}
    \begin{proof}
    We choose compactly supported $\varphi\in \mathcal{S}(\mathbb{R}^2)$ to be 1 on $S=\big\{(x,y)\in (0,1)^2: k-x+y \in (0,1)\big\}$ and nonzero on $S_1$ where $|S_1\setminus S|\lesssim N^{-10\delta}T^{-1}$. Then using the result in Proposition \ref{vector_counting}, we get:
    \begin{align}
        &\sum_{\substack{x, y\in \Z_N \cap (0,1)\\k-x+y \in (0,1)}} F(x, y)\chi(T\Omega(x, y))=\sum_{x, y\in \Z_N}F(x, y)\varphi(x,y)\chi(T\Omega(x, y))+O(N^{2 - \delta}T^{-1}) \label{4.2.0}
    \end{align}
    And similarly following the calculation in Proposition \ref{vector_counting}, we can write the first term in \eqref{4.2.0} as:
    \begin{align}
        &\sum_{x, y\in \Z_N}F(x, y)\varphi(x,y)\chi(T\Omega(x, y))=N^2\int_{(x,y)\in \mathbb{R}^2} F(x,y)\varphi(x,y)\chi(T\Omega(x,y))dxdy \notag\\
        &+N^2T^{-1}\int_{-\infty}^{\infty}\hat{\chi}\left(\frac{\tau}{T}\right)\sum_{(c,d)\in\mathbb{Z}^2\setminus\{(0,0)\}}\int_{\mathbb{R}^2} F(x,y)\varphi(x,y)e\left(\tau \Omega(x,y)-cNx-dNy\right)dxdyd\tau. \label{4.2.2} \\
        &= N^2\int_{\substack{x, y\in (0,1)\\k-x+y \in (0,1)}} F(x, y)\chi(T\Omega(x, y))d x d y+ O(N^{2 - \delta}T^{-1}). \notag
    \end{align}
    Again we can use $|S_1\setminus S|\lesssim N^{-10\delta}T^{-1}$ to show that:
    \begin{align*}
        N^2\int_{(x,y)\in \mathbb{R}^2} F(x,y)|\varphi(x,y)-\boldsymbol{1}_S|\chi(T\Omega(x,y))dxdy \lesssim  O(N^{2 - \delta}T^{-1}),
    \end{align*}
    and apply integration by parts sufficiently many times to show that \eqref{4.2.2}$\lesssim O(N^{2 - \delta}T^{-1})$, from which \eqref{eq-iterate-conv} follows.
    \end{proof}
    
    \begin{proposition}[Two Vector Counting] \label{vector_counting-2}
    Let $1<T\leq N^{1-\epsilon}$ for $\epsilon \ll 1$. Then uniformly in $k\in \mathbb{Z}_N \cap (-1,2)$ and $m\in \mathbb{R}$, the sets:
    \begin{align}
        S_2^{\pm} &= \Big\{(x,y)\in \mathbb{Z}_N^2\cap (0,1)^2: x\pm y= k,\notag\\
        &\left|\sin\left(\pi x\right)\pm\sin\left(\pi y\right)-m\right| \leq T^{-1}\Big\}.
    \end{align}
    satisfy the bounds:
    \begin{align}
        \sum_{\substack{(x,y)}\in S_2^{+}}G(x) &\lesssim NT^{-\frac{1}{2}},\\
        \sum_{\substack{(x,y)}\in S_2^{-}}G(x) &\lesssim \left\{ \begin{array}{rcl}
        NT^{-\frac{1}{2}} & \text{if } |k| \gtrsim T^{-\frac{1}{2}},\\
        N & \text{otherwise},
        \end{array}\right.
    \end{align}
    where 
    \begin{align}
        G(x) &= |\sin(\pi x)\sin(\pi (k-x))|.
    \end{align}
    \end{proposition}
\begin{proof}
    Similarly as in Proposition \ref{vector_counting}, for $\sum_{\substack{(x,y)}\in S_2^{\pm}}G(x)$, we define: \[\Omega_2(x)=\sin\left(\pi x\right)+\sin\left(\pi (k-x)\right),\] 
    $\chi\in \mathcal{S}(\mathbb{R})$ is 1 on the unit ball $B(mT,1)$, and $\hat{\chi}$ is supported on a ball of radius constant $C$. Let $e(x)=\exp(2\pi i x)$, and $\psi(x)\in \mathcal{S}(\mathbb{R})$ being 1 on $(-1,1)$ and compact support on $[-2,2]$. We replace $G$ with a positive Schwartz function that equals to $G$ on $S_2^{\pm}$. Applying Fourier Transform on $\chi$ and Poisson Summation again on $\sum_{x} G(x)\psi(x)e\left(\tau\Omega_2(x)\right)$, we get:
    \begin{align}
        \sum_{\substack{(x,y)}\in S_2^{\pm}}G(x)  &\leq \sum_{x} G(x)\psi(x)\chi\left(T\Omega_2(x)\right) \notag\\
        &=NT^{-1}\int_{-\infty}^{\infty}\hat{\chi}\left(\frac{\tau}{T}\right)\int_{x\in (-2,2)} G(x)\psi(x)e\left(\tau \Omega_2(x)\right)dx d\tau \notag\\
        &+NT^{-1}\int_{-\infty}^{\infty}\hat{\chi}\left(\frac{\tau}{T}\right)\sum_{c\in\mathbb{Z}\setminus\{0\}}\int_{x\in (-2,2)} G(x)\psi(x)e\left(\tau \Omega_2(x)-cNx\right)dx d\tau \notag\\
        &=III + IV \notag
    \end{align}
    Note that:
    \begin{align}
        \Omega_2'(x) = \pi[\cos\left(\pi x\right)-\cos\left(\pi (k-x)\right)] = 2\pi\sin\left(\frac{\pi k}{2}\right)\sin\left(\pi\left(x-\frac{k}{2}\right)\right)
    \end{align}
    which means $\Omega_2'(x)\neq 0$ unless $x \in \{x: 2x - k \in 2\mathbb{Z}\}$ or $k=0$. If $k=0$, the trivial countings $\sum_{\substack{(x,y)}\in S_2^{+}}G(x)=0$ and $\sum_{\substack{(x,y)}\in S_2^{-}}G(x)\leq N$ are true. If $|k-2|\lesssim T^{-\frac{1}{2}}$ or $|k|\lesssim T^{-\frac{1}{2}}$, then $\sum_{\substack{(x,y)}\in S_2^{+}}G(x)\lesssim NT^{-\frac{1}{2}}$. If $k\neq0$ and $|k-2|\gtrsim T^{-\frac{1}{2}}$, we let
    \begin{align}
    E' &= \{x: 2x - k \in 2\mathbb{Z}\},\\
    E'_1 &= \{x\in (-2,2):|x-E'|\geq T^{-\frac{1}{2}}\},\\ 
    E'_2 &= \{x\in (-2,2):|x-E'|< T^{-\frac{1}{2}}\}.
    \end{align}
    Then similarly we have:
    \begin{align}
        III &= N\int_{x\in E'_1}\psi(x)\chi(T\Omega_2(x))dx + N\int_{x\in E'_2}\psi(x,y)\chi(T\Omega_2(x))dx \\
        &= III.1 + III.2
    \end{align}
    Applying substitution $dx = d\Omega_2/\Omega_2'$ to the region $E'_1$ for $III.1$, we have:
    \begin{align}
        |III.1| &= \left|N\int_{\mathbb{R}} \left(\int_{\substack{\Omega=s\\x\in E'_1}}\psi(x(s))\frac{G(x)}{\Omega_2'(x(s))}dx\right)\chi(Ts)ds\right| \\
        &\lesssim N\int_{x(s)\in E'_1}\psi(x(s))\left|\frac{\sin(\pi x(s))\sin(\pi (k-x(s)))}{\sin\left(\frac{\pi k}{2}\right)\sin\left(\pi\left(x(s)-\frac{k}{2}\right)\right)}\right|\chi(Ts)ds.
    \end{align}
    We estimate the integrand again:
    \begin{align}
        &\left|\frac{\sin(\pi x)\sin(\pi (k-x))}{\sin\left(\frac{\pi k}{2}\right)\sin\left(\pi\left(x-\frac{k}{2}\right)\right)}\right| = \left|\frac{\sin(\pi x)\sin(\pi (k-x))}{\sin\left(\frac{\pi k}{2}\right)-\sin\left(\pi\left(x-\frac{k}{2}\right)\right)}\cdot\left(\frac{1}{\sin\left(\frac{\pi k}{2}\right)}-\frac{1}{\sin\left(\pi\left(x-\frac{k}{2}\right)\right)}\right)\right|\notag\\
        &=\left|\frac{\sin(\pi x)\sin(\pi (k-x))}{2\cos\left(\frac{\pi x}{2}\right)\sin\left(\pi\left(\frac{k-x}{2}\right)\right)}\cdot\left(\frac{1}{\sin\left(\frac{\pi k}{2}\right)}-\frac{1}{\sin\left(\pi\left(x-\frac{k}{2}\right)\right)}\right)\right|\notag\\
        &=\left|2\sin\left(\frac{\pi x}{2}\right)\cos\left(\pi\left(\frac{k-x}{2}\right)\right)\cdot\left(\frac{1}{\sin\left(\frac{\pi k}{2}\right)}-\frac{1}{\sin\left(\pi\left(x-\frac{k}{2}\right)\right)}\right)\right|\notag\\
        &\lesssim \left|\frac{1}
        {\sin\left(\frac{\pi k}{2}\right)}\right|+
        \left|\frac{1}
        {\sin\left(\pi\left(x-\frac{k}{2}\right)\right)}\right|.
    \end{align}
    Then similarly as in Proposition \ref{vector_counting}:
    \begin{align}
        |III.1|&\lesssim 
        N\int_{x(s)\in E'_1}\psi(x(s))\chi(Ts)\left(\left|\frac{1}
        {\sin\left(\frac{\pi k}{2}\right)}\right|+
        \left|\frac{1}
        {\sin\left(\pi\left(x-\frac{k}{2}\right)\right)}\right|\right)ds\notag\\
        &\lesssim \max(NT^{-\frac{1}{2}},NT^{-1}|k|^{-1}).
    \end{align}
    On the other hand, for $III.2$, we know that $|E'_2|\lesssim T^{-\frac{1}{2}}$, and then $|III.2| \lesssim NT^{-\frac{1}{2}}$. For $IV$, the phase function $\Psi(x)=\tau \Omega_2(x)-cNx$ satisfies:
    \begin{align}
        |\Psi'(x)|=|\tau\Omega_2'(x)-Nc|\gtrsim N|c|,
    \end{align}
    for $c \neq 0$ and $\tau\leq CT\lesssim N^{1-\epsilon}$. Then we can integrate by parts many times as:
    \begin{align}
        \frac{d}{dx}e\left(\tau \Omega_2(x)-cNx\right)=\frac{e\left(\tau \Omega_2(x)-cNx\right)}{2\pi i \Psi(x)},
    \end{align}
    to show that $|IV|\ll NT^{-\frac{1}{2}}$. 
    Thus we have: 
    \begin{align}
    \sum_{\substack{(x,y)}\in S_2^{+}}G(x) &\lesssim NT^{-\frac{1}{2}},\\
    \sum_{\substack{(x,y)}\in S_2^{-}}G(x) &\lesssim \left\{ \begin{array}{rcl} NT^{-\frac{1}{2}} & &\text{if } |k| \gtrsim T^{-\frac{1}{2}},\\ NT^{-1}|k|^{-1} & &\text{if } T^{-\frac{1}{2}} \gtrsim |k| \gtrsim T^{-1}, \\N & &\text{otherwise,}
\end{array}\right.
    \end{align}
where we take the trivial counting $N$ for $(x,y)\in S_2^{-}$ and $|k| \lesssim T^{-1}$.
\end{proof}

\section{Molecules and Couples} 
\label{section-couples}
In this section, we introduce the irregular chains, the splicing process to remove the small gap irregular chains, the preprocessing steps we need to add to the algorithm from \cite{ODW}, and finally prove the bounds on couples using a bootstrap argument.
\subsection{Double bonds and chains}\label{chain}
In this and next sections we repeat the definitions of double bonds, chains, unit twist, congruent couples, and splicing from \cite{WKE, ODW}.
\begin{definition}
Let $\Q$ be an enhanced couple such that $\mathbb{M}(\Q)$ has two atoms, $v_1$ and $v_2$, connected by precisely two edges oriented in opposite directions (resp. same direction). Denote by $\mathfrak{n}_j = \mathfrak{n}(v_j)$ for $j=1,2$. Then, up to symmetry, exactly one of the following two scenarios arises by Proposition 5.1 in \cite{ODW}:

\begin{enumerate}[label=(\roman*)]
\item \textbf{Cancellation (CL) Double Bond:} There exists a child $\mathfrak{n}_{12}$ of $\mathfrak{n}_1$ which is paired with a child $\mathfrak{n}_{21}$ of $\mathfrak{n}_2$. 
Moreover, $\mathfrak{n}_2$ is itself a child of $\mathfrak{n}_1$ such that $\mathfrak{n}_{12}$ and $\mathfrak{n}_2$ have opposite signs (resp. same sign). 
All other leaf-children of $\mathfrak{n}_1$ and $\mathfrak{n}_2$ remain unpaired. 
In the corresponding molecular structure, this configuration represents a double bond composed of one LP bond and one PC bond.

\item \textbf{Connectivity (CN) Double Bond:} There are children $\mathfrak{n}_{11}, \mathfrak{n}_{12}$ of $\mathfrak{n}_1$ with opposite signs (resp. same sign), as well as children $\mathfrak{n}_{21}, \mathfrak{n}_{22}$ of $\mathfrak{n}_2$ having opposite signs (resp. same sign). These four children pair off according to their signs, leaving any remaining leaf-children of $\mathfrak{n}_1$ and $\mathfrak{n}_2$ unpaired. In addition, neither $\mathfrak{n}_1$ nor $\mathfrak{n}_2$ is a child of the other in this scenario. In the molecular structure, this corresponds to a double bond in which both bonds are LP.
\end{enumerate}

\end{definition}

\begin{definition}[Chains]\label{def:chain-molecule}
Let $\Q$ be an enhanced couple, and let $\M(\Q)$ be its corresponding molecule. A \emph{chain} in $\M(\Q)$ is a sequence of atoms (nodes) $(v_0, \ldots, v_q)$ such that each consecutive pair $v_i,v_{i+1}$ (for $0 \leq i \leq q-1$) is connected by either a CL or CN double bond. We define the following special types of chains:
\begin{enumerate}[label=(\roman*)]
    \item A \emph{CL chain} is a chain whose double bonds are all CL.
    \item A \emph{negative chain} is a chain whose bonds are oriented with opposite directions. 
    \item An \emph{irregular chain} is a negative chain that is also CL.
    \item A \emph{maximal chain} is a chain for which $v_0$ and $v_q$ do not connect (via CL or CN double bonds) to any atoms outside the chain.
    \item A \emph{hyperchain} is a chain where $v_0$ and $v_q$ are also joined by a single bond.
    \item A \emph{pseudo-hyperchain} is a chain where $v_0$ and $v_q$ each connect (via single bonds) to another common atom $v \notin \{v_0,\dots,v_q\}$.
    \item A \textit{wide ladder} is a collection of chains ${(v_0^{(1)}, \ldots, v_{q^{(1)}}^{(1)}), \ldots, (v_0^{(m)}, \ldots, v_{q^{(m)}}^{(m)})}$, where each chain is referred to as a \textit{rung}, such that or each $1 \leq i \leq m - 1$, there is a bond connecting the vertices $v_0^{(i)}$ and $v_0^{(i+1)}$, as well as a bond connecting the vertices $v_{q^{(i)}}^{(i)}$ and $v_{q^{(i+1)}}^{(i+1)}$.
\end{enumerate}
\end{definition}

\begin{definition}[Irregular Chains and Splicing]\label{def:chain-couple}
A sequence of nodes $(\mathfrak{n}_0, \ldots, \mathfrak{n}_q)$ in the couple $\Q$ is called a \emph{CL chain} if, for each $0 \leq j \leq q-1$, 
\begin{enumerate}[label=(\roman*)]
    \item \(\mathfrak{n}_{j+1}\) is a child of \(\mathfrak{n}_{j}\),
    \item \(\mathfrak{n}_{j}\) has a child \(\mathfrak{m}_{j+1}\) paired with a child \(\mathfrak{p}_{j+1}\) of \(\mathfrak{n}_{j+1}\),
    \item \(\mathfrak n_0^1\) to be the remaining child of \(\mathfrak n_0\),
    \item \(\mathfrak n_0^{2}\) and \(\mathfrak n_0^{3}\) to be the remaining children of \(\mathfrak n_q\).
\end{enumerate}
If, in addition, each pair \((\mathfrak n_i,\mathfrak m_i)\) has opposite sign for \(1 \le i \le q\), then the chain is called \emph{irregular}. By the labeling convention relating couples to molecules, a sequence \((v_0,\dots,v_q)\) in \(\M(\Q)\) is a CL chain (resp. irregular chain) if and only if \((\mathfrak{n}_0,\dots,\mathfrak{n}_q)\) is.

\emph{Splicing} a CL chain $(\mathfrak{n}_0, \ldots, \mathfrak{n}_q)$ at $\mathfrak{n}_1, \ldots, \mathfrak{n}_q$ means removing those nodes and their leaf-children $\mathfrak n_i, \mathfrak m_i, \mathfrak p_i$ for  $(i = 1, \ldots, q)$ from $\Q$ and redefining the children of $\mathfrak{n}_0$ accordingly. The children of $\mathfrak n_0$ become $\mathfrak n_0^1$, $\mathfrak n_0^{2}$, and $\mathfrak n_0^{3}$, with their position determined by sign or by their relative position as children of $\mathfrak n_q$.  The resulting couple is denoted $\Q^{\mathrm{sp}}$, with its corresponding molecule $\mathbb{M}^{\mathrm{sp}}$.
\end{definition}

\begin{definition}[Gaps in Negative Chains]\label{def:gap}
Suppose $\M(\Q)$ has a double bond whose edges point in opposite directions. For a chosen \emph{decoration} of $\M(\Q)$, let $k$ and $\ell$ be the decoration values on these two edges, and define the \emph{gap} $h := k - \ell$. The double bond is called a \emph{large gap (LG) double bond} if $|h| \gtrsim T^{-\frac{1}{2}}$ and a \emph{small gap (SG) double bond} if $|h| \lesssim T^{-\frac{1}{2}}$. 

By \cite[Proposition~8.3]{WKE}, all double bonds in a given chain share the same gap magnitude. Specifically, in an irregular chain $(\mathfrak{n}_0, \ldots, \mathfrak{n}_q)$, the gap can also be tracked via 
\[
   h = k_{\mathfrak{n}_0} - k_{\mathfrak{n}_0^1}
   = k_{\mathfrak{n}_j} - k_{\mathfrak{m}_j}
   \quad(1 \leq j \leq q).
\]
\end{definition}
\subsection{Cancellation of irregular chains}\label{irregular}
\begin{definition}[Twist-Admissibility and Unit Twist]\label{def:twist}
Let $\Q$ be an enhanced couple with node set $\mathcal{N}$.
\begin{enumerate}[label=(\roman*)]
    \item We call a node $\mathfrak{n}_2 \in \mathcal{N}$ \emph{twist-admissible} if it has a parent $\mathfrak{n}_1$ such that the atoms $v_1, v_2$ corresponding to $\mathfrak{n}_1,\mathfrak{n}_2$ in the molecule $\M(\Q)$ are joined by a CL double bond and further that these two bonds have opposite orientations (i.e., the double bond is negative). Denote by $\mathfrak{n}_{12}$ and $\mathfrak{n}_{21}$ the paired children of $\mathfrak{n}_1$ and $\mathfrak{n}_2$, respectively.
    \item A \emph{unit twist} of $\Q$ at a twist-admissible node $\mathfrak{n}_2$ is performed by:
    \begin{enumerate}[label=(\alph*)]
        \item Swapping the node $\mathfrak{n}_{12}$ with $\mathfrak{n}_2$ itself,
        \item Swapping the two children (and subtrees) of $\mathfrak{n}_2$ that are \emph{not} $\mathfrak{n}_{21}$,
        \item Leaving all other parent-child relationships (and their signs) unchanged.
    \end{enumerate}

    \item Suppose $\Q$ is a \emph{decorated} couple, where each node or leaf $\mathfrak{m}$ has a decoration $k_{\mathfrak{m}}$. A unit twist at $\mathfrak{n}_2$ induces a corresponding decoration on the new couple $\Q'$ by:
    \[
       k_{\tilde{\mathfrak{m}}} = k_{\mathfrak{m}}
       \quad \text{for all } \mathfrak{m}
       \text{ except } \mathfrak{n}_2,\mathfrak{n}_{12},\mathfrak{n}_{21},
    \]
    and setting
    \[
       k_{\tilde{\mathfrak{n}}_2} = k_{\mathfrak{n}_{12}} 
       = k_{\mathfrak{n}_{21}}, 
       \quad
       k_{\tilde{\mathfrak{n}}_{12}} = k_{\tilde{\mathfrak{n}}_{21}}
       = k_{\mathfrak{n}_2}.
    \]
    If one splices either $\Q$ or $\Q'$ at $\mathfrak{n}_2$, the resulting couple (and its molecule) agrees in structure and decoration. 
\end{enumerate}
\end{definition}

\begin{definition}[Congruent Couples]\label{def:congruence}
Two enhanced couples $\Q_1$ and $\Q_2$ are \emph{congruent}, written $\Q_1 \sim \Q_2$, if one can obtain $\Q_2$ from $\Q_1$ (and vice versa) by performing any finite sequence of unit twists. If $\Q_1$ and $\Q_2$ are decorated, these decorations correspond one-to-one under the same twists. Since unit twists commute, specifying a set of twist-admissible nodes $\mathcal{M}\subset\mathcal{N}$ determines a family of congruent couples. Denote by
\[
   \textgoth{Q}_{\mathcal{M}}
   \;=\;
   \bigl\{
     \Q' \mid \Q' \sim \Q \ \text{via some subset of twists in} \ \mathcal{M}
   \bigr\}.
\]
If $\mathcal{K_{Q}}(t,s,k)$ denotes a function or weight associated to $\Q$, define
\[
   \K_{\textgoth{Q}_{\mathcal M}}(t,s,k)
   \;=\;
   \sum_{\Q' \in \textgoth{Q}_{\mathcal M}} \mathcal{K}_{\Q'}(t,s,k).
\]
\end{definition}

\begin{lemma}\label{lem-splice}
For an enhanced couple $\mathcal Q$ and SG irregular chain $(\mathfrak n_0, \ldots, \mathfrak n_{q})$ of length $q$, 
\begin{align}
&\K_{\textgoth{Q}_\mathcal M} (t,s,k) \notag\\
&= \left( \frac{\beta T}{N}\right)^{n_{\mathrm{sp}}} \zeta(\mathcal Q^{\mathrm{sp}}) \sum_{\mathscr E^{\mathrm{sp}}} \tilde \epsilon_{\mathscr E^{\mathrm{sp}}} \int _0^1\int_{\mathcal{E}^{\mathrm{sp}}_{\tau}} P_q(\tau, t_{\mathfrak n_0}, k[\mathcal N^{\mathrm{sp}}]) \left( \prod_{\mathfrak n \in \mathcal N^{\mathrm{sp}}}e^{\zeta_{\mathfrak{n}}i(\Omega_{\mathfrak{n}}Tt_{\mathfrak{n}}+\widetilde{\Omega}_{\mathfrak{n}}(t_{\mathfrak{n}}))} d t_{\mathfrak{n}}\right)d \tau  \notag\\
& \hspace{6cm}\times \prod^+_{\mathfrak l \in \mathcal L^{\mathrm{sp}}} n_{\mathrm{in}}(k_{\mathfrak l})\E \left[\mathcal{B}_{\mathcal{T}^+}(\eta_{\mathcal{T}^+}(\varrho))\mathcal{B}^*_{\mathcal{T}^-}(\eta_{\mathcal{T}^-}(\varrho))\right],\label{kqm}
\end{align}
where for some constant $C_0$, $\widetilde{\Omega}_{\mathfrak{n}}$ and $P_q$ satisfies 
\begin{align}
\widetilde{\Omega}_{\mathfrak{n}}(s) &= \Omega_{\mathfrak{n}}\left[(C_0\beta Ts+{A}_{\geq}(s)\right], \\
     \left|\dot{A}_{\geq}(s)\right| & \lesssim  \beta T,\\
\sup_{|k_{\mathfrak n_0} - k_{\mathfrak n^1_0}| \leq T^{-1/2}} \left|\left|P_q(\tau,t_0,  k[\mathcal N^{\mathrm{sp}}])\right|\right|_{\ell^\infty} &\lesssim \left(\beta T^{1/2}\right)^q, \label{eq-pestimate}\\
\sup_{|k_{\mathfrak n_0} - k_{\mathfrak n^1_0}| \leq T^{-1/2}} \left|\left|\frac{d}{dt_0}P_q(\tau,t_0,  k[\mathcal N^{\mathrm{sp}}])\right|\right|_{\ell^\infty} &\lesssim \left(\beta T^{1/2}\right)^q,\label{eq-dpestimate}
\end{align} and 
\begin{equation}\label{eq-sig-setE}
\mathcal E^{\mathrm{sp}}_{\tau} = \mathcal E^{\mathrm{sp}} \cap \{ t_{\mathfrak n_0} \geq t_{\mathfrak n_0^{2}} + \tau, t_{\mathfrak n_0} \geq t_{\mathfrak n_0^{3}} + \tau \}.
\end{equation}
In the above, we let $\mathcal{Q}^{\mathrm{sp}}$ be the couple obtained from $\mathcal{Q}$ by splicing out its irregular chain, and its corresponding decorations $\mathscr E^{\mathrm{sp}}$, domain of integration $\mathcal E^{\mathrm{sp}}$ and its nodes $\mathcal N^{\mathrm{sp}}$ and leaves $\mathcal L^{\mathrm{sp}}$. The notation $k[\mathcal{N}^{\mathrm{sp}}]$ denotes the decorations on the nodes of this spliced couple, and we define $\mathcal{M} = \{\mathfrak{n}_1, \ldots, \mathfrak{n}_q\}$. Furthermore, $\tilde \epsilon_{\mathscr E^{\mathrm{sp}}}$ permits degeneracies at $\mathfrak{n}_0$. 
\end{lemma}
\begin{proof}
Without loss of generality, we assume that $\Q$ is the element of $\textgoth{Q}_{\mathcal M}$ where all $\zeta_{\mathfrak n_j}$ $(j = 0, \ldots, q)$ have the same sign. Following the notation in Definition \ref{def:chain-couple}, we set $k_{\mathfrak n_i} = k_i$, $k_{\mathfrak p_i} = k_{\mathfrak m_i} = \ell_i$, and $t_i = t_{\mathfrak n_i} (0 \leq i \leq q)$. In addition, we let $\ell_0=k_{\mathfrak n_0^{1}}, k_{q+1} = k_{\mathfrak n_0^{2}},$ and $\ell_{q+1} = k_{\mathfrak n_0^{3}}$, and $t_{q+1} = \text{max}(t_{\mathfrak n_0^{2}}, t_{\mathfrak n_0^{3}})$. We keep all $k_{\mathfrak n}$ and $t_{\mathfrak n}$ for $\mathfrak n \notin \{\mathfrak n_i, \mathfrak p_i, \mathfrak m_i\} (1 \leq i \leq q)$ fixed so that the small gap is also fixed by $k_0, \ell_0, k_{q+1}, \ell_{q+1}$: $$h:=k_0 - \ell_0 = k_1 - \ell_1 = \ldots = k_q - \ell_q = k_{q+1} - \ell_{q+1}.$$
As before, we define the frequency mismatch as:
    \begin{align*}
    \Omega_j &= \omega_{k_{j+1}} -  \omega_{\ell_{j+1}} +\omega_{\ell_{j}} - \omega_{k_{j}}\\
    &= 2(\sin(\pi k_{j+1}) - \sin(\pi \ell_{j+1}) + \sin(\pi \ell_{j}) - \sin(\pi k_{j}))\\
    &= 4\sin\left(\frac{\pi h}{2}\right)\left[\cos\left(\frac{\pi (k_{j+1}+\ell_{j+1})}{2}\right)-\cos\left(\frac{\pi (k_{j}+\ell_{j})}{2}\right)\right].
    \end{align*}
    This sinusoidal dispersion relation gives: $|\Omega_j|\lesssim h$ for any $j=0,\dots,q$. Summing over the chain, we get:
    \begin{align*}
    \Omega = \sum_{j = 0}^q \Omega_{j} = 2(\sin(\pi k_{q+1}) - \sin(\pi \ell_{q+1}) + \sin(\pi \ell_{0}) - \sin(\pi k_{0})).
    \end{align*}
    For each $\mathcal Q' \in \textgoth{Q}_{\mathcal M}$, the difference in $\mathcal K_{\mathcal Q'}$ compared to $\mathcal K_{\mathcal Q}$ can be identified as:
    \begin{align*}
    & \left( \frac{\beta T}{N}\right)^q\sum\limits_{\substack{k_i, \ell_i \in \Z_N \cap (0,1) \\ k_{i} - \ell_{i}=h \\ i = 1, \ldots, q}}\; \int\limits_{t_{q+1} < t_q < \ldots < t_1 < t_0} \left(\prod_{j = 1}^q i\zeta_{\mathfrak n_j}\right) \left(\prod_{j = 0}^q\epsilon_{k_{j+1}\ell_{j+1}m_{j}}\right)  \nonumber \\
    & \hspace{4.5cm}\times \left( \prod_{j = 0}^q e^{\zeta_{\mathfrak n_j}i( \Omega_jT t_j+ \widetilde{\Omega}_{j}(t_{j}))}\right)\left( \prod_{j = 1}^q n_{\mathrm{in}}(m_j)\right) d t_1 \ldots d t_q,
    \end{align*}
    where $m_0=\ell_0$, and $m_j = \ell_j$ or $m_j = k_j$ for congruent couples with unit twist performed at $\mathfrak n_j$. A unit twist at node $\mathfrak{n}_j$ reverses the sign of $\zeta_{\mathfrak{n}_j}$ but leaves the product $\zeta_{\mathfrak{n}_j}\Omega_j$ unchanged. For congruent couples, the enhanced couple structure will not change which leaves the part $\E \left[\mathcal{B}_{\mathcal{T}^+}(\eta_{\mathcal{T}^+}(\varrho))\mathcal{B}^*_{\mathcal{T}^-}(\eta_{\mathcal{T}^-}(\varrho))\right]$ also unchanged. Hence, when we sum over all couples $\mathcal{Q}' \in \textgoth{Q}_{\mathcal{M}}$, we are effectively summing over every possible choice of $\zeta_{\mathfrak{n}_j}$ for $j=1,\ldots,q$. Note that for our FPUT model, in addition to the linear frequency shift at $t_0$, we also need to extract the nonlinear frequency shift at $t_0$ and the new coefficient $\epsilon_{k_{q+1}\ell_{q+1}\ell_{0}}$ after splicing. The resulting expressions are derived as follows:
    \begin{align*}
    & \hspace{.4cm}\left( \frac{\beta T}{N}\right)^q\sum\limits_{\substack{\substack{k_i \in \Z_N \cap (0,1) \\ i = 1, \ldots, q}}} \hspace{.2cm}\int\limits_{t_{q+1} < t_q < \ldots < t_1 < t_0}\left(\prod_{j = 0}^q\epsilon_{k_{j+1}\ell_{j+1}m_{j}}\right) \left( \prod_{j = 0}^q e^{i( \Omega_jT t_j+ \widetilde{\Omega}_{j}(t_{j}))}\right)\\
    &\hspace{8.5cm}\times \left( \prod_{j = 1}^q n_{\mathrm{in}}(k_j - h)-n_{\mathrm{in}}(k_j)\right) d t_1 \ldots d t_q\\
    &= \epsilon_{k_{q+1}\ell_{q+1}\ell_{0}}\cdot e^{i( \Omega T t_0+ \widetilde{\Omega}(t_{0}))} \sum\limits_{\substack{\substack{k_i \in \Z_N \cap (0,1) \\ i = 1, \ldots, q}}} \hspace{.2cm}\int\limits_{t_{q+1} < t_q < \ldots < t_1 < t_0} \left( \frac{\beta T}{N}\right)^q \left(\prod_{j = 1}^q\epsilon_{k_{j}\ell_{j}}\right) \\
    &\hspace{3cm}\times  \left( \prod_{j = 1}^q e^{i( \Omega_jT(t_j - t_0)+ \widetilde{\Omega}_{j}(t_{j})-\widetilde{\Omega}_{j}(t_{0}))}\right)\left( \prod_{j = 1}^q n_{\mathrm{in}}(k_j - h)-n_{\mathrm{in}}(k_j)\right) d t_1 \ldots d t_q\\
    &= \epsilon_{k_{q+1}\ell_{q+1}\ell_{0}}\cdot e^{i( \Omega T t_0+ \widetilde{\Omega}(t_{0}))} \int\limits_{0 < s_1 < \ldots < s_q < t_0 - t_{q+1}}  \left( \frac{\beta T}{N}\right)^q  \sum\limits_{\substack{\substack{k_i \in \Z_N \cap (0,1) \\ i = 1, \ldots, q}}} \left(\prod_{j = 1}^q\epsilon_{k_{j}\ell_{j}}\right)\\
    &\hspace{3cm}\times \left( \prod_{j = 1}^q e^{i( \Omega_jT(s_j)+ \widetilde{\Omega}_{j}(s_{j}+t_0)-\widetilde{\Omega}_{j}(t_{0}))}\right)\left( \prod_{j = 1}^q n_{\mathrm{in}}(k_j - h)-n_{\mathrm{in}}(k_j)\right) d t_1 \ldots d t_q\\
    & = \epsilon_{k_{q+1}\ell_{q+1}\ell_{0}}\cdot e^{i( \Omega T t_0+ \widetilde{\Omega}(t_{0}))} \int\limits_{\tau \in \left[0, t_0 - t_{q+1}\right]} P_q(\tau, t_0, k[\mathcal N^{\mathrm{sp}}]) d\tau,
    \end{align*}
    where we define:
    \begin{align}
    \epsilon_{k_{j}\ell_{j}}&:=\epsilon_{k_{j}\ell_{j}m_{j-1}}\cdot\sqrt{\frac{\omega_{k_j}\omega_{\ell_j}}{\omega_{k_{j-1}}\omega_{\ell_{j-1}}}},\\
    P_q(\tau, t_0, k[\mathcal N^{\mathrm{sp}}]) &:= \left( \frac{\beta T}{N}\right)^q \sum\limits_{\substack{\substack{k_i \in \Z_N \cap (0,1) \\ i = 1, \ldots, q}}} \; \int\limits_{0 < s_1 < \ldots < s_{q-1} < \tau}   \left(\prod_{j = 1}^q\epsilon_{k_{j}\ell_{j}}\right) \notag \\
    &\times \left( \prod_{j = 1}^q e^{i( \Omega_jT(s_j)+ \widetilde{\Omega}_{j}(s_{j}+t_0)-\widetilde{\Omega}_{j}(t_{0}))}\right)\left( \prod_{j = 1}^q n_{\mathrm{in}}(k_j - h)-n_{\mathrm{in}}(k_j)\right)d s_1 \ldots d s_{q-1}.
    \end{align}
    Using the regularity of $n_{\mathrm{in}}$, we can bound $P_q(\tau,t_0, k[\mathcal N^{\mathrm{sp}}])$. Then for \emph{SG} case we have:
    \begin{align*}
    \left(\frac{\beta T}{N}\right) \sum_{k \in \Z_N \cap (0,1)} \left| n_{\mathrm{in}}(k - h)-n_{\mathrm{in}}(k)\right| &\lesssim \beta T h \lesssim \beta T^{\frac{1}{2}}.
    \end{align*}
    It will then recover (\ref{eq-pestimate}). Compared to the result in \cite[Lemma~5.11]{ODW}, the expression $P_q$ does not only depends on the time difference $\tau$, but also depends on the initial splicing time $t_0$ due to the nonlinear frequency shift, which means that we are not able to isolate the integral of $P_q$ from the multi-integral of frequency shifts. To bound the new integral as shown in \eqref{kqm}, we need to apply the technique we used in Proposition \ref{integral_est} and let $f_{\mathfrak{n}_0}(Tt_0)=P_q(\tau, t_0, k[\mathcal N^{\mathrm{sp}}])/\left(\beta T^{\frac{1}{2}}\right)^q$, which indicates $|f_{\mathfrak{n}_0}(Tt_0)|\lesssim1$. Thus we further need to prove $\left|\frac{d}{dt_0}f_{\mathfrak{n}_0}(Tt_0)\right|\lesssim1$. Similarly as above and applying the fact $|\Omega_j|\lesssim h$ for any $j=0,\dots,q$, then for $\beta T^{\frac{4}{5}} < 1$ we can get \eqref{eq-dpestimate}:
    \begin{align*}
    \left|\frac{d}{dt_0}P_q(\tau,t_0,  k[\mathcal N^{\mathrm{sp}}])\right| &\lesssim \left(\beta T^{\frac{1}{2}}\right)^q \max_{j=1,\dots,q}\left|\widetilde{\Omega}'_{j}(s_{j}+t_0)-\widetilde{\Omega}'_{j}(t_{0}))\right|\lesssim \left(\beta T^{\frac{1}{2}}\right)^qh \lesssim \left(\beta T^{\frac{1}{2}}\right)^q. 
    \end{align*}
    We are then able to write $\K_{\textgoth Q_{\mathcal M}}(t,s,k)$ as: 
    \begin{align*}
    \K_{\textgoth{Q}_\mathcal M} (t,s,k) &= \left( \frac{\beta T}{N}\right)^{n_{\mathrm{sp}}} \zeta(\mathcal Q^{\mathrm{sp}}) \sum_{\mathscr E^{\mathrm{sp}}} \tilde\epsilon_{\mathscr E^{\mathrm{sp}}} \int_{\mathcal{E}^{\mathrm{sp}}} \int\limits_{0 \leq \tau \leq (t_{\mathfrak n_0} - \max(t_{\mathfrak n_0^{2}}, t_{\mathfrak n_0^{3}}))} P_q(\tau,t_{\mathfrak n_0}, k[\mathcal N^{\mathrm{sp}}]) d \tau \\ 
    & \hspace{0.8cm} \times \left( \prod_{\mathfrak n \in \mathcal N^{\mathrm{sp}}}e^{\zeta_{\mathfrak{n}}i(\Omega_{\mathfrak{n}}Tt_{\mathfrak{n}}+\widetilde{\Omega}_{\mathfrak{n}}(t_{\mathfrak{n}}))} d t_{\mathfrak{n}}\right) \prod^+_{\mathfrak l \in \mathcal L^{\mathrm{sp}}} n_{\mathrm{in}}(k_{\mathfrak l})\E \left[\mathcal{B}_{\mathcal{T}^+}(\eta_{\mathcal{T}^+}(\varrho))\mathcal{B}^*_{\mathcal{T}^-}(\eta_{\mathcal{T}^-}(\varrho))\right]\\
    &=\left( \frac{\beta T}{N}\right)^{n_{\mathrm{sp}}} \zeta(\mathcal Q^{\mathrm{sp}}) \sum_{\mathscr E^{\mathrm{sp}}} \tilde \epsilon_{\mathscr E^{\mathrm{sp}}} \int _0^1\int_{\mathcal{E}^{\mathrm{sp}}_{\tau}} P_q(\tau, t_0, k[\mathcal N^{\mathrm{sp}}])   \\
    & \hspace{0.8cm}\times\left( \prod_{\mathfrak n \in \mathcal N^{\mathrm{sp}}}e^{\zeta_{\mathfrak{n}}i(\Omega_{\mathfrak{n}}Tt_{\mathfrak{n}}+\widetilde{\Omega}_{\mathfrak{n}}(t_{\mathfrak{n}}))} d t_{\mathfrak{n}}\right)d \tau \prod^+_{\mathfrak l \in \mathcal L^{\mathrm{sp}}} n_{\mathrm{in}}(k_{\mathfrak l})\E \left[\mathcal{B}_{\mathcal{T}^+}(\eta_{\mathcal{T}^+}(\varrho))\mathcal{B}^*_{\mathcal{T}^-}(\eta_{\mathcal{T}^-}(\varrho))\right].
    \end{align*}
    \end{proof}

We repeat \cite[Definition 5.13]{ODW} as follows:
\begin{definition}\label{def-M}
Consider an enhanced couple $\mathcal Q$ and enhanced molecule $\M(\Q)$. Let $\mathscr C$ be defined as a unique collection of disjoint atomic groups, such that each atomic group in $\mathscr C$ is a negative chain, negative hyperchain, or negative pseudo-hyperchain and any negative chain $\C$ of $\M$ is a subset of precisely one atomic group in $\mathscr C$. 

Suppose also that we have chosen $\mathscr C_1 \subset \mathscr C$ of all SG negative chain-like objects. Consider the set $\mathscr D_{\mathscr C_1}$ defined as a unique collection of disjoint atomic groups, such that each atomic group in $\mathscr D_{\mathscr C_1}$ is a maximal negative wide ladder of $\mathscr C_1$ and each chain $\C \in \mathscr C_1$ is a subset of precisely one atomic group in $\mathscr D_{\mathscr C_1}$. Existence and uniqueness of $\mathscr C$ and $\mathscr D_{\mathscr C_1}$ can be found in \cite[Lemma 5.12]{ODW}.

We define the set $\mathcal M_{\mathscr C_1}$ below. If $\mathcal L \in \mathscr D_{\mathscr C_1}$ consists of a single chain $\C = (v_0, \ldots, v_q) \in \mathscr C_1$:
\begin{enumerate}[label=(\roman*)]
\item If $\C$  is a chain, we include all admissible $\mathfrak n(v_i)$ into $\mathcal M$.
\item If $\C$ is a hyperchain or pseudo-hyperchain with a CN double bond, we include all admissible $\mathfrak n(v_i)$ into $\mathcal M$. If there is no CN double bond, we exclude one admissible $\mathfrak n(v_i)$. 
\end{enumerate}
Otherwise, let $\mathcal{L} = \{\mathcal{C}_1, \ldots, \mathcal{C}_m\}$, where each chain $\mathcal{C}_j$ has length $q^{(j)}$. We perform the following procedure:

\begin{enumerate}
  \item Start with $\mathcal{C}_1$ and determine which of its nodes should be added to $\mathcal{M}$ by applying the previously described procedure, treating it as a hyperchain or pseudo-hyperchain if applicable.
  
  \item For each subsequent chain $\mathcal{C}_{j+1}$ with $j \leq m - 2$, check whether all $q^{(j)}$ nodes of $\mathcal{C}_j$ are contained in $\mathcal{M}$. If so, treat $\mathcal{C}_{j+1}$ as in step (ii) above when determining which nodes to add to $\mathcal{M}$. Otherwise, treat it as in step (i).

  \item For the final chain $\mathcal{C}_m$, if it is a hyperchain or pseudo-hyperchain, or if all $q^{(m-1)}$ nodes of $\mathcal{C}_{m-1}$ are in $\mathcal{M}$, then process $\mathcal{C}_m$ as in step (ii). Otherwise, treat it as in step (i).
\end{enumerate}

Otherwise $\mathcal L = \{\mathcal C_1, \ldots, \mathcal C_m\}$ with chain $j$ having length $q^{(j)}$, perform the following:
\begin{enumerate}
\item Starting with $\mathcal C_1$ determine the nodes to be added to $\mathcal M$ by performing the above, considering it as a hyperchain or pseudo-hyperchain if it is.
\item For $\mathcal C_{j + 1}$ and $j \leq m - 2$, if there are $q^{(j)}$ nodes of $\mathcal C_j$ contained in $\mathcal M$, treat $\mathcal C_{j + 1}$ as to point (2) above in determining which nodes to add to $\mathcal M$. Otherwise, treat it according to (1). 
\item For $\mathcal C_{m}$, if it is a hyperchain or pseudo-hyperchain or there are $q^{(m-1)}$ nodes of $\mathcal C_{m-1}$ in $\mathcal M$, treat $\mathcal C_m$ according to (2). Otherwise, treat it according to (1). 
\end{enumerate}

\end{definition}
\begin{corollary} \label{cor-M}
For an enhanced couple $\Q$, and choice of SG negative chain-like objects $\mathscr C_1$ in $\M(\Q)$, consider the resulting couple $\mathcal Q^{\mathrm{sp}}$ obtained by splicing $\mathcal Q$ at the nodes in $\mathcal M = \mathcal M_{\mathscr C_1}$.
    \begin{align} \label{eq-splice-exp}
    \begin{split}
    \K_{\textgoth{Q}_\mathcal M} (t,s,k) &= \left( \frac{\beta T}{N}\right)^{n_{\mathrm{sp}}} \zeta(\mathcal Q^{\mathrm{sp}}) \sum_{\mathscr E^{\mathrm{sp}}} \tilde \epsilon_{\mathscr E^{\mathrm{sp}}} \int _{\left[ 0,1\right]^{\mathcal N_0^{\mathrm{sp}}}}\int_{\mathcal{E}^{\mathrm{sp}}_{\pmb{\tau}}} \prod_{\mathfrak n_0 \in \mathcal N^{\mathrm{sp}}_0} P_{q_{\mathfrak n_0}}(\tau_{\mathfrak n_0}, t_{\mathfrak n_0}, k[\mathcal N^{\mathrm{sp}}])  \\
    & \hspace{0.8cm}\times \left( \prod_{\mathfrak n \in \mathcal N^{\mathrm{sp}}}e^{\zeta_{\mathfrak{n}}i(\Omega_{\mathfrak{n}}Tt_{\mathfrak{n}}+\widetilde{\Omega}_{\mathfrak{n}}(t_{\mathfrak{n}}))} d t_{\mathfrak{n}}\right) d \pmb{\tau}\prod^+_{\mathfrak l \in \mathcal L^{\mathrm{sp}}} n_{\mathrm{in}}(k_{\mathfrak l})\E \left[\mathcal{B}_{\mathcal{T}^+}(\eta_{\mathcal{T}^+}(\varrho))\mathcal{B}^*_{\mathcal{T}^-}(\eta_{\mathcal{T}^-}(\varrho))\right],
    \end{split}
    \end{align}
    where $\mathcal N_0^{\mathrm{sp}}$ are the nodes at which an irregular chain was spliced out and $P_{q_{\mathfrak n_0}}$ is given in Lemma \ref{lem-splice} and $q_{\mathfrak n_0}$ is the length or irregular chain spliced out below $\mathfrak n_0$. Additionally, $\mathcal{E}^{\mathrm{sp}}_{\pmb{\tau}} = \mathcal E \cap \{t_{\mathfrak n_0} \geq t_{\mathfrak n_0^{2}} + \tau_{\mathfrak n_0}, t_{\mathfrak n_0^{3}} + \tau_{\mathfrak n_0}\}_{\mathfrak n_0 \in \mathcal N_0^{\mathrm{sp}}}$.
\end{corollary}
\begin{proof}
    See \cite[Proposition~5.14]{ODW}.
\end{proof}
\subsection{Pre-processing step in the algorithm}\label{subsec-ass}
We implement the algorithm described in Section 7.2 of \cite{ODW}. This algorithm enables us to bound the number of decorations of a molecule incrementally. Here, we define the number of decorations to be the sum of $\tilde \epsilon_{\mathscr E^{\mathrm{sp}}}$ over all decorations $\mathscr E^{\mathrm{sp}}$. Consider a molecule $\M^{\mathrm{pre}}$ at some stage of the algorithm, which is transformed into a molecule $\M^{\mathrm{post}}$ via a single operation. 
Denote by $\mathfrak{D}^{\mathrm{pre}}$ and $\mathfrak{D}^{\mathrm{post}}$ the corresponding numbers of decorations. 
Let $\mathfrak{C}$ be the counting estimate for this operation. 
Then, we have
\begin{align}
    \mathfrak{D}^{\mathrm{pre}} \le \mathfrak{C}\cdot\mathfrak{D}^{\mathrm{post}}.
\end{align}
To reduce our case with enhanced molecules to the molecules that satisfy modified assumptions in section 7.1 of \cite{ODW}:
\begin{enumerate}
\item The molecule $\M$ has 2 degree 3 atoms, with the rest degree 4. 
\item The molecule $\M$ contains no degenerate atoms (excluding the fully degenerate ones).
\end{enumerate}

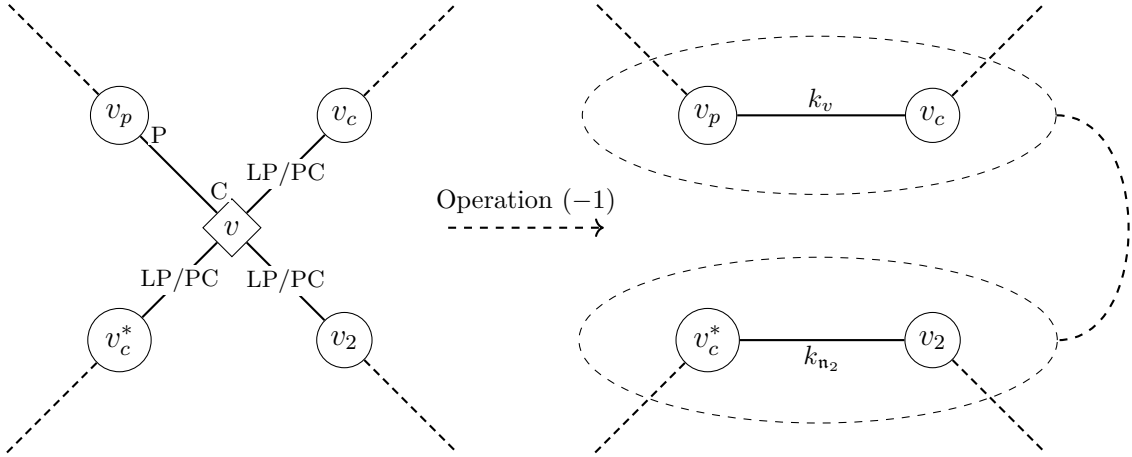
\begin{figure}[ht]
  \centering
  \begin{tikzpicture}[
      scale=1.15,
      open node/.style = {circle, draw=black, inner sep=3},
      deg atom/.style  = {draw, diamond, aspect=1.0, inner sep=2.7},
      bond/.style      = {thick},
      stub/.style      = {densely dashed, thick},
      em/.style        = {inner sep=1, outer sep=0, fill=white},
  ]

  \node[open node] (vpL) at (-1.3,  1.3) {$v_p$};
  \node[open node] (a2L) at ( 1.3,  1.3) {$v_c$};
  \node[open node] (a1L) at (-1.3, -1.3) {$v_c^*$};
  \node[open node] (a3L) at ( 1.3, -1.3) {$v_2$};
  \node[deg atom]  (vL)  at ( 0,    0  ) {$v$};

  \draw[bond] (vpL) -- (vL)
    node[em, pos=.20, anchor=east, xshift= 5.5pt, yshift= 6.5pt]{\scriptsize P}
    node[em, pos=.80, anchor=west, xshift= 2.2pt, yshift= 2.5pt]{\scriptsize C};

  \draw[bond] (vL) -- (a2L) node[em, midway]{\scriptsize LP/PC};
  \draw[bond] (vL) -- (a1L) node[em, midway]{\scriptsize LP/PC};
  \draw[bond] (vL) -- (a3L) node[em, midway]{\scriptsize LP/PC};

  \foreach \src/\dst in {a2L/vL, a1L/vL, a3L/vL, vpL/vL}{
    \draw[stub] (\src) -- ++($(\src)-(\dst)$*0.4);
  }

  \draw[thick, dashed, ->] (2.5,0) -- (4.3,0)
      node[em, midway, yshift=10pt]{\footnotesize Operation $(-1)$};

  \begin{scope}[shift={(6.8,0)}]
    \node[open node] (vpR) at (-1.3,  1.3) {$v_p$};
    \node[open node] (a2R) at ( 1.3,  1.3) {$v_c$};
    \node[open node] (a1R) at (-1.3, -1.3) {$v_c^*$};
    \node[open node] (a3R) at ( 1.3, -1.3) {$v_2$};

    \draw[bond] (vpR) -- (a2R)
      node[em, midway, yshift= 6pt]{\footnotesize $k_{v}$};
    \draw[bond] (a1R) -- (a3R)
      node[em, midway, yshift=-7pt]{\footnotesize $k_{\mathfrak n_2}$};

    \coordinate (cR) at (0,0);   

    \foreach \src/\dst in {a2R/cR, a1R/cR, a3R/cR, vpR/cR}{
      \draw[stub] (\src) -- ++($(\src)-(\dst)$*0.4);
    }

    \node[draw, ellipse, dashed, fit=(vpR)(a2R), inner sep=10pt] (topE) {};
    \node[draw, ellipse, dashed, fit=(a1R)(a3R), inner sep=10pt] (botE) {};

    \draw[bond, dashed] (topE.east) .. controls ($(topE.east)+(1.1,0)$) and
                                   ($(botE.east)+(1.1,0)$)
                         .. (botE.east);
  \end{scope}
  \end{tikzpicture}

  \caption{Operation $(-1)$ on a degenerate atom \(v\) (diamond)
           with parent \(v_p\) and neighbors
           \(v_c^*,v_2\) and $v_c$.  
           Left: original configuration with
           \(PC\) and \(LP\) bonds; dashed stubs indicate external
           connections.  
           Right: after removing the degenerate atom \(v\) and all the bonds connected,
           \(v_p\) joins \(v_c\) via bond with decoration \(k_v\) and
           \(v_c^*\) joins \(v_2\) via bond with decoration \(k_{\mathfrak n_2}\).
           The resulting sub‑molecules are highlighted by dashed
           ellipses and connected by some additional bonds represented by the curved dashed line, preserving connectivity.}
  \label{fig:op-minus-one-final}
\end{figure}

We would like to perform a pre-processing step to remove all the degenerate atoms (excluding the fully degenerate ones): 

\underline{Operation DEL} If there is any degenerate atom $v$ but not fully degenerate, corresponding to a degenerate node in $\mathcal{N}_D$, we remove the degenerate atom $v$, and also do the following (see Figure \ref{fig:op-minus-one-final} for detailed illustration) in the order:
\begin{enumerate}[label=(\alph*)]
    \item Remove all bonds connected to it.
    \item If both of its parent atom $v_p$ and child atom (or corresponding paired atom) $v_c$ (recall Definition \ref{def:enhanced-molecule}) exist in the molecule $\mathbb{M}$ and were previously connected to $v$ via bonds with decoration $k_v$, we connect them via a bond with decoration $k_v$.
    \item If Step (a) does not create a new connected component, we also connect its child atoms (or corresponding paired atoms) $v_c^*$ and $v_2$ via a bond with decoration $k_{\mathfrak{n}_2}$ (the same decoration as the bonds previously connecting them to $v$.) Otherwise, we do not add a new bond between $v_c^*$ and $v_2$.
\end{enumerate}
Start the following \underline{Operation $(-1)$} as follows:
\begin{enumerate}[label=(\roman*)]
    \item If $v$ is connected to another degenerate atom $v'$ via triple bonds, of which two must have decoration $k_{\mathfrak{n}_2}$ and the remaining one must have $k_{v}$, we remove the two degenerate atoms $v$ and $v'$ together using Operation DEL. Go to (i).
    \item Otherwise $v$ is connected to another degenerate atom $v'$ via double bonds both with decoration $k_{\mathfrak{n}_2}$, we remove the two degenerate atoms $v$ and $v'$ together using Operation DEL. Go to (i).
    \item Otherwise we remove the degenerate atom $v$ using Operation DEL. Go to (i).
\end{enumerate}
There is only one way to create a new component of: after we remove two (or more) degenerate atoms $v$ and $v'$ with the above Operation $(-1)$, their child atoms $v_c^*,v_2$ and ${v_c^*}',v_2'$ are connected in a new component. According to Step (c) of removing a degenerate atom, we will not connect ${v_c^*}'$ and $v_2'$, which become the two degree 3 atoms in the new connected component with all of the rest atoms degree 4. The degrees of all of the rest atoms in the molecule $\M$ left remain the same. Therefore, we can treat the resulting components as new molecules satisfying assumptions in section 7.1 of \cite{ODW}.
\begin{remark}
    We always gain from the above removing degeneracy Operation $(-1)$, as we can analyze the counting of decorations in the above operations:
\begin{itemize}
    \item In the Operation (i) and (ii), we reduce 4 edges and 2 vertices each time removing two degenerate atoms while preserving the number of connected components. Then we have \[\Delta \chi = \Delta E-\Delta V + \Delta F=-4+2+0 = -2,\] with counting $\mathfrak{C}=N$ for the decoration $k_{\mathfrak{n}_2}$.
    \item In the Operation (iii), we either reduce 2 edges and 1 vertex each time while preserving the number of connected components, or reduce 3 edges and 1 vertex and get a new connected component, as we only allow enhanced pairings for the couples. Consequently, after two Operation (iii) we have \[\Delta \chi = \Delta E-\Delta V + \Delta F=(-2 + 1) \times 2+0 =-5 + 2 +1 = -2,\] with counting $\mathfrak{C}=1$ if no new component and counting $\mathfrak{C}=N$ for the decoration $k_{\mathfrak{n}_2}$ if a new component split out.
\end{itemize}
In the above scenarios, we are permitted to lose a three-vector counting $\mathfrak{C}=N^2T^{-1}\log T$ where we actually lose a counting $\mathfrak{C}\leq N < N^2T^{-1}\log T$ for $T < N$.
\end{remark}
\subsection{Operation types}\label{sec-optype}
We categorize the operations of the algorithm in \cite{ODW} based on the corresponding values of $\mathfrak{C}$ and $\Delta \chi$ into the following types, while tracking the count of each type:

\begin{itemize}
    \item \textbf{Bridge Operations}. These operations have $\Delta \chi = 0$ and $\mathfrak{C} = 1$, for all atoms on one side of the bridge, and every bond except the bridge appears exactly twice with opposite signs. Let $m_0$ denote the total number of bridge operations.
    
    \item \textbf{Sole Atom Operations}. Here, $\Delta \chi = 0$, and since no bonds are removed, $\mathfrak{C} = 1$. The total count of such operations is represented by $m_1$.
    
    \item \textbf{Two-Vector Counting Operations}. For this type, $\Delta \chi = -1$ and $\mathfrak{C} = N$ as given by Proposition \ref{vector_counting-2}. Let $m_2$ represent the count of two-vector counting operations.
    
    \item \textbf{Three-Vector Counting Operations}. These operations are characterized by $\Delta \chi = -2$ and $\mathfrak{C} = N^2T^{-1}\log T$, according to Proposition \ref{vector_counting}. Let $m_3$ denote the total number of three-vector counting operations.

    \item \textbf{Removing Degeneracy Operations} consist of Operation $(-1)$. In this case $\Delta \chi = -2$ and $\mathfrak{C} \leq N$ removing every two degenerate atoms, which are gains compared with three-vector counting $\Delta \chi = -2$ and $\mathfrak{C} = N^2T^{-1}\log T$ for $T \leq N$. Let $m_{-1}$ denote the total number of degenerate atoms removed. Let $m_c$ denote the number of new connected components split out. 
\end{itemize}

Then we can similarly prove the following proposition as in \cite{ODW}:

\begin{proposition}\label{prop-op-bound}
For any molecule $\M(\Q) = \cup_{j=0}^{m_c}\M_j$ with each component $\M_j$ as assumed in Section \ref{subsec-ass}, we must have:
\begin{align}
    m_0 &\leq m_3 - 1,\\
    m_2 &\leq 3m_3 - 3.
\end{align}
\end{proposition}
\begin{proof}
Assume that $\M(\Q)$ has $n - m_{-1}$ atoms. Among these $n - m_{-1}$ atoms, $(n - m_{-1} - 2(m_c+1))$ of them have degree $4$, and the remaining $2(m_c+1)$ atoms have degree $3$.

\begin{enumerate}
    \item \textbf{Pre-processing and edge counts.}
    Since each degree-$4$ atom contributes $4$ edges and each degree-$3$ atom contributes $3$ edges, the total number of edges is
    \begin{align}
        \#\text{edges}
        =
        \frac{4(n - m_{-1} - 2(m_c+1)) + 3\cdot2(m_c+1)}{2}
        =
        2n - 2m_{-1} - (m_c+1)
    \end{align}
    \item \textbf{The edge and $\chi$ equations.}
    Recall the following “operation counts,” denoting them as $m_0,m_2,m_3,m_{-1}$:
    \begin{align}
        3m_3 + 2m_2 + m_0 
        = 
        \#\text{edges}
        =
        2n - 2m_{-1} - (m_c+1), \label{eq-5.9.1}
    \end{align} 
    and
    \begin{align}
        2m_3 + m_2 
        &= 
        \bigl(\#\text{edges}\bigr) 
        - \bigl(\#\text{atoms}\bigr) 
        + \bigl(\#\text{components}\bigr)\notag\\
        &=
        \left(2n - 2m_{-1} - (m_c+1)\right) - (n-m_{-1}) + (m_c+1)
        = n-m_{-1}. \label{eq-5.9.2}
    \end{align} 
    From the two equations \eqref{eq-5.9.1} and \eqref{eq-5.9.2}, we can solve for $m_0$ in terms of $m_3$ and $m_{-1}$:
    \begin{align}
        m_0
        =
        m_3 - (m_c+1),
    \end{align}
    which recovers a similar statement in \cite[Proposition 7.5]{ODW}.

    \item \textbf{Deriving the inequality for $m_2$.}
    With the same proof for \cite[Proposition 7.6]{ODW}, we have:
    \begin{align}
        m_2 
        \le m_3 + 2m_0 -(m_c+1)
        \leq 3m_3 - 3(m_c+1).
    \end{align}
\end{enumerate}
\end{proof}
\subsection{Bounds on couples} \label{sec-boundoncouple}
We record the following modified Proposition, proved in \cite[Proposition 7.6]{ODW}.
\begin{proposition}\label{prop-rigidity}
Consider a labeled molecule $\mathbb M = \mathbb M(\mathcal Q)$ of order $n\leq L^3$. Suppose we fix $k \in \Z_N\cap(0,1)$ and $\alpha_v \in \R$ for each atom $v$ of $\mathbb M$. Consider all $k$-decorations $(k_\ell)$ of $\M$ such that 
\begin{enumerate}[(i)]
\item The $k$-decoration $(k_\ell)$ is inherited from a $k$-decoration $\mathscr E$of $\mathcal Q$ that satisfies the SG and LG assumptions in \cite[Proposition 5.14]{ODW} as well as non-degeneracy conditions $\tilde \epsilon_{\mathscr E}.$
\item The decoration is restricted by $(\alpha_v)$ in the sense that $|\Omega_v - \alpha_v| \leq T^{-1}$, where $T < N$. 
\end{enumerate}
Then, the number $\mathfrak D$ of such $k$-decorations is bounded by 
\begin{equation}\label{eq-counting}
\mathfrak D \leq C^n N^{n}T^{-\frac{n}{5} - \frac{3}{5}}(\log T)^n.
\end{equation}
\end{proposition}
\begin{proof}
    We have the equation for $\chi$: $2m_3 + m_2 = n-m_{-1}$ and by Proposition \ref{prop-op-bound}, $5m_3 - 3 \geq n-m_{-1}$. Note that we lose at most a counting of $N$ each time we remove two (or more) degenerate atoms (whether they are connected by triple bonds, double bonds, or result in the formation of a new component). Then we have: 
    \begin{align*}
    \mathfrak D &\lesssim C^n (N^2T^{-1}\log T)^{m_3}N^{m_2}N^{\frac{m_{-1}}{2}} \\
    &= C^n N^{n-\frac{m_{-1}}{2}} (\log T)^{m_3} T^{-m_3} \\
    & \lesssim C^n N^{n-\frac{m_{-1}}{2}} (\log T)^{m_3}T^{-\frac{n}{5} - \frac{3}{5}+\frac{m_{-1}}{5}}\\
    & \lesssim C^n N^{n} (\log T)^{n}T^{-\frac{n}{5} - \frac{3}{5}},
    \end{align*}
for $T < N$.
\end{proof}
\begin{remark}
    The validity of three-vector and two-vector counting requires $T < N$. On the other hand, removing degenerate atoms is advantageous when $N^{-\frac{m_{-1}}{2}}T^{\frac{m_{-1}}{5}}\leq1$, which corresponds to $T\leq N^{\frac{5}{2}}$.
\end{remark}
\begin{lemma} \label{multiplicity}
For enhanced trees $\mathcal{T}_1$ and $\mathcal{T}_2$,
\begin{align}
    \mathbb{E}\left(\mathcal{B}_{\mathcal{T}_1}(\eta_{\mathcal{T}_1}(\varrho)){\mathcal{B}^*_{\mathcal{T}_2}(\eta_{\mathcal{T}_2}(\varrho))}\right) = \sum_{\substack{\mathscr P \textit{: enhanced pairings of}\\ \mathcal{T}_1,\mathcal{T}_2}} 1. \label{enhanced-pairing}
\end{align}
\end{lemma}
\begin{proof}
    We prove by induction. Suppose the formula is true for $|\mathcal N_D|=k-1$, then for $|\mathcal N_D|=k$, without loss of generality, we assume that there is one more degenerate node $\mathfrak{n}_*$ coming from $\mathcal{T}_1$. Assume that $\widetilde{\mathcal{B}}_{\mathcal{T}_1}$ represents the multiplicity (product of random phases) of the enhanced tree ${\mathcal{T}_1}$ with $|\mathcal N_{D_1}|=k_1-1$, that is $\mathfrak{n}_* \notin \mathcal N_{D_1}$, $\mathcal{T}_{1j}$'s are subtrees of degenerate node $\mathfrak{n}_*$ for $j=1, 2, 3$, and $\widetilde{\widetilde{\mathcal{B}}}_{\mathcal{T}_1}$ represents the multiplicity of ${\mathcal{T}_1}$ excluding $\mathcal{T}_{11}$ and $\mathcal{T}_{12}$. Here we assume that the root $\mathfrak{n}_{11}$ of $\mathcal{T}_{11}$ is $\mathfrak{n}_{1c^*}$. Then from Proposition \ref{prop-treeexp} we know that:
\begin{align}
    \mathbb{E}\left(\mathcal{B}_{\mathcal{T}_1}{\mathcal{B}^*_{\mathcal{T}_2}}\right) = \mathbb{E}\left(\widetilde{\mathcal{B}}_{\mathcal{T}_1}{\mathcal{B}^*_{\mathcal{T}_2}}\right) - \mathbb{E}\left(\mathcal{B}_{\mathcal{T}_{11}}{\mathcal{B}^*_{\mathcal{T}_{12}}}\right)\mathbb{E}\left(\widetilde{\widetilde{\mathcal{B}}}_{\mathcal{T}_1}{\mathcal{B}^*_{\mathcal{T}_2}}\right)
\end{align}
which is the number of enhanced pairing of $\mathcal{T}_1,\mathcal{T}_2$ (excluding any complete pairing from $\mathcal{T}_{11}$ and $\mathcal{T}_{12}$) and concludes the proof.
\end{proof}
\begin{proof}[Proof of Proposition \ref{bound-couple}]
Recall that
\begin{align*}
    \widetilde \Omega_k(t)&=\Omega_kA(t),\\
    \dot{A}(t) &= G(t, A(t)) = \frac{3\beta T}{2\kappa^2N}\sum_{|\mathcal{T}_1|,|\mathcal{T}_2|\le n}\sum_{l}\omega_l\mathbb{E}c_{l}^{\mathcal{T}_1}(t){c_{l}^{\mathcal{T}_2}}^*(t).
\end{align*}
We prove (\ref{E1}) and (\ref{E2}) using bootstrap arguments. We establish the hypotheses as:
\begin{align}
     \left|\sum_{|\mathcal{T}_{1}|+|\mathcal{T}_{2}|=n_{\Q}}\mathbb{E}\left[c_{k}^{\mathcal{T}_{1}}(t){c_{k}^{\mathcal{T}_{2}}}^*(t)\right]\right|&\leq 2C (\log N \log T)^{n_{\Q}} T^{-\frac{3}{5}}\left(\beta T^{4/5}\right)^{n_{\Q}},\label{H1}\\
     \left|\sum_{|\mathcal{T}_{1}|+|\mathcal{T}_{2}|=n_{\Q}}\mathbb{E}\left[\dot{c}_{k}^{\mathcal{T}_{1}}(t){c_{k}^{\mathcal{T}_{2}}}^*(t)\right]\right|&\leq 2C (\log N \log T)^{n_{\Q}-1}T^{\frac{2}{5}}\left(\beta T^{4/5}\right)^{n_{\Q}},\label{H2}
\end{align}
for the same constant $C$ as in (\ref{E1}, \ref{E2}).
    (\ref{H1}, \ref{H2}) give upper bound for $\dot{A}(t)$. From \eqref{eq-12} and \eqref{eq.def of A_k}, $|c_l^{\T}|$ and the function $G(t,A(t))$ are bounded for almost every $\varrho$, and smooth in its variables, which guarantees the global existence and uniqueness of $A(t)$ for almost all $\varrho$. $A(t) = C_0\beta Tt+{A}_{\geq}(t)$, where
    \begin{align}
        \dot{A}_{\geq}(t) = \frac{3\beta T}{2\kappa^2N}\sum_{1 \leq |\mathcal{T}_1|,|\mathcal{T}_2|\leq n}\sum_{l}\omega_l\mathbb{E}\left[c_{l}^{\mathcal{T}_1}(t){c_{l}^{\mathcal{T}_2}}^*(t)\right].
    \end{align}
    Then from (\ref{H1}, \ref{H2}), we have: 
    \begin{align}
        \dot{A}_{\geq}(t) & \lesssim \beta T^{1-\frac{3}{5}} \frac{\left(\beta T^{4/5}\right)^2}{1-\beta T^{4/5}} \lesssim \beta T,
        \\
        \ddot{A}_{\geq}(t) & \lesssim \beta T^{1+\frac{2}{5}} \frac{\left(\beta T^{4/5}\right)^2}{1-\beta T^{4/5}} \lesssim \beta T^{\frac{9}{5}}, \label{eq-Add}
    \end{align}
    for $\beta T^{4/5} \ll 1$. Therefore we can apply proposition \ref{integral_est}.
    Using Proposition \ref{vector_counting}, \ref{integral_est} and \cite[Proposition 2.6]{ODW}, we can show bounds for couples (\ref{E1}, \ref{E2}), which closes the bootstrap argument. Because the order of our couples is at most \(L^3\), the total number of couples of order \(n_{\Q}\) (independent of \(n_h\)) is on the order of
    $O(C^{L^3}\,(L^3)!).$ From Lemma \ref{multiplicity}, we can bound $\mathbb{E}\left(\mathcal{B}_{\mathcal{T}_1}(\eta_{\mathcal{T}_1}(\varrho)){\mathcal{B}^*_{\mathcal{T}_2}(\eta_{\mathcal{T}_2}(\varrho))}\right)\lesssim C^{L^3}\,(L^3)!$. 
    
    Similarly, given an enhanced couple $\Q$, there are $O(C^L)$ choices for the collection of SG negative-chain like objects $\mathscr C_1$ and $\mathcal M_{\mathscr C_1}$ defined in Definition \ref{def-M}. It suffices to fix an enhanced couple $\Q$ along with a choice $\mathscr C_1$ and $\mathcal M=\mathcal M_{\mathscr C_1}$ and bound the corresponding expression for  $\K_{\textgoth{Q}_{\mathcal M}}$ in (\ref{eq-splice-exp}). We could use the advantageous 2-vector counting $NT^{-\frac{1}{2}}$ for the LG irregular chains and uni-directional $S_2^{+}$, which are as good as 3-vector counting operations and only improve the bounds.
    
    As we have shown in the proof of Lemma \ref{lem-splice}, to bound the oscillatory integral after splicing in \eqref{eq-splice-exp}:
    \begin{align}
        \int _{\left[ 0,1\right]^{\mathcal N_0^{\mathrm{sp}}}}\int_{\mathcal{E}^{\mathrm{sp}}_{\pmb{\tau}}} \prod_{\mathfrak n_0 \in \mathcal N^{\mathrm{sp}}_0} P_{q_{\mathfrak n_0}}(\tau_{\mathfrak n_0}, t_{\mathfrak n_0}, k[\mathcal N^{\mathrm{sp}}])\left( \prod_{\mathfrak n \in \mathcal N^{\mathrm{sp}}}e^{\zeta_{\mathfrak{n}}i(\Omega_{\mathfrak{n}}Tt_{\mathfrak{n}}+\widetilde{\Omega}_{\mathfrak{n}}(t_{\mathfrak{n}}))} d t_{\mathfrak{n}}\right) d \pmb{\tau}, \label{eq-osc}
    \end{align}
    we need to apply the technique we used in Proposition \ref{integral_est} and let \[f_{\mathfrak{n}_0}(Tt_{\mathfrak n_0})=\left(\beta T^{\frac{1}{2}}\right)^{-q_{\mathfrak n_0}}P_{q_{\mathfrak n_0}}(\tau_{\mathfrak n_0}, t_{\mathfrak n_0}, k[\mathcal N^{\mathrm{sp}}]),\] which satisfies $|f_{\mathfrak{n}_0}(Tt_{\mathfrak n_0})|\lesssim1$ and $\left|\frac{d}{dt_{\mathfrak n_0}}f_{\mathfrak{n}_0}(Tt_{\mathfrak n_0})\right|\lesssim1$ for each $\mathfrak n_0 \in \mathcal N^{\mathrm{sp}}_0$. Then by Corollary \ref{cor-M}, we can derive:
    \begin{align}
    &\left|\sum_{|\mathcal{T}_{1}|+|\mathcal{T}_{2}|=n_{\Q}}\mathbb{E}\left[c_{k}^{\mathcal{T}_{1}}(t){c_{k}^{\mathcal{T}_{2}}}^*(t)\right]\right| \notag\\
    \lesssim& \;
      C^{n_{\Q}}\left(\beta T^{1/2} \right)^{n_0} \hspace{-.15cm}\sup_{\pmb \tau \in [0,1]^{\mathcal N_0^{\mathrm{sp}}}}\sum_{{\lambda[\mathcal N^{\mathrm{sp}}]}}\left( \frac{\beta T}{N}\right)^{n_{{\mathrm{sp}}}}\notag\\
      &\hspace{3cm}\times\sum_{\tilde{\mathscr E}_{\lambda[\mathcal N^{\mathrm{sp}}]}^{\mathrm{sp}}}\tilde \epsilon_{\tilde{\mathscr E}_{\lambda[\mathcal N^{\mathrm{sp}}]}^{\mathrm{sp}}}\left|\int_{\mathcal{E}^{\mathrm{sp}}_{\pmb{\tau}}}\prod_{\mathfrak{n}_0 \in \mathcal{N}^{\mathrm{sp}}_0}f_{\mathfrak{n}_0}(Tt_{\mathfrak n_0})\prod_{\mathfrak{n} \in \mathcal{N}^{\mathrm{sp}}}e^{\zeta_{\mathfrak{n}}i(\Omega_{\mathfrak{n}}Tt_{\mathfrak{n}}+\widetilde{\Omega}_{\mathfrak{n}}(t_{\mathfrak{n}}))} d t_{\mathfrak{n}}\right|\notag\\
      \lesssim& \; C^{n_{\Q}}\left(\beta T^{1/2} \right)^{n_0}\left(\sup_{\lambda[\mathcal N^{\mathrm{sp}}]}\sum_{\tilde{\mathscr E}_{\lambda[\mathcal N^{\mathrm{sp}}]}^{\mathrm{sp}}}\tilde \epsilon_{\tilde{\mathscr E}_{\lambda[\mathcal N^{\mathrm{sp}}]}^{\mathrm{sp}}}\left( \frac{\beta T}{N}\right)^{n_{\mathrm{sp}}}\right)\notag\\
      &\hspace{3cm}\times\left(\sum_{{\lambda[\mathcal N^{\mathrm{sp}}]}}\sup_{\pmb \tau \in [0,1]^{\mathcal N_0^{\mathrm{sp}}}}\left|\int_{\mathcal{E}^{\mathrm{sp}}_{\pmb{\tau}}}\prod_{\mathfrak{n}_0 \in \mathcal{N}^{\mathrm{sp}}_0}f_{\mathfrak{n}_0}(Tt_{\mathfrak n_0})\prod_{\mathfrak{n} \in \mathcal{N}^{\mathrm{sp}}}e^{\zeta_{\mathfrak{n}}i\lambda_{\mathfrak{n}}(Tt_{\mathfrak{n}}+A(t_{\mathfrak{n}}))} d t_{\mathfrak{n}}\right|\right),\label{eq-sp1}
    \end{align}
    where $\lambda[\mathcal N^{\mathrm{sp}}] \in \Z^{|\mathcal N^{\mathrm{sp}}|}$ and ${\tilde{\mathscr E}_{\lambda[\mathcal N^{\mathrm{sp}}]}^{\mathrm{sp}}}$ denotes decorations of the couple satisfying $|T\Omega_{\mathfrak n} - \lambda_{\mathfrak n}| \leq 1$ for each $\mathfrak n \in \mathcal N^{\mathrm{sp}}$, and $C_1>0$ is a constant. Here, $n_0 + n_{\mathrm{sp}} = n_{\Q}$, where $n_0$ is number of nodes spliced out. We apply Proposition \ref{integral_est} to bound:
    \begin{align}
        &\sum_{{\lambda[\mathcal N^{\mathrm{sp}}]}}\sup_{\pmb \tau \in [0,1]^{\mathcal N_0^{\mathrm{sp}}}}\left|\int_{\mathcal{E}^{\mathrm{sp}}_{\pmb{\tau}}}\prod_{\mathfrak{n}_0 \in \mathcal{N}^{\mathrm{sp}}_0}f_{\mathfrak{n}_0}(Tt_{\mathfrak n_0})\prod_{\mathfrak{n} \in \mathcal{N}^{\mathrm{sp}}}e^{\zeta_{\mathfrak{n}}i\lambda_{\mathfrak{n}}(Tt_{\mathfrak{n}}+A(t_{\mathfrak{n}}))} d t_{\mathfrak{n}}\right| \notag\\
        \lesssim&\; \sum_{\lambda[\mathcal N]} \sum_{d_{\mathfrak n} \in \{0,1\}}\frac{1}{\langle \lambda_{\mathfrak n} + T{q}_{\mathfrak n}\rangle} \lesssim C^{n_{\Q}} (\log N)^{n_{\Q}}.
    \end{align}
    Then by Proposition \ref{prop-rigidity}, we get:
    \begin{align}
        \eqref{eq-sp1}
    \lesssim& \; C^{n_{\Q}} \left(\beta T^{1/2} \right)^{n_0}\left( \sup_{\lambda[\mathcal N^{\mathrm{sp}}]} \sum_{\tilde{\mathscr E}_{\lambda[\mathcal N^{\mathrm{sp}}]}^{\mathrm{sp}}}\tilde \epsilon_{\tilde{\mathscr E}_{\lambda[\mathcal N^{\mathrm{sp}}]}^{\mathrm{sp}}} \left( \frac{\beta T}{N}\right)^{n_{\mathrm{sp}}}\right)  (\log N)^{n_{\Q}}\notag \\
    <& \; C_1 (\log N \log T)^{n_{\Q}} T^{-\frac{3}{5}}\left(\beta T^{4/5}\right)^{n_{\Q}},
    \end{align}
    Here we use $\beta T^{\frac{1}{2}} \ll \left(\frac{\beta T}{N}\right)NT^{-\frac{1}{5}}=\beta T^{\frac{4}{5}}$ and the size of the oscillatory integral \eqref{eq-osc} can be bounded by $(\log N)^{n_{\Q}}$. Similarly for $\mathbb{E}\left[\dot{c}_{k}^{\mathcal{T}_{1}}(t){c_{k}^{\mathcal{T}_{2}}}^*(t)\right]$, we may expand $\dot{c}_{k}^{\mathcal{T}_{1}}(t)$ and get:
    \begin{align}
    &\quad\left|\sum_{|\mathcal{T}_{1}|+|\mathcal{T}_{2}|=n_{\Q}}\mathbb{E}\left[\dot{c}_{k}^{\mathcal{T}_{1}}(t){c_{k}^{\mathcal{T}_{2}}}^*(t)\right]\right|\notag\\
    & \lesssim \frac{\beta T}{N}\max\Bigg(\sum_{\substack{k_3-k_4+k_5=k\\k_3,k_5 \neq k_4}}T_{k,3,4,5}\left|\mathbb{E}\left[c_{k_3}^{\mathcal{T}_3}(t){c_{k_4}^{\mathcal{T}_4}}^*(t)c_{k_5}^{\mathcal{T}_5}(t){c_{k}^{\mathcal{T}_{2}}}^*(t)\right]\right|, \notag\\
    &\qquad \sum_{k_3}T_{k,3,3,k}\Big(\left|\mathbb{E}\left[c_{k_3}^{\mathcal{T}_3}(t){c_{k_3}^{\mathcal{T}_4}}^*(t)c_{k}^{\mathcal{T}_5}(t){c_{k}^{\mathcal{T}_{2}}}^*(t)\right]\right|-\left|\mathbb{E}\left[c_{k_3}^{\mathcal{T}_3}(t){c_{k_3}^{\mathcal{T}_4}}^*(t)\right]\right|\cdot\left|\mathbb{E}\left[c_{k}^{\mathcal{T}_5}(t){c_{k}^{\mathcal{T}_{2}}}^*(t)\right]\right|\Big)\Bigg)\notag\\
    &\lesssim C^{n_{\Q}} \left(\beta T^{1/2} \right)^{n_0}\left( \sup_{\lambda[\mathcal N^{\mathrm{sp}} \setminus \mathfrak{r_1}]} \sum_{\tilde{\mathscr E}_{\lambda[\mathcal N^{\mathrm{sp}} \setminus \mathfrak{r_1}]}} \tilde \epsilon_{\tilde{\mathscr E}_{\lambda[\mathcal N^{\mathrm{sp}} \setminus \mathfrak{r_1}]}^{\mathrm{sp}}}\left( \frac{\beta T}{N}\right)^{n_{\mathrm{sp}}}\right) (\log N)^{n_{\Q}-1}\notag\\
    &< C_2 (\log N \log T)^{n_{\Q}-1}T^{\frac{2}{5}}\left(\beta T^{\frac{4}{5}}\right)^{n_{\Q}}. \label{eq-partialt}
    \end{align}
    where $C_2>0$ is a constant. Note that for the decorations $\{(k_3,k_4,k_5):k_3-k_4+k_5=k,k_3,k_5 \neq k_4\}$ we lose a counting $N^2$ instead of $N^2T^{-1}\log T$ since there is no control on the factor $\Omega$ within a $T^{-1}$ strip. We take $C=\max(C_1, C_2)$ to complete the proof.
\end{proof}

\section{The Remainder}
\label{section-remainder}
In this section, we prove Proposition \ref{prop-remainder}.
\subsection{Kernels of \texorpdfstring{$\Ell$}{L}} In order to introduce the kernels of \(\Ell\), we recall from \cite{WKE} the following extension of trees and couples:

\begin{definition}\label{def-flower}
A \emph{flower tree} is a tree \(\mathcal{T}\) in which a particular leaf \(\mathfrak{f}\) is distinguished and called the \emph{flower}. Choosing a different leaf \(\mathfrak{f}\) for the same tree \(\mathcal{T}\) yields a different flower tree. The unique path connecting the root \(\mathfrak{r}\) and the flower \(\mathfrak{f}\) is referred to as the \emph{stem}. A \emph{flower couple} is a pair of flower trees whose flowers are matched in such a way that they carry opposite signs.

The \emph{height} of a flower tree \(\mathcal{T}\) is defined as the number of branching nodes along its stem. Evidently, to construct a flower tree of height \(n_h\), one can start with a single node and successively attach two sub-trees \(n_h\) times. A flower tree is said to be \emph{admissible} if each of the attached sub-trees has scale at most \(L\).
\end{definition}
We adapt the following proposition from \cite[Proposition 11.2]{WKE} and \cite[Proposition 6.2]{ODW} to our $\beta$-FPUT model:
\begin{proposition} \label{prop-Lm}
Let $\mathscr L$ be defined as in (\ref{eq-linearop}). Note that $\mathscr L^{n_h}$ is an $\R$-linear operator for $n \geq 0$. Define its kernels $(\mathscr L^{n_h})_{k \ell}^\zeta(t,s)$ for $\zeta \in \{\pm\}$ by 
\begin{equation*}
(\mathscr L^{n_h} b)_k (t) = \sum_{\zeta \in \{\pm\}} \sum_{\ell} \int_{\R} (\mathscr L^{n_h})_{k \ell}^{\zeta} (t,s) b_{\ell}(s)^{\zeta}d s.
\end{equation*}
Then, for each $1 \leq n \leq L$ and $\zeta \in \{\pm\}$, we can decompose
\begin{equation}
(\mathscr L^{n_h})^\zeta_{k, \ell} = \sum_{n \leq m \leq L^3} (\mathscr L^{n_h})_{k, \ell}^{m, \zeta},
\end{equation}
such that for any $n_h \leq m \leq L^3$ and $k,\ell \in \Z_N\cap(0,1)$ and $t,s \in (0,1)$ with $t > s$, we have 
\begin{equation} \label{eq-Lnm}
\E |(\mathscr L^{n_h})_{k, \ell}^{m, \zeta}(t,s)|^2 \lesssim  (\beta T^{4/5})^m N^{30}.
\end{equation}
\end{proposition}
\begin{proof}
    We follow the proof in \cite[Proposition 11.2]{WKE} and \cite[Proposition 6.2]{ODW}, for flower trees $\T_\mathfrak f$ of height $n_h$ and scale $m$ such that the $\zeta_\mathfrak r=+$ and $\zeta_\mathfrak f=\zeta$. Then for $t \geq s$, we similarly define as in \eqref{eq-13}:
    \begin{equation} \label{eq-Jtilde}
    \widetilde{c}^{\mathcal T_\mathfrak f}_{k,\ell}(t,s) = \left(\frac{\beta T}{N}\right)^m \zeta(\mathcal{T}_\mathfrak f) \sum_{\mathscr D} \epsilon_{\mathscr D} \boldsymbol{\delta}(t_{\mathfrak f^p} - s) \mathcal{A}_{\mathcal{T}_\mathfrak f}(t,\Omega[\mathcal{N}], \widetilde{\Omega}[\Nc])\prod_{\mathfrak f \neq \mathfrak l \in \mathcal L} \sqrt{n_{\text{in}}(k_{\mathfrak{l}})}\mathcal{B}_{\mathcal{T}_\mathfrak f}(\eta_{\mathcal{T}_\mathfrak f})(\varrho)) \boldsymbol{1}_{k_{\mathfrak f} = \ell},
    \end{equation}
     where $\mathscr D$ is a $k$-decoration of $\T_\mathfrak f$, and $\mathfrak f^p$ is the parent of $\mathfrak f$. Then for flower couples $\mathcal Q_\mathfrak f=(\mathcal T^+_\mathfrak f, \mathcal T^-_\mathfrak f)$, where $\mathcal T^{\pm}_\mathfrak f$ has sign $\pm \zeta$, height $n_h$ and scale $m$. Then for $t \geq s$, we can similarly define as in \eqref{eq-KQ}:
     \begin{align}\label{eq-KQtilde}
         (\widetilde{\mathcal K}_{\mathcal Q_\mathfrak f})(t,s,k, \ell)  &= \left( \frac{\beta T}{N}\right)^{2m} \zeta(\mathcal Q) \sum_{\mathscr E} \epsilon_{\mathscr E}\int_\mathcal{E} \prod_{\mathfrak n \in \mathcal N}e^{\zeta_{\mathfrak{n}}i(\Omega_{\mathfrak{n}}Tt_{\mathfrak{n}}+\widetilde{\Omega}_{\mathfrak{n}}(t_{\mathfrak{n}}))} d t_{\mathfrak{n}} \prod_{\mathfrak f} \boldsymbol{\delta}(t_{\mathfrak f^p} - s)  \notag\\
         &\qquad \qquad \qquad \qquad \qquad \times \prod^+_{\mathfrak f \neq \mathfrak l \in \mathcal L}n_{\mathrm{in}}(k_{\mathfrak l})\mathbb{E}\left(\mathcal{B}_{\mathcal{T}_1}(\eta_{\mathcal{T}_1}(\varrho)){\mathcal{B}^*_{\mathcal{T}_2}(\eta_{\mathcal{T}_2}(\varrho))}\right)\boldsymbol{1}_{k_{\mathfrak f} = \ell}
     \end{align}
     where $\mathscr E$ is a $k$-decoration of $\Q_\mathfrak f$. In \eqref{eq-KQtilde}, the Dirac delta factor \(\boldsymbol{\delta}(t_{\mathfrak f^p} - s)\) imposes \(t_{\mathfrak f^p} = s\), so we skip any phase integral with respect to \(t_{\mathfrak f^p}\) for each of the two flowers \(\mathfrak f\). Although this alters the counting, the resulting discrepancy can be bounded by at most \(N^{10}\). With $\widetilde{c}^{\mathcal T_\mathfrak f}_k(t)$ and $(\widetilde{\mathcal K}_{\mathcal Q_\mathfrak f})(t,s,k, \ell)$ we can express:
     \begin{align}
    (\mathscr L^{n_h})_{k, \ell}^{m, \zeta}(t,s) &= \sum_{\mathcal T_\mathfrak f} \widetilde{c}^{\mathcal T_\mathfrak f}_{k,\ell}(t,s) \\
    \E \left| (\mathscr L^{n_h})_{k, \ell}^{m, \zeta}(t,s)\right|^2 &= \sum_{\mathcal Q_\mathfrak f} \widetilde{\mathcal K}_{\mathcal Q_\mathfrak f}(t,s,k,\ell).
    \end{align}
     Suppose $\mathcal{Q}_\mathfrak f$ is an admissible flower couple and that $\widetilde{\mathcal{Q}_\mathfrak f}$ is congruent to $\mathcal{Q}_\mathfrak f$ in the sense of Definition~\ref{def:congruence}. Then $\widetilde{\mathcal{Q}_\mathfrak f}$ is also an admissible flower couple, provided its flower is chosen as the image of the flower of $\mathcal{Q}_\mathfrak f$. To apply the same proof as in Proposition \ref{bound-couple}, we also need to first splice out the irregular chains with small gaps as we did in Lemma \ref{lem-splice}. Note that in our $\beta$-FPUT model, in the remainder we have extra degenerate terms as in \eqref{eq-LD} compared to the NLS model. The first extra term $\left(\left|c^{\leq n}_{k_1}(s)\right|^2 - \E \left|c^{\leq n}_{k_1}(s) \right|^2\right)\mathcal{R}^{(n+1)}_{k}$ is similar to our enhanced structure which only allow enhanced pairing under the degenerate nodes and will not cause irregular chains stacking up. The last two terms $\left(c^{\leq n*}_{k_1}(s)\mathcal{R}^{(n+1)}_{k_1}+c^{\leq n}_{k_1}(s)\mathcal{R}^{(n+1)*}_{k_1}\right)c^{\leq n}_{k}(s)$ can also be handled well by the splicing process after summing up all admissible congruent flower couples. One more loss we could gain is that if the chains contain the flowers, we need to avoid splicing the flowers out which may cause a loss of $N^2$ with two extra double bond in the molecule. Then applying Proposition \ref{bound-couple}, we can recover the bound \eqref{eq-Lnm} by adding a loss of $N^{30}$.
\end{proof}
\begin{lemma}{(Gaussian Hypercontractivity)} \label{lem-hypercontractivity}
Let $\{\eta_k\}$ be i.i.d. Gaussians or random phase. Given $\zeta_j \in \{\pm\}$ and random variable $X$ of the form 
\begin{equation}
X = \sum_{k_1, \ldots, k_n} a_{k_1, \ldots, k_n} \prod_{j = 1}^n \eta_{k_j}^{\zeta_j}(\omega),
\end{equation}
where $a_{k_1, \ldots, k_n}$ are constants. Then, for $q \geq 2$,
\begin{equation}
\E \left|X\right|^q \leq (q-1)^{\frac{nq}{2}}\cdot \left( \E |X|^2\right)^{q/2}
\end{equation}
\end{lemma}
\begin{proof}
See, for instance, \cite[Lemma 2.6]{Hypercontractivity}. 
\end{proof}
\subsection{Bound on the operator $\Ell$}
\begin{proof}[Proof of Proposition \ref{prop-remainder}]
After applying Cauchy-Schwarz inequality:
\begin{align*}
||\mathscr L^{n_h}(a)||_Z^2 &= \sup_{0 \leq t \leq 1} N^{-1} \sum_{k \in \Z_N\cap(0,1)} |(\mathscr L^{n_h} a)_k(t)|^2 \\
&= \sup_{0 \leq t \leq 1} N^{-1} \sum_{k \in \Z_N\cap(0,1)} \left| \sum_{\zeta \in \{\pm\}} \sum_{\ell \in \Z_N\cap(0,1)} \int\limits_{0 \leq s \leq t} \sum_{n \leq m \leq L^3} (\mathscr L^{n_h})_{k, \ell}^{m, \zeta}a_{\ell}^\zeta (s) ds  \right|^2 \\
&\lesssim \sup_{0 \leq s \leq t \leq 1} \sup_{\zeta} \sup_m N^{-1}L^3 \sum_{k \in \Z_N\cap(0,1)} \left| \sum_{\ell \in \Z_N\cap(0,1)} (\mathscr L^{n_h})_{k, \ell}^{m, \zeta}(t,s) a_\ell^\zeta(s) \right|^2 \\
& \lesssim ||a||_Z^2 L^3\sup_{0 \leq s \leq t \leq 1} \sup_\zeta \sup_m \sum_{k \in \Z_N\cap(0,1)}\sum_{\ell \in \Z_N\cap(0,1)}  \left| (\mathscr L^{n_h})_{k, \ell}^{m, \zeta}(t,s) \right|^2 \\ 
& \lesssim ||a||_Z^2 N^2L^3\sup_{0 \leq s \leq t \leq 1} \sup_\zeta \sup_m  \sup_{k, \ell} |(\mathscr L^{n_h})_{k, \ell}^{m, \zeta}(t,s)|^2.
\end{align*} 
Let $L(t,s) = \sup_{k, \ell} \sup_{0 \leq s \leq t \leq 1} |(\mathscr L^{n_h})_{k, \ell}^{m, \zeta}(t,s)|$, for fixed $m$ and $\zeta$. As taking $\partial_t$ derivative just corresponds to omitting the $t_\mathfrak{r}$ integration and producing something like (\ref{eq-partialt}), then by Proposition \ref{prop-Lm} and Lemma \ref{lem-hypercontractivity}:
\begin{align}\label{eq-LR-bound}
\E ||L(t,s)||_{L_{t,s}^p((0,1)^2) L^p_{k,\ell}(\Z_N^2\cap (0,1)^2)}^p& \lesssim p^{mp}(\beta T^{4/5})^{mp/2} N^{30p},\\
\E ||\partial_{(t,s)}L(t,s)||^p_{L_{t,s}^p((0,1)^2) L^p_{k,\ell}(\Z_N^2\cap (0,1)^2)} & \lesssim p^{mp}(\beta T^{4/5})^{mp/2} N^{31p}.
\end{align}
By using the Gagliardo-Nirenberg inequality for $(t,s)\in(0,1)^2$, and bounding the $L_{k,\ell}^\infty$ norm by the $L_{k,\ell}^p$ norm for $k,\ell \in \Z_N\cap(0,1)$ with an extra loss $N^{2/p}$, we conclude that
\begin{align}
    \E ||L(t,s)||_{L_{t,s}^\infty((0,1)^2) L^\infty_{k,\ell}(\Z_N^2\cap (0,1)^2)}^p& \lesssim p^{mp}(\beta T^{4/5})^{mp/2} N^{31p+2}.
\end{align}
Thus with probability $\geq 1-N^{-p/2}$, we have
\begin{align}
    L(t,s) \lesssim p^{m}(\beta T^{4/5})^{m/2} N^{33},
\end{align}
which implies 
\begin{equation} \label{remains}
\sup_{k, \ell} \sup_{0 \leq s \leq t \leq 1} |(\mathscr L^{n_h})_{k, \ell}^{m, \zeta}(t,s)| \lesssim (\beta T^{4/5})^{m/2} N^{35}.
\end{equation}
with probability $\geq 1-N^{-A}$ by taking $p \geq 2A$.
\end{proof}

\section{Proof of the Main Theorem}
\label{section-main}
\begin{proposition} \label{prop-b}
With probability $\geq 1 - N^{-A}$, the mapping defined by the right hand side of (\ref{eq-op}) is a contraction mapping from the set $\{\mathcal{R}: ||\mathcal{R}||_Z \leq N^{-500}\}$ to itself. 
\end{proposition}
\begin{proof}
By Proposition~\ref{prop-remainder}, there is an exceptional event of probability at most \(N^{-A}\) on which our required estimates may fail. We exclude that event, so on the complementary event of probability at least \(1 - N^{-A}\) we have, in particular,
\begin{align}
  \|\mathscr{R}(t)\|_{Z}\;&\lesssim\;(\beta T^{4/5})^{n/2}\,N^5 \label{eq-Rest-restate}\\
  \left\|\sum_{|\mathcal{T}|= m}c_k^{\mathcal{T}}(t)\right\|_{Z}\;&\lesssim\;(\beta T^{4/5})^{m/2}\,N^5. \label{eq-Jest-restate}
\end{align}
We obtain the above estimates by applying Proposition~\ref{bound-couple} to \eqref{eq-9} and arguments as in the proof of Proposition~\ref{prop-remainder} to remove expectations. Both \(\mathscr{R}\) and \(\sum_{|\mathcal{T}|= m}c_k^{\mathcal{T}}(t)\) are represented by enhanced trees preserved under congruence. Next, assume \(\|\mathcal{R}\|_Z \leq N^{-500}\). Observe from Section \ref{eq-multilinear} that:
\[
  \|\mathscr{R}\|_Z \;+\;
  \|\mathscr{Q}(\mathcal{R},\mathcal{R})\|_Z \;+\;
  \|\mathscr{C}(\mathcal{R},\mathcal{R},\mathcal{R})\|_Z 
  \;\lesssim\; N^{-600}.
\]
Here, the estimates for \(\|\mathscr{Q}(\mathcal{R},\mathcal{R})\|_Z\) and \(\|\mathscr{C}(\mathcal{R},\mathcal{R},\mathcal{R})\|_Z\) follow from Cauchy-Schwarz that
\begin{align*}
    \|\mathcal{W}(u,v,w)\|_Z, \|\mathscr{Q}_D (u,v,w)\|_Z, \|\mathscr{C}_D (u,v,w)\|_Z
    \;\lesssim\; N^{3}\,\|u\|_{Z}\,\|v\|_{Z}\,\|w\|_{Z},
\end{align*}
together with the bound \(\|\mathcal{R}\|_{Z}\le N^{-500}\) for at least two of the factors \(u,v,w\), and the estimate \eqref{eq-Jest-restate} for the remaining factor. The bound \(\|\mathscr{R}\|_Z \lesssim N^{-600}\) follows from \eqref{eq-Rest-restate}, the scaling law $\beta=N^{-\gamma}$, $T=N^{-\epsilon}\min(N,N^{\frac{5}{4}\gamma})$, and choosing \(n = 10^4/\epsilon\) large enough.

We also note that \((1 - \mathscr{L})^{-1}\) maps \(Z\) to itself, since
\[
  (1 - \mathscr{L})^{-1} 
  \;=\; (1 - \mathscr L^{n_h})^{-1}\,\bigl(1 + \mathscr{L} + \cdots + \mathscr{L}^{n_h-1}\bigr),
\]
and \(\|\mathscr L^{n_h}\|_{Z \to Z} < 1\) for sufficiently large \(n_h\). Therefore we can invert \((1 - \mathscr L^{n_h})\) by a Neumann series, and it follows that
\[
  \|(1 - \mathscr{L})^{-1}\|_{Z \to Z} \;\lesssim\; N^{30}.
\]

Putting these facts together shows that the map given by the right-hand side of \eqref{eq-op} is a self-map on the set \(\{\mathcal{R}: \|\mathcal{R}\|_Z \leq N^{-500}\}\), with the contraction property guaranteed by the smallness of the terms above. Consequently, by the Banach Fixed Point Theorem, there is a unique solution \(\mathcal{R}\) in that ball, and this occurs with probability at least \(1 - N^{-A}\).
\end{proof}
\begin{proof}[Proof of Theorem \ref{main}]
As we can derive from the initial data $\left\|b_k\right\|_{\ell^2} \sim O(1)$, and in this setting, the solution is globally well-posed for the nonlinear smooth ODE system \eqref{FPU}.  Let $E$ denote the complement of the union of all exceptional sets arising from Propositions \ref{prop-remainder} and \ref{prop-b}, so that 
\[
\mathbb{P}(E) \;\ge\; 1 - N^{-A}.
\]
Recall that $n\left(t,k\right)=\E |b_k(t)|^2 = \E|c_k(s)|^2$ where $s = \frac{t}{T}$.
We decompose
\begin{equation}
    \E\left(|c_k(s)|^2\right)
    \;=\;
    \E\left(|c_k(s)|^2\boldsymbol{1}_{E}\right)
    \;+\;
    \E\left(|c_k(s)|^2\boldsymbol{1}_{E^c}\right).
\end{equation}
Since $\E\left(|c_k(s)|^2\boldsymbol{1}_{E^c}\right)\lesssim N^{-A + 20}$,
it suffices to consider $\E\left(|c_k(s)|^2\boldsymbol{1}_{E}\right)$. Then we can get:
\begin{align*}
\E\left(|c_k(s)|^2\boldsymbol{1}_{E}\right) = &\E \left( \left|c^{\leq n}_k(s)\right|^2 \boldsymbol{1}_{E}\right)+ 2 \mathrm{Re} \E \left( c^{\leq n}_k(s) \mathcal{R}^{(n+1)*}_{k}(s)\boldsymbol{1}_{E}\right) + \E\left(\left|\mathcal{R}^{(n+1)}_{k}(s)\right|^2 \boldsymbol{1}_{E}\right).
\end{align*}
By Proposition~\ref{prop-b} and \eqref{eq-Jest-restate}, any term involving $\mathcal{R}$ can be controlled at the order $N^{-500}$. Consequently, our main interest is in $\left|c^{\leq n}_k(s)\right|^2=\sum_{|\T_1|,|\T_2|\leq n}c^{\T_1}_k(s)c^{\T_2*}_k(s)$ the double sum over $|\T_1|$ and $|\T_2|$. Using Cauchy-Schwarz, we observe that
\begin{equation}
\left| \E \left(\left|c^{\leq n}_k(s)\right|^2\boldsymbol{1}_{E^c}\right) \right|\leq \left(\E \left|c^{\leq n}_k(s)\right|^4\right)^{1/2}\left( \mathbb P(E^c)^{1/2}\right) \lesssim N^{-A/2 + 10}
\end{equation}
where we have used \eqref{eq-Jest-restate} and Lemma~\ref{lem-hypercontractivity}. Thus, in the expectation $\E \left(\left|c^{\leq n}_k(s)\right|^2 \boldsymbol{1}_{E}\right)$, the indicator $\mathbf{1}_E$ can effectively be replaced by $1$. In that case, Proposition~\ref{bound-couple} shows that $\E \left|c^{\leq n}_k(s)\right|^2$
is $o_{\ell_k^\infty}\!\bigl(\tfrac{t}{T_{\mathrm{kin}}}\bigr)$ for $T_{\mathrm{kin}}=\frac{1}{4 \pi \beta^2}$ whenever $|\T_1|+|\T_2| \ge 3$. Finally, we handle the main contributing cases. If $|\T_1|=|\T_2|=0$, the only contribution to $\E |b_k(t)|^2$ is $n_{\mathrm{in}}(k)$. In the case $|\T_1|+|\T_2|=1$, we have
\begin{equation*}
2\mathrm{Re} \left(\sum_{|\T_0|=0,|\T_1|=1}\E \left( c^{\T_0}_k(s) c^{\T_1*}_k(s)\right)\right) = 0,
\end{equation*}
since no allowed couples under enhance pairing occur there. When $|\T_1|+|\T_2| = 2$, we need to calculate the terms as shown in Figure \ref{fig-iterates-compute} and its variants with degenerate nodes: \[\mathcal S_k := \sum_{|\T_1|=1}\E \left|c^{\T_1}_k(s)\right|^2 + 2\mathrm{Re} \left(\sum_{|\T_0|=0,|\T_2|=2}\E \left(  c^{\T_2}_k(s)c^{\T_0*}_k(s)\right)\right)\] as in the proof of \cite[Theorem 1.3]{2019}. In the following calculations, we omit the fully degenerate cases which can be bounded by $O\left( \frac{Ts}{T_{\mathrm{kin}}}N^{-\delta}\right)$ trivially.
\begin{figure}
\begin{subfigure}[t]{1\linewidth}
\centering
\begin{tikzpicture}[level distance=.8cm,
  level 1/.style={sibling distance=.8cm},
  level 2/.style={sibling distance=.8cm}]
\tikzstyle{hollow node}=[circle,draw,inner sep=1.6]
\tikzstyle{solid node}=[circle,draw,inner sep=1.6,fill=black]
\tikzset{
red node/.style = {circle,draw=black,fill=red,inner sep=1.6},
blue node/.style= {circle,draw = black, fill= blue,inner sep=1.6}, 
purple node/.style= {circle,draw = black, fill= purple,inner sep=1.6}, 
orange node/.style= {circle,draw = black, fill= orange,inner sep=1.6},
yellow node/.style= {circle,draw = black, fill= yellow,inner sep=1.6},
green node/.style = {circle,draw=black,fill=green,inner sep=1.6}}
\node[solid node, label = right:{$k$}] at (-5.5,.5){}
    child{node[blue node, label=below:{$k_1$}]{}}
    child{node[red node, label=below: $k_2$]{}}
    child{node[green node, label=below: $k_3$]{}}
;
\node[solid node, label = right: {$k$}] at (-3,.5){}
    child{node[blue node, label=below: $k_1$]{}}
    child{node[red node, label=below: $k_2$]{}}
    child{node[green node, label=below: $k_3$]{}}
;
\node[solid node, label = right: $k$ ] at (1,.5){}
    child{node[solid node, label = left: $k_1$ ]{}
        child{node[blue node, label=below: $k_2$]{}}
        child{node[red node, label=below: $k_3$]{}}
        child{node[green node, label=below: $k$]{}}
    }
    child{node[blue node, label=below: $k_2$]{}}
    child{node[red node, label=below: $k_3$]{}}
;
\node[green node, label=right: $k$] at (3.5,.5){};

\node at (-5.5, -2.7)[solid node, label = right: $k$]{}
    child{node[blue node, label=below: $k_1$]{}}
    child{node[solid node, label = right: $k_2$ ]{}
        child{node[blue node, label=below: $k_1$]{}}
        child{node[red node, label=below: $k$]{}}
        child{node[green node, label=below: $k_3$]{}}
    }
    child{node[green node, label=below: $k_3$]{}}
;
\node[red node, label=right: $k$] at (-3,-2.7){};

\node[green node, label=right: $k$] at (3.5,-2.7){};
\node[solid node, label = right: {$k$}] at (1, -2.7){}
    child{node[blue node, label=below: $k_1$]{}}
    child{node[red node, label=below: $k_2$]{}}
    child{node[solid node, label = right: {$k_3$}]{}
        child{node[red node, label=below: $k_2$]{}}
        child{node[blue node, label=below: $k_1$]{}}
        child{node[green node, label=below: $k$]{}}
    }
;
\end{tikzpicture}
\end{subfigure}
\caption{All types of couples $\Q$ without degenerate nodes of order $n = 2$.}
\label{fig-iterates-compute}
\end{figure}
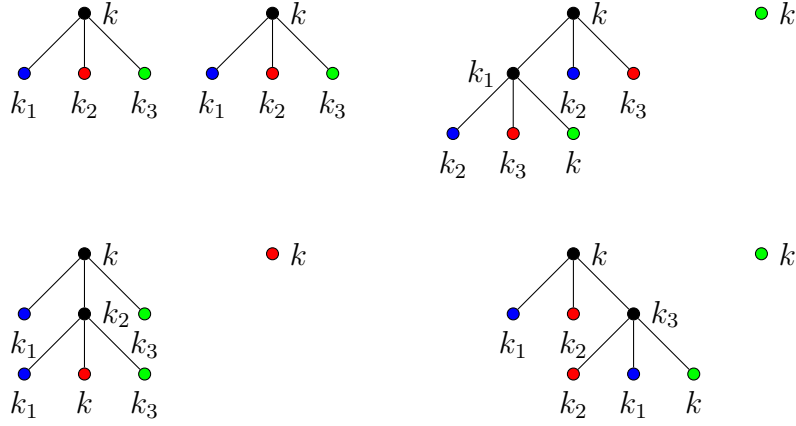

Since we have a nonlinear frequency shift in the phase, we need to calculate the following for $\Omega_k\neq 0$ using the fact that $A(0)=0$:
\begin{align}
&\int_0^se^{i\Omega_k(Ts_1+A(s_1))}ds_1=\int_0^s\frac{d}{ds_1}\left(\int_{0}^{s_1}e^{i\Omega_kTs_2}ds_2\right)e^{i\Omega_kA(s_1)}ds_1 \notag\\
=&\left[\left(\int_{0}^{s_1}e^{i\Omega_kTs_2}ds_2\right)e^{i\Omega_kA(s_1)}\right]_0^s-\int_0^s\frac{e^{i\Omega_kTs_1}}{i\Omega_kT}i\Omega_k\dot{A}(s_1)e^{i\Omega_kA(s_1)}ds_1 \notag\\
=&\frac{e^{i\Omega_k(Ts+A(s))}-1}{i\Omega_kT}-\int_0^s\frac{\dot{A}(s_1)}{T}e^{i\Omega_k(Ts_1+A(s_1))}ds_1. \label{eq-int1}
\end{align}
Note that the second term can be bounded by $O\left(\frac{N^{-\gamma/2}}{\langle\Omega_kT\rangle}\right)$ using $\dot A(s) \lesssim \beta T$, $\ddot A(s) \lesssim (\beta T)^3$ (using the tight bound from \eqref{eq-Add}) and Proposition \ref{integral_est}. Then using the fact that 
\begin{align*}
    \left|\frac{e^{i\Omega_kTs}-1}{i\Omega_kTs}\right|^2&= \frac{1}{|\Omega_kTs/2|^2}\cdot\frac{e^{i\Omega_kTs/2}-e^{-i\Omega_kTs/2}}{2i}\cdot\frac{e^{-i\Omega_kTs/2}-e^{i\Omega_kTs/2}}{-2i}\notag\\
&=\left|\frac{\sin(\Omega_kTs/2)}{\Omega_kTs/2}\right|^2=2\mathrm{Re}\left[\frac{1-e^{i\Omega_kTs}}{(\Omega_kTs)^2}\right],
\end{align*}
after summing all the couples with two scale 1 enhanced trees as shown in Figure \ref{fig-iterates-compute} \& \ref{fig-iterates-compute2} we can get:
\begin{align}
\sum_{|\T_1|=1}\E \left|c^{\T_1}_k(s)\right|^2 &= \left( \frac{\beta Ts}{N}\right)^2 \Bigg[2 \sum_{\Omega(\vec k) \neq 0}^{\times} |T_{k,1,2,3}|^2\phi_{k_1} \phi_{k_2} \phi_{k_3}\left|\frac{\sin(\Omega(\vec{k})(Ts+A(s))/2)}{\Omega(\vec{k})Ts/2}\right|^2 \notag\\
&+2\sum_{\Omega(\vec k) = 0}^{\times}|T_{k,1,2,3}|^2\phi_{k_1} \phi_{k_2} \phi_{k_3} +4\sum_{k_1 \in \Z_N\cap(0,1)}|T_{k,1,1,k}|^2{\phi_{k_1}^2}{\phi_{k}}+ O\left( \frac{Ts}{T_{\mathrm{kin}}}N^{-\delta}\right)\Bigg],\label{eq-first-iterate1}
\end{align}
where $\phi_{k_i} := n_{\mathrm{in}}\left(k_i\right)$, $\vec k =(k,k_1, k_2, k_3)$, and ${\times}=\{(k_1, k_2, k_3)\in (\Z_N\cap(0,1))^3:k_1-k_2+k_3=k;k_1,k_3 \neq k_2\}$. Note that the sum involved with the lower order integral in \eqref{eq-int1} can be bounded as:
\begin{align*}
    \left( \frac{\beta T}{N}\right)^2 \sum_{\Omega(\vec k) \neq 0}^{\times} |T_{k,1,2,3}|^2O\left(\frac{N^{-\gamma/2}}{\langle\Omega_kT\rangle^2}\right) \lesssim \left( \frac{\beta T}{N}\right)^2 N^{-\gamma/10} N^2T^{-1}\lesssim \frac{Ts}{T_{\mathrm{kin}}}N^{-\delta}.
\end{align*}
Similarly we need to calculate for $\Omega_k,\Omega_{k'}\neq 0$:
\begin{align*}
& \int_0^se^{i\Omega_k(Ts_1+A(s_1))}\int_0^{s_1} e^{i\Omega_{k'}(Ts_2+A(s_2))}ds_2ds_1 \notag\\ 
=& \int_0^se^{i\Omega_k(Ts_1+A(s_1))}\left[\frac{e^{i\Omega_{k'}(Ts_1+A(s_1))}-1}{i\Omega_{k'}T}-\int_0^{s_1}\frac{\dot{A}(s_2)}{T}e^{i\Omega_{k'}(Ts_2+A(s_2))}ds_2\right] ds_1 \notag\\
=& \frac{1}{i\Omega_{k'}T} \left(\int_0^se^{i(\Omega_k+\Omega_{k'})(Ts_1+A(s_1))}ds_1-\int_0^se^{i\Omega_k(Ts_1+A(s_1))}ds_1\right) \\
&\qquad -\int_0^se^{i\Omega_k(Ts_1+A(s_1))}\int_0^{s_1}\frac{\dot{A}(s_2)}{T}e^{i\Omega_{k'}(Ts_2+A(s_2))}ds_2ds_1\\
=& \frac{1}{i\Omega_{k'}T} \int_0^se^{i(\Omega_k+\Omega_{k'})(Ts_1+A(s_1))}ds_1-\frac{e^{i\Omega_k(Ts+A(s))}-1}{(i\Omega_{k}T)(i\Omega_{k'}T)}+O\left(\frac{N^{-\gamma/2}}{\langle\Omega_kT\rangle\langle\Omega_{k'}T\rangle}\right) \\
=& \frac{-s}{i\Omega_{k}T}+\frac{1-e^{i\Omega_k(Ts+A(s))}}{(\Omega_{k}T)^2}+O\left(\frac{N^{-\gamma/2}}{\langle\Omega_kT\rangle^2}\right),
\end{align*}
since we must have $\Omega_k+\Omega_{k'}=0$ in all cases from Figure \ref{fig-iterates-compute}. 
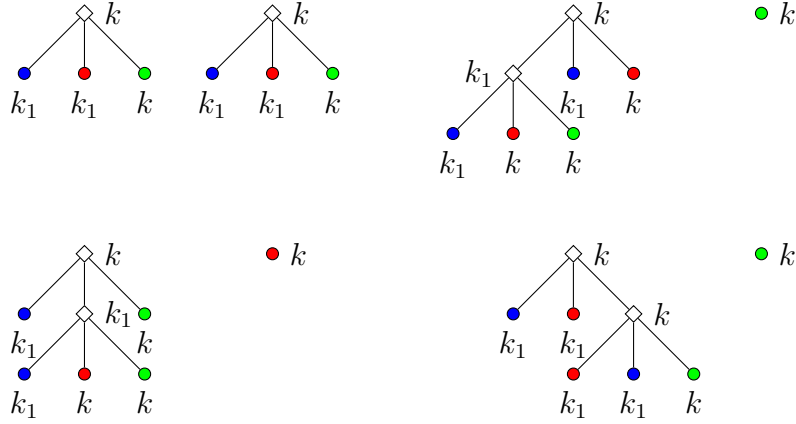
\begin{figure}
\begin{subfigure}[t]{1\linewidth}
\centering
\begin{tikzpicture}[level distance=.8cm,
  level 1/.style={sibling distance=.8cm},
  level 2/.style={sibling distance=.8cm}]
\tikzstyle{hollow node}=[circle,draw,inner sep=1.6]
\tikzstyle{solid node}=[circle,draw,inner sep=1.6,fill=black]
\tikzstyle{deckdiamond}=[shape=diamond, draw=black, fill=white, minimum size=6pt, inner sep=1.6pt]
\tikzset{
red node/.style = {circle,draw=black,fill=red,inner sep=1.6},
blue node/.style= {circle,draw = black, fill= blue,inner sep=1.6}, 
purple node/.style= {circle,draw = black, fill= purple,inner sep=1.6}, 
orange node/.style= {circle,draw = black, fill= orange,inner sep=1.6},
yellow node/.style= {circle,draw = black, fill= yellow,inner sep=1.6},
green node/.style = {circle,draw=black,fill=green,inner sep=1.6}}
\node[deckdiamond, label = right:{$k$}] at (-5.5,.5){}
    child{node[blue node, label=below:{$k_1$}]{}}
    child{node[red node, label=below: $k_1$]{}}
    child{node[green node, label=below: $k$]{}}
;
\node[deckdiamond, label = right: {$k$}] at (-3,.5){}
    child{node[blue node, label=below: $k_1$]{}}
    child{node[red node, label=below: $k_1$]{}}
    child{node[green node, label=below: $k$]{}}
;
\node[deckdiamond, label = right: $k$ ] at (1,.5){}
    child{node[deckdiamond, label = left: $k_1$ ]{}
        child{node[blue node, label=below: $k_1$]{}}
        child{node[red node, label=below: $k$]{}}
        child{node[green node, label=below: $k$]{}}
    }
    child{node[blue node, label=below: $k_1$]{}}
    child{node[red node, label=below: $k$]{}}
;
\node[green node, label=right: $k$] at (3.5,.5){};

\node at (-5.5, -2.7)[deckdiamond, label = right: $k$]{}
    child{node[blue node, label=below: $k_1$]{}}
    child{node[deckdiamond, label = right: $k_1$ ]{}
        child{node[blue node, label=below: $k_1$]{}}
        child{node[red node, label=below: $k$]{}}
        child{node[green node, label=below: $k$]{}}
    }
    child{node[green node, label=below: $k$]{}}
;
\node[red node, label=right: $k$] at (-3,-2.7){};

\node[green node, label=right: $k$] at (3.5,-2.7){};
\node[deckdiamond, label = right: {$k$}] at (1, -2.7){}
    child{node[blue node, label=below: $k_1$]{}}
    child{node[red node, label=below: $k_1$]{}}
    child{node[deckdiamond, label = right: {$k$}]{}
        child{node[red node, label=below: $k_1$]{}}
        child{node[blue node, label=below: $k_1$]{}}
        child{node[green node, label=below: $k$]{}}
    }
;
\end{tikzpicture}
\end{subfigure}
\caption{All types of enhanced couples $\Q$ with degenerate nodes of order $n = 2$. Note that there must be two degenerate nodes in the enhanced couple to achieve admissible pairings.}
\label{fig-iterates-compute2}
\end{figure}
Summing all the couples with one scale 2 enhanced tree and one root node, we can get:
\begin{align}
    &2\mathrm{Re} \left(\sum_{|\T_0|=0,|\T_2|=2}\E \left(  c^{\T_2}_k(s)c^{\T_0*}_k(s)\right)\right) \notag\\
    =& \left( \frac{\beta T}{N}\right)^2 \Bigg[2s^2\sum_{\Omega(\vec k) = 0}^{\times}|T_{k,1,2,3}|^2\left(-\phi_{k} \phi_{k_2} \phi_{k_3}+\phi_{k} \phi_{k_1} \phi_{k_3}-\phi_{k} \phi_{k_1} \phi_{k_2}\right)\notag\\
&+4 \sum_{\Omega(\vec k) \neq 0}^{\times} |T_{k,1,2,3}|^2 \mathrm{Re}\left(\frac{-s}{i\Omega_{k}T}+\frac{1-e^{i\Omega_k(Ts+A(s))}}{(\Omega_{k}T)^2} \right)\left(-\phi_{k} \phi_{k_2} \phi_{k_3}+\phi_{k} \phi_{k_1} \phi_{k_3}-\phi_{k} \phi_{k_1} \phi_{k_2}\right)\notag\\ 
&+4s^2\sum_{k_1 \in \Z_N\cap(0,1)}|T_{k,1,1,k}|^2\left(-{\phi_{k_1}}{\phi_{k}^2}+{\phi_{k_1}}{\phi_{k}^2}-{\phi_{k_1}^2}{\phi_{k}}\right)+ O\left( \frac{Ts}{T_{\mathrm{kin}}}N^{-\delta}\right)\Bigg],\label{eq-first-iterate2}
\end{align}
Then summing up \eqref{eq-first-iterate1} \& \eqref{eq-first-iterate2}, we can derive:
\begin{align*}\label{eq-K2}
\mathcal S_k &= \frac{\beta^2 T^2 s^2}{N^2} \Bigg\{2\sum_{ \Omega(\vec k) \neq 0}^{\times} |T_{k,1,2,3}|^2\phi_k \phi_{k_1} \phi_{k_2} \phi_{k_3} \left[\frac{1}{\phi_k} - \frac{1}{\phi_{k_1}} + \frac{1}{\phi_{k_2}} - \frac{1}{\phi_{k_3}} \right]  \notag\\
&\qquad \qquad \times\left|\frac{\sin(\Omega(\vec{k})(Ts+A(s))/2)}{\Omega(\vec{k})(Ts+A(s))/2}\right|^2 \cdot \left|\frac{Ts+A(s)}{Ts}\right|^2 \notag\\
&\qquad \qquad+ 2\sum_{ \Omega(\vec k) = 0}^{\times} |T_{k,1,2,3}|^2\phi_k \phi_{k_1} \phi_{k_2} \phi_{k_3} \left[\frac{1}{\phi_k} - \frac{1}{\phi_{k_1}} + \frac{1}{\phi_{k_2}} - \frac{1}{\phi_{k_3}} \right]\Bigg\}+O\left( \frac{Ts}{T_{\mathrm{kin}}}N^{-\delta}\right)\notag\\
&=\frac{2\beta^2 T^2 s^2}{N^2} \sum_{ \Omega(\vec k) \neq 0}^{\times} |T_{k,1,2,3}|^2\phi_k \phi_{k_1} \phi_{k_2} \phi_{k_3} \left[\frac{1}{\phi_k} - \frac{1}{\phi_{k_1}} + \frac{1}{\phi_{k_2}} - \frac{1}{\phi_{k_3}} \right] \notag\\
&\qquad \qquad \qquad \qquad \qquad \times \left|\frac{\sin(\Omega(\vec{k})(Ts+A(s))/2)}{\Omega(\vec{k})(Ts+A(s))/2}\right|^2 \cdot \left|\frac{Ts+A(s)}{Ts}\right|^2+O\left( \frac{Ts}{T_{\mathrm{kin}}}N^{-\delta}\right). 
\end{align*}
 For the exact resonant term, we use the bound for 3-vector countings when $\Omega = 0$ and $\{k_1, k_3\} \neq \{k, k_2\}$, taking $m =0$ and replacing $T$ by $TN^{\delta}$ in the proof of Proposition \ref{vector_counting}. In addition we set $\delta < \frac{\epsilon}{10}$ and every step of integration by parts would gain by a factor $N^{\frac{\epsilon}{2}}$ as $T\lesssim N^{1-\epsilon}$. Applying integration by parts sufficiently many times will lead to the result $\lesssim N^{2-\delta}T^{-1}$. Using Corollary \ref{cor-iterates}, we have that 
\begin{align*}
\mathcal S_k &= 2\beta^2 t \cdot \frac{t+A(t/T)}{t} \cdot (t+A(t/T))\int_{k_i \in \mathbb{T}} \phi_k\phi_{k_1} \phi_{k_2} \phi_{k_3} \left[\frac{1}{\phi_k} - \frac{1}{\phi_{k_1}} + \frac{1}{\phi_{k_2}} - \frac{1}{\phi_{k_3}} \right] \notag\\
&\qquad \qquad \qquad \qquad \qquad \qquad \times \left| \frac{\sin(\Omega(\vec{k})(t+A(t/T))/2)}{\Omega(\vec{k})(t+A(t/T))/2}\right|^2dk_1dk_2dk_3+O\left( \frac{t}{T_{\mathrm{kin}}}N^{-\delta}\right).  
\end{align*}
Using the fact that $A(t/T)/t \sim O(N^{-\gamma})$, and for smooth functions $f$,
\begin{align}
\left|t\int\left| \frac{\sin(t x/2)}{t x/2}\right|^2 f(x) d x - 2\pi f(0)\right| \lesssim C(f) t^{-\frac{1}{2}},
\end{align}
where $C(f)$ depends on the $L^\infty$ norms of $f$ and $f'$, we are able to get:
\begin{align}
     \mathcal S_k &=  \frac{t}{T_{\mathrm{kin}}} \K(n_\mathrm{in})(k)+o_{\ell_k^{\infty}}\left( \frac{t}{T_{\mathrm{kin}}}\right),
\end{align}
which concludes the proof.
\end{proof}

\section*{References}
\printbibliography[heading = none]

\end{document}